\documentclass{scrartcl}
\usepackage[utf8]{inputenc}
\usepackage[T1]{fontenc}
\usepackage{amsfonts}
\usepackage{amsthm}

\usepackage{tipx}
\usepackage{enumitem}
\usepackage{stmaryrd}
\usepackage{MnSymbol}
\usepackage{xspace}
\usepackage{tikz-cd}
\usepackage{ tipa }
\usepackage{tipx}
\usepackage{wasysym}

\usepackage[french=guillemets]{csquotes}
\MakeOuterQuote{"}
\setlist{nosep}
\usepackage{xcolor}
\usepackage[style=alphabetic,backend=biber,sorting=nyt,giveninits=true]{biblatex}
\usepackage[colorlinks=true, breaklinks=true, urlcolor= black, linkcolor= black, citecolor= black, 
bookmarksopen=true,linktocpage=true,plainpages=false,pdfpagelabels]{hyperref}
\usepackage[capitalise]{cleveref}
\usepackage{microtype}

\renewbibmacro{in:}{}
\bibliography{biblio}
\newtheorem{theorem}{Theorem}[section]
\newtheorem*{theorem*}{Theorem}

\newtheorem{lemma}[theorem]{Lemma}
\newtheorem{corollary}[theorem]{Corollary}

\newtheorem{proposition}[theorem]{Proposition}
\newtheorem{remark}[theorem]{Remark}
\newtheorem{definition}[theorem]{Definition}
\newtheorem{IS}[theorem]{Incorrect Statement}
\newtheorem{notation}[theorem]{Notation}
\newtheorem{fact}[theorem]{Fact}
\newcommand\NIP{\ensuremath{\mathrm{NIP}}}
\newcommand\BS{\ensuremath{\mathrm{BS}}}
\newcommand\ABS{\ensuremath{\mathrm{ABS}}}
\newcommand\ACVF{\ensuremath{\mathrm{ACVF}}}

\newcommand\La{\mathrm{L}}
\newcommand\red{\mathrm{red}}
\newcommand\T{\mathrm{T}}

\newcommand\D{\mathrm{D}}
\newcommand\C{\mathrm{C}}
\newcommand\X{\mathrm{X}}
\newcommand\R{\mathrm{R}}
\newcommand\f{\mathrm{f}}
\newcommand\g{\mathrm{g}}
\newcommand\q{\mathrm{q}}
\newcommand\A{\mathrm{A}}
\newcommand\M{\mathrm{M}}
\newcommand\N{\mathrm{N}}
\newcommand\I{\mathrm{I}}
\newcommand\B{\mathrm{B}}
\newcommand\Z{\mathrm{Z}}
\newcommand\K{\mathrm{K}}
\newcommand\J{\mathrm{J}}
\newcommand\Q{\mathrm{Q}}
\newcommand\F{\mathrm{F}}
\newcommand\G{\mathrm{G}}

\newcommand\E{\mathrm{E}}
\newcommand\EMtype{\mathrm{EM type}}
\newcommand \EM{\mathrm{EM}}
\newcommand \Av{\mathrm{Av}}
\newcommand\tensor{\otimes}
\newcommand\Lat{\mathrm{S}}

\newcommand{\acl}{\mathrm{acl}}
\newcommand{\dcl}{\mathrm{dcl}}
\newcommand{\tp}{\mathrm{tp}}
\newcommand{\Cb}{\mathrm{Cseq}}

\newcommand{\cof}{\mathrm{cof}}
\newcommand{\St}{\mathrm{St}}
\newcommand{\p}{\mathrm{p}}

\newcommand{\Int}{\mathrm{Int}}
\newcommand{\Aut}{\mathrm{Aut}}
\newcommand{\h}{\mathrm{h}}
\newcommand{\rest}{\upharpoonright}
\newcommand{\ACF}{\mathrm{ACF}}
\newcommand{\V}{\mathrm{V}}
\newcommand{\Y}{\mathrm{Y}}
\newcommand{\EP}{\mathrm{EP}}
\newcommand{\GS}{\mathrm{GS}}
\newcommand{\Code}{\mathrm{Code}}
\newcommand{\dist}{\mathrm{dist}}
\newcommand{\dom}{\mathrm{dom}}
\newcommand{\cX}{\mathcal{X}}
\newcommand{\cY}{\mathcal{Y}}
\newcommand{\Tree}{\mathcal{T}}
\title{On Descent and germs}
\author{Pierre Simon and Mariana Vicar\'ia}
\begin{document}

\maketitle

\begin{abstract}
We present a new proof of descent for stably dominated types in any theory, dropping the hypothesis of the existence of global invariant extensions. Additionally, we give a much simpler proof of descent for stably dominated types in 
$\ACVF$. Furthermore, we demonstrate that any stable set in an $\NIP$ theory has the bounded stabilizing property. This result is subsequently used to correct Proposition 6.7 from the book on stable domination and independence in 
$\ACVF$.
\end{abstract}
\section{Introduction}
In \cite{HHM} D. Haskell, E. Hrushovski, and D. Macpherson introduced the theory of stable domination, which describes how a structure is governed by its stable part. A classical example of this is the theory of algebraically closed valued fields with a non trivial valuation ($\ACVF$), where the stable part comes from the residue field, which is an algebraically closed field (whose theory is $\ACF$ and is stable).

Stable domination later served as a bridge to lift machinery from the stable context to $\ACVF$, which is $\NIP$ (without the independence property). A key example is the description of definable abelian groups in  $\ACVF$ by E. Hrushovski and S. Rideau-Kikuchi in \cite[Theorem $5.16$]{metastable}\footnote{Recently, Paul Wang pointed out a gap in the published article. This has been successfully solved by the two original authors in collaboration with Wang in \cite{correction}}.  Another significant application is the development of a model-theoretic analogue of Berkovich analytification for varieties, along with the characterization of their homotopy types due to E. Hrushovski and F. Loeser \cite[Theorem $11.1.1$]{tame}.\\

We remind the reader of key definitions from prior work on stable domination. Given $\T$ a complete first order theory, $\M$ its monster model and $\A$ a small set of parameters the stable part of the structure, denoted as $\St_{\A}$, is the multi-sorted structure of all the $\A$-definable stable and stably embedded sets (including imaginary sorts) with the $\A$-induced structure (see \cref{def: stable part of the structure}).

Given a global type $\p$ we say that it is \emph{stably dominated over $\A$} if there is an $\A$-definable function $\f$ into a stable $\A$-definable set such that $\p$ is dominated by its pushforward along $\f$. Note that $\f$ might be a function with range in a pro-$\A$-definable set, in other words  $\f$ might be an infinite tuple of functions to ordinary sorts (see Definition \ref{def: stably dominated types}).\\

The notion of a stably dominated type is not obviously invariant under a change of base. In \cite[Chapter 4]{HHM}, D. Haskell, E. Hrushovski, and D. Macpherson address this issue. More precisely, they establish the following result:
\begin{theorem}(\cite[Proposition $4.1$,Theorem $4.9$]{HHM}) Let $\A=\acl(\A) \subseteq \B$ and $\p$ a global $\A$-invariant type. Assume that:
\begin{center}
\textbf{$(\EP)$} the type $\tp(\B/ \A)$ has a global $Aut(\M/\A)$-invariant extension. 
\end{center}
Then:
\begin{enumerate}
\item \emph{Going up:} if $\p$ is stably dominated over $\A$ then it is stably dominated over $\B$.
\item \emph{Descent:} if $\p$ is stably dominated over $\B$, then it is stably dominated over $\A$. 
\end{enumerate}
\end{theorem}

The proof of descent for stably dominated types seemed more complex than necessary. Indeed, in \cite{HHM}, the authors observe: “To show that stable domination is preserved when decreasing the parameter set is rather more difficult… The hypothesis in Theorem 4.9 (Descent) that 
$\tp(\B/\A)$ has a global invariant extension is stronger than might be needed; it would be beneficial to investigate the weakest assumptions under which some version of descent could be proved.” Later, in their work on metastable groups, E. Hrushovski and S. Rideau-Kikuchi pose the question: “Can descent be proved without the additional hypothesis that 
$\tp(\B/\A)$ has a global $Aut(\M/\A)$-invariant extension?” (see \cite[Question 1.3 (1)]{metastable}). They also note that the need for a global invariant extension in the definition of metastable theories arises precisely for descent to hold.\\
In this paper, we address these issues by clarifying the situation. A global type that is stably dominated is generically stable, and generically stable types are definable. 
One of our contributions is to present a significantly simplified version of the descent theorem in the case where the stable set is defined over the base of definition of the stably dominated type. 
\begin{theorem*}[\cref{thm: descent for internal sets}]Let $\p$ be generically stable over $\A$ and assume that it is dominated over $\A b$ by an $\A b$-definable function $\f_{b}$ into some $\A$-definable stably embedded set $\Lat$.  Assume furthermore that either:
\begin{enumerate}
\item $\tp(b/ \A)$ does not fork over $\A$, or
\item $\p^{\otimes n}$ is generically stable for all $n < \omega$ (which holds if $\Lat$ is stable). 
\end{enumerate}
 Then there is an $\A$-definable function $\h \colon \p \rightarrow \Int(\Lat,\A)$ that dominates $\p$ over $\A$, where $\Int(\Lat, \A)$ denotes the union of the $\A$-definable sets internal to $\Lat$. 
\end{theorem*}
This theorem applies in particular to
$\ACVF$, thus giving a much simpler proof in this case (see \cref{rem: easy ACVF}). 

We then address the general case of descent, with a substantially more technical proof.
\begin{theorem*}[\cref{thm:descent}]
    Let $\p$ be a global $\A$-invariant type and let $b$ be such that $\p$ is stably dominated over $\A b$. Then $\p$ is stably dominated over $\A$.
\end{theorem*}
In the first part of the paper we address and correct the incorrect statement found in \cite[Proposition 6.7]{HHM}. Specifically:\\
\textbf{Incorrect statement:} Let $\p$ be a global $\A$-definable type and assume that $\St_{\A}$ has $\BS$. Let $\f$  be a definable function on $\p(\M)$, the set of realizations of $\p \upharpoonright_{\A}$ and suppose that $\f(a) \in \St_{\A a}$ for all $a \in \p(\M)$. Then: 
\begin{itemize}
\item The germ $[\f]_{p}$ is strong over $\A$; 
\item $[\f]_{p} \in \St_{\A}$.
\end{itemize}
We demonstrate that the first part of this statement is false. In doing so, we show that any stable set in an 
$\NIP$ theory possesses the bounded stabilizing property (see \cref{def: BS}). Additionally, we provide a correct proof of the second part of the statement under the assumption of $\NIP$ in \cref{prop: correction} and discuss natural generalizations in \cref{rem: generalizations}.\\

 The paper is organized as follows:
 \begin{itemize}
 \item Section $2$: Preliminaries.
 \item Section $3$: We introduce the definitions of the bounded stabilizing property and the algebraic stabilizing property. We demonstrate that any stable set in an $\NIP$ theory has the bounded stabilizing property. Additionally, we present a counterexample to the first part of the incorrect statement (see \cref{notstrong}) and offer a corrected proof of the second part (see \cref{prop: correction}).
 \item Section $4$: Descent for stably dominated types. We begin with simplified versions of the descent theorem in \cref{sec: easy descent}, including a concise proof for $\ACVF$ provided by \cref{thm: descent for internal sets}. We then prove the general theorem. For sake of clarity, we first present a simplified proof under the assumption of global invariant extensions and then extend it to the general case. 
 \end{itemize}

\subsection*{Acknowledgments}
The authors wish to thank E. Hrushovski, S. Rideau-Kikuchi, I. Kaplan and N. Ramsey for
discussions on this topic. We thank K. Gannon and G. Conant for pointing out to the second author that \cite[Example $1.7$]{pillaygeneric} is incorrect. Lastly, we want to thank the anonymous referee for their careful reading and comments, that helped us to improve the presentation on the paper. The second authour was supported by the Humboldt Fellowship. 

\newpage
\section{Preliminaries}
Let $\T$ be a complete first order theory and $\M$  its monster model.

A number of results in this paper hold without any restriction on the theory $\T$, and when some tameness condition of the theory is required this will be explicitly stated.

Through the paper, we assume that $\T$ is a complete first order $\La$-theory that eliminates imaginaries. Such assumption is not strictly necessary as one can always work in $\T^{eq}$,  but it will help us to simplify notation, for example to write simply $\acl$ instead of $\acl^{eq}$, or $\dcl$ instead of $\dcl^{eq}$. 

For $\A$ a set of parameters, we let $\mathcal{S}_x(A)$ be the set of types over $\A$ in variable $x$. We write $\mathcal{S}^m(A)$ for $\mathcal{S}_{x_1\ldots x_m}(A)$. We drop $x$ if it is clear from the context.
\subsection{Properties of non-forking  and non-dividing in arbitrary theories}
\begin{definition} Let $a,b \in \M$ be (possibly infinite) tuples and $\A \subseteq \M$ a small set of parameters. We write 
\begin{itemize}
\item $a \downfree_{\A}^{d} b$ when $\tp(a/ \A b)$ does not divide over $\A$; 
\item $a \downfree_{\A}^{f} b$ when $\tp(a/ \A b)$ does not fork over $\A$. 
\end{itemize}
\end{definition}

\begin{proposition}\label{prop: properties of non dividing}
\begin{enumerate}
\item \emph{(Base Monotonicity for dividing)} If $a \downfree_{\A}^{d} b d$, then $a \downfree_{\A b}^{d} d$. 
\item \emph{(Left transitivity of non-dividing)} If $a \downfree_{\A b}^{d}d$ and $b \downfree_{\A}^{d} d$, then $a b \downfree_{\A}^{d} d$. 
\end{enumerate}
\end{proposition}
\begin{proof}
The first statement is a folklore result while the second one is \cite[Lemma 1.5]{transdiv}. 
\end{proof}
\begin{definition}\label{def:non-dividing}  Let $\A$ be a set of parameters. We say that the sequence of tuples $\bar b= (b_i)_{i\in \I}$ is \emph{non-dividing over $\A$} if $b_{i} \downfree^{d}_{\A} b_{<i}$ for all $i \in \I$. 
\end{definition}

\begin{proposition}\label{prop: properties of non-forking}
 Let $\A$ be a small set of parameters and $a,b,a',b'$ tuples. The following statements hold: 
\begin{enumerate}
\item \emph{(Invariance under automorphism)} if $a \downfree_{\A}^{f} b$ then $\sigma(a) \downfree_{\sigma(\A)}^{f} \sigma(b)$ for $\sigma \in \Aut(\M)$; 
\item \emph{(Finite Character)} if for all finite tuples $a' \subseteq a$, $b' \subseteq b$ $a' \downfree_{\A}^{f} b'$ then $a \downfree_{\A}^{f} b$;
\item \emph{(Monotonicity)} $a a' \downfree_{\A}^{f} b b'$ implies $a \downfree_{\A}^{f} b$;
\item \emph{(Base Monotonicity)} if $a \downfree_{\A}^{f} b d$ then $a \downfree_{\A b}^{f} d$; 
\item \emph{(Left transitivity)} if $a \downfree_{\A}^{f} b$ and $a' \downfree_{\A a}^{f} b$ then $a a' \downfree_{\A}^{f} b$; 
\item \emph{(Right Extension)} if $a \downfree_{\A}^{f} b$ for any $d$ there is $d' \equiv_{\A b} d$ such that $a \downfree_{\A}^{f} b d'$;
\item  if $a \downfree_{\A}^{f} b$, then for any $d$ there is $a' \equiv_{\A b} a$ such that $a' \downfree_{\A}^{f} b d$. 
\end{enumerate}
\end{proposition}
\subsubsection{Generically stable types}
\begin{definition} A global type $\p(x)$ is generically stable over $\A$ if it is $\A$-invariant and for every ordinal $\alpha \geq \omega$, any Morley sequence $(a_{i})_{i < \alpha}$ of $\p$ over $\A$ (i.e $a_{i} \vDash \p \rest_{\A a_{<i}}$) and any $\La(\M)$ formula $\phi(x)$, the set $\{ i < \alpha \ | \ \M \vDash \phi(a_{i})\}$ is finite or co-finite. 
\end{definition}
 We will use some results for generically stable types that hold in arbitrary theories. 
\begin{fact}\label{fact: first things about generically stable} 
\begin{enumerate}
\item Let $\p$ be a generically stable type over $\A$, and $\bar a= (a_{i}: i< \alpha)$ be an infinite Morley sequence of $\p$ over $\A$, then $\displaystyle{\p=\Av(\bar a/ \M)}$, where 
\[\Av(\bar a/ \M)=\{ \phi(x, b) \ | \ \{ i < \alpha \ | \ \M \models \phi(a_{i},b)\} \ \text{is infinite}\}.\]
\item Let $\p$ be a generically stable type over $\A$ and $\bar{a}= (a_{i} \ | \ i \in \I)$, where $|\I| \geq |\T(\A)|^{+}$, a Morley sequence of $\p$ over $\A$, and $b \in \M$, then there is a subset $\J \subseteq \I$ such that $|\I \backslash \J| \leq |\T(\A)|$ and $a_{j}\models \p\rest_{\A b}$ for $j \in \J$. 
\item If $\p$ is generically stable over a set $\B$ and $\A$-invariant, then it is generically stable over $\A$.
\end{enumerate}
\end{fact}
\begin{proof}
The first and third statement are \cite[Remark $9.3$(3,1)]{Casanovas}. The second statement is a direct consequence of the first one. 
\end{proof}

\begin{proposition}\label{prop: generically stable properties} Let $\p$ be a generically stable type over $\A$ and  $a \models \p \rest_{\A}$. 
\begin{enumerate}
\item Any Morley sequence of $\p$ over $\A$ is totally indiscernible over $\A$; 
\item $\p \rest_{\A}$ is stationary and $\p$ is its only global non-forking extension;
\item  $\p$ is definable over $\A$;
\item (Symmetry) if  $b \downfree_{\A}^{f} \A$, then $a \downfree _{\A}^{f} b$ if and only if $b \downfree_{\A}^{f} a$; 
\item (Right transitivity)  $a \downfree_{\A}^{f} b$ and $a \downfree_{\A b}^{f} d$  if and only if $a \downfree_{\A}^{f} bd$. 
\end{enumerate}
\end{proposition}
\begin{proof}
These are \cite[Proposition 9.6 (4), 9.7,9.6 (3),9.8]{Casanovas}. 
\end{proof}
\begin{lemma}\label{lem: Existence of $p$-basis} Let $b \in \M$ be any tuple and $\p$ a global $\A$-invariant generically stable type. There is a Morley sequence $\bar a$ of $\p$ over $\A$ such for any $a\models\p \rest_{\A}$, we have the implication
    \[a\models \p\rest_{\A \bar a} \Longrightarrow a\models \p\rest_{\A b}.\]
    If this holds, we say that $\bar a$ is a \emph{$\p$-basis for $b$} over $\A$. 
\end{lemma}
\begin{proof}
We construct such a sequence by a greedy algorithm: assume the sequence $\bar a$ has been constructed. If there is $a\models \p \rest_{\A \bar a}$ such that $a$ does not realize $\p$ over $\A b$, then we add $a$ to the sequence $\bar a$. By Fact \ref{fact: first things about generically stable}(2) this process must stop after less than $|\T(\A)|^{+}$ steps.
\end{proof}

\begin{definition}
    Let $\p$ be an $\A$-invariant type such that $\p^{\otimes n}$ is generically stable for each $n$. We say that the sequence $\bar a$ is a \emph{$\p$-basis for $b$ over $\A$} if it is a Morley sequence of $\p$ and for each $n$ and $\bar c \models \p^{\otimes n}$, we have the implication
    \[\bar c\models \p\rest_{\A \bar a} \Longrightarrow \bar c\models \p\rest_{\A b}.\]
\end{definition}

Note that a weak $\p$-basis for $\p^{\omega}$ is a $\p$-basis for $\p$. In particular, there always exists a $\p$-basis for $b$ over $\A$.

\begin{lemma}\label{lem: sequence for gen n stable} Let $b \in \M$ and $\p$ an $\A$-invariant type such that $\p^{\otimes n}$ is generically stable for each $n$. Then there is a sequence $(\bar{a}_{i}, b_{i} \ | \ i < |\T(\A)|^{+})$ where:
\begin{enumerate}
\item $\bar{a}_{i}$ is a Morley sequence in $\p^{\omega}$ over $\A$,
\item $\bar{a}_{i}$ is a $\p$-basis of $b_{i}$ over $\A$ for all $i < |\T(\A)|^{+}$, and 
\item $\bar{a}_{i} b_{i} \equiv_{\A} \bar{a}_{0}b$. 
\end{enumerate}
\end{lemma}
\begin{proof}
By Lemma \ref{lem: Existence of $p$-basis} one can find $\bar a$ a $\p$-basis of $b$ of length $\alpha$. We construct inductively the sequence $\displaystyle{(\bar a_{i}, b_{i} \ | \ i < |\T(\A)|^{+})}$. Set $\bar{a}_{0}= \bar{a}$ and $b_{0}=b$. Let $\eta<| \T(\A)|^{+}$ and assume $(\bar{a}_{i},b_{i} \ | \ i < \eta)$ has been constructed. Let $\bar{a}_{\eta} \models \p^{\alpha} \rest_{\A (\bar{a_{i}})_{i <\eta}}$. Let $\sigma \in \Aut(\M/ \A)$ sending the tuple $\bar{a}_{0}$ to $\bar{a}_{\eta}$, and let $b_{\eta}= \sigma(b_{0})$. By construction, $\bar{a}_{\eta} b_{\eta} \equiv_{\A} \bar{a}_{0}b_{0}$, and because $\bar{a}_{0}$ is a $\p$-basis of $b_{0}$ then $\bar{a}_{\eta}$ is also a $\p$-basis of $b_{\eta}$ (The notion of being a $\p$ basis is preserved under automorphisms fixing the base parameter set $\A$). 
\end{proof}
\begin{definition} Let $\p$ be a global type which is definable over a small set $\A \subseteq \M$. Let $\f_{c}$ be a definable function over a parameter $c$. We define an equivalence relation on the set of realizations of $\tp(c/\A)$ in $\M$ by $\E(c_{1},c_{2})$ if and only if $\f_{c_{1}}(x)=\f_{c_{2}}(x) \in \p(x)$. Since $\p$ is $\A$-definable, $\E$ is $\A$-definable. 
\begin{enumerate}
\item The class $c/\E$ is called the \emph{germ} of $\f_{c}$ on $\p$, denoted as $[\f_{c}]_{\p}$. 
\item We say that the germ  $[\f_{c}]_{\p}$ \emph{is strong over $\A$} if for any realization $a \vDash \p\rest_{\A}$, we have  $\f_{c}(a) \in \dcl([\f_{c}]_{\p},a, \A)$.  Equivalently, there is an $\A[\f_{c}]_{\p}$-definable function $\g$ such that for any realization $a \vDash \p \upharpoonright_{\A c}$, $\f_{c}(a)=\g(a)$. 
\end{enumerate}
\end{definition}
The following is \cite[Theorem 2.2]{stronggerms}.
\begin{theorem}\label{stronggerms} Let $\p$ be a global type generically stable over a small set $\A \subseteq \M$ and $\f_{c}$ a definable function defined on $\p$. Then the germ  $[\f_{c}]_{\p}$ is strong over $\A$. 
\end{theorem}
\subsection{Other notions of independence}
In this section we present results about various notions of independence that we will use. 
\subsubsection{$\acl$ and $\dcl$-independence}
\begin{definition}\label{def: acldclind} 
\begin{enumerate}
\item Let $a,b$ be tuples and $\A$ a set of parameters. We write $a \downfree_{\A}^{\acl} b$ to indicate $\acl(\A a) \cap \acl(\A b) \subseteq \acl(\A)$. Likewise, we denote $a \downfree_{\A}^{\dcl} b$ if $\dcl(\A a) \cap \dcl(\A b) \subseteq \acl(\A)$.
\item  Let $\bar b= (b_i)_{i\in \I}$ be a sequence of finite tuples and $\A$ a set of parameters. We say that $\bar{b}$ is $\acl$-independent over $\A$ if for any $\J, \J' \subseteq \I$ such that $\J \cap \J'=\emptyset$ we have $\acl(\A b_{\J}) \cap \acl(\A b_{\J'})=\acl(\A)$. 
\end{enumerate}
\end{definition}
We start by establishing a basic fact about non-dividing sequences (see \cref{def:non-dividing}.) 
\begin{lemma}\label{lem: acl non dividing}
   Let $\bar b= (b_i)_{i\in \I}$ be a sequence of finite tuples and $\A$ a set of parameters. Assume $\bar{b}$ is  indiscernible and non-dividing over $\A$, then it is $\acl$-independent over $\A$. 
\end{lemma}
\begin{proof}
We will show later (\emph{c.f.} \cref{lem: equivalence indipendence dcl and acl}) that it is enough to show the result for $\J < \J'$. By base monotonicity and left transitivity of non-dividing (\cref{prop: properties of non dividing}), we have $b_{\J'} \downfree^{d}_\A b_{\J}$. This implies that $b_{\J'}$ and $b_\J$ are $\acl$-independent over $\A$.
%To keep the presentation simple we rather present how to solve a particular case that can be generalized to the arbitrary case. 
%Let $\I_{0}=\{b_{1},b_{3}\}$ while $\I_{1}=\{b_{2},b_{4}\}$ we aim to show $\acl(\A b_{1} b_{3}) \cap \acl(\A b_{2} b_{4}) \subseteq \acl(\A)$. Since $b_{4} \downfree_{\A b_{1} b_{2}}^{d} b_{3}$ then $\acl(\A b_{1} b_{2} b_{4}) \cap \acl(\A b_{1} b_{2} b_{3}) \subseteq \acl(\A b_{1} b_{2})$, and consequently $\acl(\A b_{1} b_{3}) \cap \acl(\A b_{2} b_{4}) \subseteq \acl(\A b_{1} b_{3}) \cap \acl(\A b_{1} b_{2})$. Because $b_{3} \downfree_{\A b_{1}}^{d} b_{2}$ then $\acl(\A b_{1} b_{3}) \cap \acl(\A b_{1} b_{2}) \subseteq \acl(\A b_{1})$. Thus $\acl(\A b_{1} b_{3}) \cap \acl(\A b_{1} b_{2}) \subseteq \acl(\A b_{1})$, and hence $\acl(\A b_{1} b_{3}) \cap \acl(\A b_{2} b_{4}) \subseteq \acl(\A b_{1}) \cap \acl(\A b_{2} b_{4})$. Because $b_{4} \downfree_{\A b_{2}}^{d} b_{1}$ and $b_{2} \downfree_{\A}^{d} b_{1}$ by Proposition \ref{prop: properties of non dividing}(2) $b_{4} b_{2} \downfree_{\A}^{d} b_{1}$ so $\acl(\A b_{1}) \cap \acl(\A b_{2} b_{4}) \subseteq \acl(\A)$.
\end{proof}
\subsubsection{$\GS$-independence}

We will make some use of the notion of generically stable partial types, which was introduced in \cite{typedecom} and further studied in \cite{treeless}. It is not essential for us, but allows us to unify some arguments. Recall that $\T$ is a theory and $\M$ is its monster model.
\begin{definition} A partial type $\pi(x)$ over $\M$ (the monster model)  is a consistent set of formulas
with parameters in $\M$ that is closed under logical consequences and finite  conjunctions, that is:
\begin{itemize}
\item if $\phi(x), \psi(x) \in \pi(x)$ then $\phi(x) \wedge \psi(x) \in \pi(x)$,
\item if $\phi(x) \in \pi(x)$ and $\M \models \phi(x) \rightarrow \psi(x)$ then $\psi(x) \in \pi(x)$. 
\end{itemize}
Given a small set of parameters $\A$, we write $\pi(x) \rest_{\A}$ to denote the partial type obtained by taking the subset of $\pi(x)$ of formulas with parameters in $\A$. 
\end{definition}
\begin{definition} Let $\A$ be a small set of parameters.
\begin{enumerate}
\item We say that a partial type $\pi(x)$ is \emph{ind-definable over $\A$} if for every $\phi(x,y)$, the set
$\displaystyle{\{b \ | \ \phi(x,b) \in \pi(x)\}}$ is ind-definable over $\A$ (i.e., is a union of $\A$-definable sets).
\item We say that a partial type $\pi(x)$ is \emph{generically stable over $\A$} if it is \emph{ind}-definable over $\A$ and the following holds:\\
$(\mathrm{GS})$ if $(a_{k}\ | \ k < \omega)$ is a sequence such that $a_{k} \models \pi\rest_{\A a_{<k}}$ and $\phi(x,b)\in \pi(x)$, then for all $k$ but finitely many, we have $\M\models \phi(a_{k},b)$.
\end{enumerate}
\end{definition}

\begin{lemma}\label{lem: existential gen stable} Let $\alpha(y)$ be a partial type, generically stable over $\A$. Fix some $a,b \in \mathrm{M}$ such that $b \models \alpha(y) \rest_{\A}$ and let $\rho(x,y) \subseteq \tp(ab/ \A)$. Then the partial type $\displaystyle{\pi(x):= \exists y \big(\alpha(y) \wedge \rho(x,y)\big)}$ is generically stable over $\A$. 
 \end{lemma}
 \begin{proof}
 This is \cite[Corollary $1.12$]{treeless}.
\end{proof}

 \begin{lemma}\label{lem:existencemax} Let $\p(x) \in \mathcal{S}(A)$. There is a unique maximal global partial type $\pi_{\p}$ generically stable over $\A$ consistent with $\p$. That is, if $\pi$ is a global generically stable
partial type consistent $\p$, then $\pi \subseteq \pi_{\p}$. It follows, in particular, that $\pi_{\p}$ extends $\p$.
 \end{lemma}
 \begin{proof}
 This is \cite[Corollary $1.9$]{treeless}.
 \end{proof}

\begin{definition} Let $a,b$ be (possibly infinite) tuples in $\M$ and $\A$ a small set of parameters. We write $a \downfree_{\A}^{\mathrm{GS}} b$ if for every partial type $\pi(x)$ generically stable over $\A$, if $b \models \pi\rest_{\A}$ then $b \models \pi\rest_{\A a}$. Note that this is equivalent to stating that $b \models \pi_{*}\rest_{\A a}$ where $\pi_{*}$ is the maximal $\A$-invariant generically stable partial type extending $\tp(b/\A)$ given by \cref{lem:existencemax}.  
\end{definition}
\begin{theorem}\label{thm: properties GS relation} The relation $\downfree^{\GS}$ satisfies:
\begin{enumerate}
\item \emph{(Invariance)} if $a \downfree_{\A}^{\GS} b$ and $\sigma \in \Aut(M)$ then $\sigma(a) \downfree_{\sigma(\A)}^{\GS} \sigma(b)$, 
\item \emph{(Normality)} if $a \downfree_{\A}^{\GS} b$ then $a \A  \downfree_{\A}^{\GS} b \A$, 
\item \emph{(Monotonicity)} if $a \downfree_{\A}^{\GS}b$, $a' \subseteq a$ and $b' \subseteq b$ then $a' \downfree_{\A}^{\GS} b'$, 
\item \emph{(Left and right existence)} for all $\A, \B$ $\A \downfree_{\B}^{\GS} \B$ and $\A \downfree_{\A}^{\GS} \B$, 
\item \emph{(Right and left extension)} if $a \downfree_{\A}^{\GS} b$ and $b \subseteq b'$, then there is $a' \equiv_{\A b} a$ such that $a' \downfree_{\A}^{\GS} b'$. And if $a \subseteq a'$, there is $b' \equiv_{\A a} b$ such that $a' \downfree_{\A}^{\GS} b'$. 
\item \emph{(Finite character)} $\C \downfree_{\A}^{\GS} \B$ if and only if for every finite  $c \subseteq \C$ and $b \subseteq \B$, $c \downfree_{\A}^{\GS} b$. 
\item \emph{(Left transitivity)} if $a \downfree_{\A c}^{\GS} b$ and $c \downfree_{\A}^{\GS} b$, then $ac \downfree_{\A}^{\GS} b$. 
\item \emph{(Local character on a club)} for every finite tuple $a$ and set of parameters $\B$ there is $\mathcal{A} \subseteq [\B]^{\leq |\T|}$ such that $a \downfree_{\A} \B$ for all $\A \in \mathcal{A}$. 
\item \emph{(Anti-reflexivity)}  $a \downfree_{\A}^{\GS} a$ if and only if $a \in \acl(\A)$, 
\item \emph{(Algebraicity)} if $a \downfree_{\A}^{\GS} b$ then $a \downfree_{\A}^{\GS} \acl(b)$ and $\acl(a) \downfree_{\A}^{\GS} b$. 
\end{enumerate}
\end{theorem}
\begin{proof}
This is \cite[Theorem $2.2$]{treeless}. 
\end{proof}
\begin{remark}\label{fullexistenceGS} By $(4)$, $(5)$ and $(3)$ the relation $\downfree^{\GS}$ satisfies \emph{ left and right full existence}, i.e. for any $a,b$ tuples in $\M$ and $\A$ a set of parameters there is $a' \equiv_{\A} a$ such that $a' \downfree_{\A}^{\GS} b$, and some $b' \equiv_{\A} b$ such that $a \downfree_{\A}^{\GS} b'$.
\end{remark}
\begin{proposition}\label{prop: forking stronger than GS and symmetry for GS}
\begin{enumerate}
\item  If $a \downfree_{\A}^{f} b$ or $b \downfree_{\A}^{f} a$ then $a \downfree_{\A}^{\GS} b$. 
\item Assume $\p^{\otimes n}$ is generically stable over $\A$ for all $n < \omega$ and let $a \models \p\rest_{\A}$. Then for any $b$, $a \downfree_{\A}^{\GS} b$ if and only if $b \downfree_{\A} ^{\GS} a$. 
\end{enumerate}
\end{proposition}
\begin{proof}
The first part of the statement is \cite[Lemma $2.1$]{treeless}. For $(2)$, the left to right direction follows by the exact same argument as in \cite[Proposition $2.4$]{treeless} (instead of working with $\pi$ we work with $\p^{\tensor n}$). For the converse, assume that $b \downfree_{\A}^{\GS} a$, then $a \models \p \rest_{\A b}$. %(as $\p\rest_{\A}$ is stationary)
Hence $a \downfree_{\A}^{f} b$, and by the first statement we must have $a \downfree_{\A}^{\GS} b$. 
\end{proof}

\begin{definition}\label{def:GSMorley}
    A GS-Morley sequence over $\A$ is a sequence $(a_i)_{i\in \I}$ which is indiscernible over $\A$ and such that $a_{<i} \downfree^{\GS}_\A a_i$ for all $i \in \I$.
\end{definition}

Note that we ask for $a_{<i} \downfree^{\GS}_\A a_i$ and not $a_i \downfree^{\GS}_\A a_{<i}$. In particular, if $\pi(x)$ is the maximal generically stable partial type consistent with $\tp(a/\A)$ and the sequence $a_0=a, a_1, \ldots$ is $\A$-indiscernible with $a_i \models \pi \rest_{A a_{<i}}$, then that sequence is a GS-Morley sequence over $\A$.

\begin{lemma}\label{lem: gs-sequencences indiscernible over a}
Let $(a_i)_{i\in \I}$ be a GS-Morley sequence over $\A$ and assume that $(a_i)_{i\in \I}$ is indiscernible over $\B\supseteq \A$. Then $\B \downfree^{\GS}_\A a_i$ for all $i< \omega$.
\end{lemma}
\begin{proof}
Let $\pi$ be a generically stable partial type over $\A$ such that $a_0 \models \pi \rest_\A$. As $(a_i)_{i\in \I}$ is a GS-Morley sequence over $\A$, we have $a_i \models \pi \rest_{\A a_{<i}}$ for all $i<\omega$. Let $\phi(x,b) \in \pi\rest_\B$. By the definition of generically stable partial types, we have $\phi(x,b) \in \tp(a_i/\B)$ for all but at most finitely many values of $i$. Since the sequence $(a_i)_{i\in \I}$ is indiscernible over $\B$, we must have $\phi(x,b) \in \tp(a_i/\B)$ for all $i$ and in particular $\phi(x,b) \in \tp(a_0/\B)$. This shows $\B \downfree^{\GS}_\A a_0$, and hence $\B \downfree^{\GS}_\A a_i$ by indiscernibility.
\end{proof}

GS-Morley sequences exist over any set. Those will however not be enough for our purposes: we will need a sequence which is a \emph{total GS-Morley sequence} in the sense that we have $a_{\J} \downfree^{GS} a_{\J'}$ for all $\J', \J \subseteq \I$ with $\J < \J'$. To show their existence, we will use tree-indiscernibles as in \cite{kaplanramsey}. The link between them and total Morley sequences was made in \cite{JH}.

We start by introducing tree bookkeeping notation, which we repeat from \cite[Section $5.1$]{kaplanramsey}.
\begin{notation} For any ordinal $\alpha$, $\La_{s,\alpha}$ is the language $\{ \trianglelefteq, \wedge, <_{lex}, (\mathrm{P}_{\beta})_{\beta < \alpha}\}$, where $\trianglelefteq$ and $<_{lex}$ are binary relations, $\wedge$ is a binary function and each $\mathrm{P}_{\beta}$ is a unary relation. We may view a tree with $\alpha$ levels as
an $\La_{s,\alpha}$-structure where we interpret:
\begin{itemize}
\item $\trianglelefteq$ as the tree partial order,
\item $\wedge$ as the binary meet
function,
\item $<_{lex}$ as the lexicographic order, and
\item $P_{\beta}$ as the level $\beta$. 
\end{itemize}
\end{notation}

\begin{definition}  Let $\alpha$ be an ordinal. We define $\mathcal{T}_{\alpha}$ to be the set of functions $\f$ such that:
\begin{itemize}
\item $ran(\f)\subseteq \omega$;
\item $dom(\f)$ is an end-segment of $\alpha$ of the form $[\beta, \alpha)$, where $\beta$ is 0 or a successor ordinal, and if $\alpha$ is a successor ordinal we allow $\beta = \alpha$, that is $f=\emptyset$;
\item $\f$ has finite support, i.e.  the set $\{ y \in dom(\f) : \f(y) \neq 0 \}$ is finite.
\end{itemize} 
 We interpret $\mathcal{T}_{\alpha}$ as a $\La_{s, \alpha}$-structures in the following way:
\begin{itemize}
\item $\f \trianglelefteq \g$ if and only if $\f \subseteq \g$;
\item $\f \wedge \g= \f\rest_{[\beta, \alpha)}= \g\rest_{[\beta, \alpha)}$ where $\beta= \min \{ \gamma \ | \ \f\rest_{[\gamma, \alpha)}= \g\rest_{[\gamma, \alpha)} \}$, if non-empty (note that $\beta$ will not be a limit, by finite support). If the set is empty, we define $\f \wedge \g$ to be the empty function (note that this cannot hold if $\alpha$ is a limit);
\item $\f <_{lex} \g$ if and only if $\f \trianglelefteq \g$ or $ \f$ and $\g$ are $\trianglelefteq$-incomparable and  $\f(\gamma)< \g(\gamma)$ where $\dom(\f \wedge \g)=[\gamma+1, \alpha)$;
\item $\mathrm{P}_{\beta}(\f)$ holds if and only if $\dom(\f)=[\beta, \alpha)$. 
\end{itemize}
\end{definition}
\begin{notation}
\begin{itemize}
%\item   We write $\mathcal{F}_{\alpha+1} := \mathcal{T}_{\alpha+1} \backslash \{\emptyset\}$.
\item If $\f \in \mathcal{T}_{\alpha}$ with $\dom(\f)=[ \beta+1, \alpha)$ and $i < \omega$ then $\f^{\smallfrown}\langle i \rangle$ denotes the function $\f \cup \{ (\beta, i)\}$. While, $\langle i \rangle^{\smallfrown} \f$ denotes the element of $\mathcal{T}_{\alpha+1}$ given by $\f \cup \{ (\alpha,i)\}$.
\item for $\beta< \alpha$ we write $\zeta_{\beta}$ to denote the function with domain $[\beta, \alpha)$ such that $\zeta_{\beta}(\gamma)=0$ for all $\gamma \in [\beta, \alpha)$.
\item For $\alpha<\beta$, we define the canonical embedding $\iota_{\alpha \beta}\colon \mathcal{T}_{\alpha}\to\mathcal{T}_{\beta}$  by extending all elements of $\mathcal{T}_{\alpha}$ by 0 on $[\alpha, \beta)$. Note that this is an $\La_{s,\alpha}$-embedding.
\end{itemize}
\end{notation}
\begin{definition} We say that a tree  $(b_{\f})_{\f \in \mathcal{T}_{\alpha}}$ is $s$-indiscernible over $\A$ if for any tuples $\f_{0},\dots, \f_{n-1}$ and $\g_{0}, \dots, \g_{n-1}$ in $\mathcal{T}_{\alpha}$ with $\f_{0}, \dots, \f_{n-1} \equiv^{\mathrm{qf}} \g_{0},\dots, \g_{n-1}$, then $b_{\f_{0}}, \dots, b_{\f_{n-1}} \equiv_{\A} b_{\g_{0}}, \dots, b_{\g_{n-1}}$ where quantifier free type is in the language $\La_{s,\alpha}$.
%We also say that $(b_{\f})_{\f \in \mathcal{F}_{\alpha}}$ is $s$-indiscernible over $\A$ (for $\alpha$ a non-limit ordinal) if it is $s$-indiscernible over $\A$ after adding a root node $b_{\epsilon_{\alpha}}=\emptyset$. \textcolor{red}{can you introduce the notation for $b_{\epsilon_{\alpha}}$? }
\end{definition}
Given $\f \in \mathcal{T}_{\alpha}$, we denote as $b_{\trianglerighteq \f}$ a fixed enumeration of the set $\displaystyle{\{ b_{\g} \ | \ \g \in \mathcal{T}_{\alpha} \ | \ \f \trianglelefteq \g \}}$. We choose this enumeration so that if $(b_{\f})_{\f \in \mathcal{T}_{\alpha}}$ is $s$-indiscernible over $\A$, then for any $\f$ of successor length, the sequence $(b_{\f^{\smallfrown} \langle i \rangle})_{i< \omega}$ is $\A$-indiscernible.

If $w\in [\alpha \setminus \lim(\alpha)]^{<\omega}$ is a finite subset of $\alpha \setminus \lim(\alpha)$, we let $\mathcal{T}_\alpha \upharpoonright w$ be the \emph{restriction of $\mathcal{T}_\alpha$ to the sets of level $w$}, that is
\[\mathcal{T}_\alpha \upharpoonright w = \{\f\in \mathcal T_\alpha\colon \min(\dom(\f))\in w \text{ and }\beta \in \dom(\f)\setminus w \Rightarrow \f(\beta)=0\}.\]

We adapt Defintion 5.7 of \cite{kaplanramsey}.
\begin{definition}
\begin{enumerate}
    \item We say that a tree $(a_\f)_{\f\in \mathcal{T}_\alpha}$ is \emph{spread out} over $\A$ if for all $\f\in \mathcal{T}_\alpha$ with $\dom(\f) = [\beta+1, \alpha)$ for some $\beta<\alpha$, the sequence of subtrees $(a_{\trianglerighteq \f^{\smallfrown} \langle i \rangle})_{i<\omega}$ is a GS-Morley sequence over $\A$.
    \item The tree $(a_\f)_{\f\in \mathcal{T}_\alpha}$ is a \emph{Morley tree} over $\A$ if it is spread out and s-indiscernible over $\A$ and furthermore for all $v,w \in [\alpha \setminus \lim(\alpha)]^{<\omega}$ with $|v|=|w|$,
    \[(a_\f)_{\f\in \mathcal{T}_\alpha \upharpoonright v} \equiv_\A (a_\f)_{\f\in \mathcal{T}_\alpha \upharpoonright w}.\]
\end{enumerate}
\end{definition}
We now restate results from \cite{kaplanramsey} in our context. In that paper, the authors work with actual Morley sequences of invariant types. We work instead with GS-Morley sequences. By doing that, we lose invariance and in order to carry out the construction of an s-indiscernible tree, we will need to recover it by extraction. The following two lemmas will make that possible.

\begin{lemma}\label{lem:modeling}
Let $(a_\f)_{\f\in \mathcal{T}_\alpha}$ for some $\alpha$, then for any $\beta$, there is $(b_\g)_{\g\in \mathcal{T}_\alpha}$ which is \emph{based on} $(a_\f)_{\f\in \mathcal{T}_\alpha}$, meaning that for any finite tuple $\bar \g = (\g_1,\ldots, \g_k) \in \mathcal{T}_\alpha^k$ and formula $\phi(x_1,\ldots,x_n) \in \tp(\bar \g)$, there is $\bar \f = (\f_1,\ldots,\f_n)\in \mathcal{T}_\alpha^k$ such that:

-- $\bar \f$ and $\bar \g$ have the same quantifier-free $\La_{s,\alpha}$-type;

-- $\models \phi(a_{\f_1},\ldots, a_{\f_n})$.
\end{lemma}
\begin{proof}
By compactness it is enough to show this for a finite $\alpha$ and this is Theorem 4.3 in \cite{tree_indiscernibles} (see the first sentence of the proof there which reduces to the case of tress of finite heights by the same compactness argument).
\end{proof}

\begin{lemma}\label{lem: limit of GS-ind is GS-ind}
    Let $c$ be a (possibly infinite) tuple and $\B \supseteq \A$ a set of parameters. Assume that for any formula $\phi(x, b) \in \tp(c/\B)$, with $b$ a tuple from $\B$, there exists $c', b'$ such that $\tp(c'/\A) = \tp(c/\A)$, $\phi(c', b')$ holds and $b' \downfree^{\GS} c'$. Then $\B \downfree^{\GS} c$.
\end{lemma}
\begin{proof}
This follows at once from the definition of GS-independence: Assume that $\B \ndownfree^{\GS}_\A c$. Then there is a generically stable partial type $\pi(x)$ over $\A$ such that $c \models \pi \rest_\A$, but $c \nmodels \pi \rest_{\B}$. Let $\phi(x,b) \in \tp(c/\B)$ be such that $\neg \phi(x,b) \in \pi$. As $\pi$ is Ind-definable over $\A$, there exists $\theta(y)\in \tp(b/\B)$ such that $\neg \phi(x,b') \in \pi$ for all $b' \models \theta(y)$. By assumption, we can find $a' \downfree_\A c'$ such that $\tp(c'/\A) = \tp(c/\A)$ and $\phi(c',b')\wedge \theta(b)$ holds. As $\tp(c'/\A) = \tp(c/\A)$, we have $c' \models \pi \rest_\A$. As $b' \downfree_\A c'$, we have $c' \models \pi \rest_{\A b'}$. This is a contradiction since $\neg \phi(x,b') \in \pi$. 
\end{proof}

\begin{lemma}
    Let $\p(x) \in \mathcal{S}(\A)$ be given, then for any ordinal $\alpha$, there is a spread out and s-indiscernible tree $(c_\f)_{\f \in \mathcal T_\alpha}$ over $\A$ such that $\tp(c_\f/\A) = p$ for all $\f$.
\end{lemma}
\begin{proof}
The construction is similar to \cite[Lemma 5.11]{kaplanramsey} with an extra extracting step to make trees indiscernible. 

We argue by induction on $\alpha$, building an increasing sequence of trees, respecting the canonical embeddings $\iota_{\alpha \beta}$. For the case $\alpha = 1$, take a GS-Morley sequence $(a_k)_{k>= - 1}$ in $p$ over $\A$. Set $c^1_\emptyset = a_{-1}$ and $c^1_{\langle i\rangle} = a_i$ for $i<\omega$.

At a limit step, we take the union of the trees constructed so far. Since the canonical embeddings respect the $\La_{s,\alpha}$-structure, s-indiscernibility goes through automatically. To see that the resulting $(c_\f)_{\f \in \Tree_\alpha}$ is spread-out, let $\f \in \mathcal T_\alpha$ with $\dom(\f) = [\beta+1,\alpha)$. Then as the domain of $\f$ is finite, $\f$ already belongs to some $\Tree_\gamma$ for some $\gamma < \alpha$. Then the sequence of subtrees $(a_{\trianglerighteq \f^{\smallfrown} \langle i \rangle})_{i<\omega}$ is a sequence of subtrees of $\Tree_\gamma$ and hence a GS-Morley sequence over $\A$ by induction.

At a successor step $\alpha+1$, let $q$ denote the type of the tree constructed at step $\alpha$. Build a GS-Morley sequence $(b_i)_{i<\omega}$ of $\q$ over $\mathrm{A}$. Let $u$ be a realization of $\p$ so that $u \downfree^{GS}_\mathrm{A} b_{<\omega}$. Let $u^{\alpha+1}_\emptyset=u$ be the root of the tree $\mathcal T_{\alpha+1}$ and place the sequence $(d_i)_{i<\omega}$ above it so that $d_i = u^{\alpha+1}_{\trianglerighteq \langle i\rangle}$. This gives a tree $(u_\f)_{\f \in \Tree_{\alpha+1}}$ which extends the tree constructed at the previous step. This tree is spread-out by construction, but is not necessarily s-indiscernible.

By Lemma \ref{lem:modeling}, there is a tree $(c_\f)_{\f\in \Tree_{\alpha+1}}$ that is s-indiscernible and based on the tree $(u_\f)_{\f\in \Tree_{\alpha+1}}$ constructed above. We now have to check that this new tree is also spread-out.

Take a tuple $(\f_1,\ldots, \f_n), (\g_1,\ldots,\g_n) \in \Tree_{\alpha+1}^n$ with the same $\La_{s,\alpha+1}$-type. Assume first that $\f_i(\alpha)=\g_i(\alpha)=0$ for all $i$, so that all the element $u_{\f_i}$, $u_{\g_i}$ lie in $b_0$. The type of the tree $b_0$ is the same as that of $(c_\f)_{\f\in \Tree_\alpha}$ from the previous stage of the construction. That tree is s-indiscernible, hence it follows that $\tp(u_{\f_1},\ldots, u_{\f_n}/\A) = \tp(u_{\g_1},\ldots, u_{\g_n}/\A)$. Assume next that the meet $(\f_1,\ldots,\f_n)$ has level $<\alpha+1$, then $\f_1, \ldots, \f_n$ all lie inside one of the subtrees $b_i$ and the same is true of $(\g_1,\ldots,\g_n)$. The trees $b_i$ all have the same type as $b_0$, hence the same conclusion holds: $\tp(u_{\f_1},\ldots, u_{\f_n}/\A) = \tp(u_{\g_1},\ldots, u_{\g_n}/\A)$. It follows that the extraction step cannot change the types of tuples whose meet is higher than the root: if $\f_1, \ldots, \f_n$ have have a meet that is strictly above the root of $\Tree_{\alpha+1}$, then $\tp(c_{\f_1},\ldots, c_{\f_n}/\A)=\tp(u_{\f_1},\ldots, u_{\f_n}/\A)$. In particular, we can assume that the tree $(c_\f)_{\f\in \Tree_{\alpha+1}}$ extends the tree $(c_\f)_{\f\in \Tree_\alpha}$ from the previous stage.

We now check that the tree is spread-out. Let $\f \in \Tree_{\alpha+1}$ with $\dom(\f) = [\beta+1, \alpha+1)$ for some $\beta \leq \alpha$. If $\beta < \alpha$, then by the previous paragraph, the sequence $(c_{\trianglerighteq \f^{\smallfrown} \langle i \rangle})_{i<\omega}$ has the same type as $(u_{\trianglerighteq \f^{\smallfrown} \langle i \rangle})_{i<\omega}$ over $\A$ and all those elements lie within one tree $b_i$. As all those trees have the same type as the tree from step $\alpha$ of the construction, and that tree is spread out, if follows that $(u_{\trianglerighteq \f^{\smallfrown} \langle i \rangle})_{i<\omega}$ and hence $(c_{\trianglerighteq \f^{\smallfrown} \langle i \rangle})_{i<\omega}$ is a GS-Morley sequence over $\A$ as required.

Assume now that $\beta = \alpha$, that is $\f = \emptyset$. The sequence 
$(u_{\trianglerighteq \f^{\smallfrown} \langle i \rangle})_{i<\omega}= (u_{\trianglerighteq \langle i \rangle})_{i<\omega}$ is just the sequence of trees $(b_i)_{i< \omega}$, and that is a GS-Morley sequence over $\A$ by construction. We have to show that the same is true of the sequence $(c_{\trianglerighteq \langle i \rangle})_{i<\omega}$. For ease of notation, set $C_i = c_{\trianglerighteq \langle i \rangle}$ and $U_i = u_{\trianglerighteq \langle i \rangle}$. We want to apply Lemma \ref{lem: limit of GS-ind is GS-ind}. Pick some $i<\omega$, and a finite tuple $e_i \in C_i$. Let $\phi(x; e_0, \dots, e_{i-1}) \in \tp(c/C_0 \ldots C_{i-1})$, where each $e_j$ is a tuple from $C_j$. By construction, we can find indices $k_1<\cdots < k_i$ and tuples $e'_j \in U_{k_j}$ such that $\phi(e'_i; e'_0, \dots, e'_{i-1})$ and furthermore (by the previous paragraph) $\tp(e'_j/\A) = \tp(e_j/\A)$ for all $j$. We can now apply Lemma \ref{lem: limit of GS-ind is GS-ind} to conclude that $e_{<i} \downfree^{GS} e_i$. Therefore $C_{<i} \downfree^{GS}_\A C_i$ by finite character and $(c_{\trianglerighteq \langle i \rangle})_{i<\omega}$ is a GS-Morley sequence over $\A$ as required. This shows that the tree built at step $\alpha+1$ is spread-out and finishes the proof of the lemma.
\end{proof}

\begin{lemma}
    Suppose $(a_\f)_{\f\in \mathcal T_\kappa}$ is a tree of tuples spread out and s-indiscernible over $M$. If the ordinal $\kappa$ is large enough, there is a Morley tree $(b_\f)_{\f\in \mathcal T_\omega}$ so that for all $w \in [\omega]^{<\omega}$, there is $v \in [\kappa \setminus \lim(\kappa)]^{<\omega}$ so that
    \[(a_\f)_{\f\in \mathcal T_\kappa \upharpoonright v} \equiv (b_\f)_{\f\in \mathcal T_\omega \upharpoonright w}.\]
\end{lemma}
\begin{proof}
    The proof is exactly the same as that of \cite[Lemma 5.10]{kaplanramsey}, replacing "Morley sequences of an invariant type" by "GS-Morley sequences". (The only property of Morley sequences of invariant types that is used in the proof is the fact that this property depends only on the type of the sequence over the base, which is of course also true for GS-Morley sequences.)
\end{proof}

\begin{proposition}\label{prop: total GS sequences exist} Given $b$ and $\A$ there is an $\A$-indiscernible sequence $(b_{i})_{i<\omega}$, $b_0=b$, such that $b_{\J} \downfree_{\A}^{\GS} b_{\I}$ for all $\I < \J$. 
\end{proposition}
 \begin{proof}
By the two previous lemmas applied successively, we construct a Morley tree $(c_\f)_{\f \in \mathcal T_\omega}$ over $\A$ such that $\tp(c_\f/\A) = \tp(b/\A)$ for all $\f$. For $k<\omega$, set $b_k = c_{\zeta_{k}}$.

By the indiscernibility assumption on the tree, this sequence is indiscernible over $\mathrm{A}$. It is then enough to show that for every $k<\omega$, we have \[b_{>k} \downfree^{GS}_\mathrm{C} b_{\leq k}.\]

Fix $k<\omega$. Consider the sequence $(e_i)_{i<\omega}$ defined by $e_i = c_{{\zeta_{k+1}}^{\smallfrown}\langle i \rangle}$. This is a GS-Morley sequence over $\mathrm{A}$. It is also indiscernible over $b_{> k}$. By Lemma \ref{lem: gs-sequencences indiscernible over a}, $b_{>k} \downfree^{GS}_\mathrm{A} e_0$. Since $b_{\leq k}$ is subtuple of $e_0$, we have what we want.
 \end{proof}
 \subsubsection{$\p$-independence}
\begin{definition} Let $\p$ be a global $\A$-invariant type. Assume $\p^{\otimes n}$ is generically stable for all $n < \omega$. For (possibly infinite) tuples $a,b$, we define $a \downfree_{\A}^{\p} b$ as: 
\begin{center}
for any $\bar{d} \models \p^{\otimes n}\rest_{\A}$ there is $\bar{d}' \equiv_{\A a} \bar{d}$ such that $\bar{d}' \models \p^{\otimes n} \rest_{\A b}$.
\end{center}
\end{definition}

\begin{proposition}\label{prop: basic facts on p rel} Assume $\p^{\otimes n}$ is generically stable for all $n < \omega$. The following are equivalent:
\begin{enumerate}
\item $a \downfree_{\A}^{\p} b$;
\item  There is $\bar{e} \models \p^{\alpha}\rest_{\A}$ a $\p$-basis of $a$ over $\A$ with $\bar{e} \models \p^{\alpha}\rest_{\A b}$;
\item $b \downfree_{\A}^{\p} a$. 
\end{enumerate}
Furthermore, given $a$, there is a generically stable partial type $\pi_{\p}(x)$ over $\A$ consistent with $\tp(a/ \A)$ such that $a \models \pi_{\p}(x) \rest_{\A b}$ if and only if $a \downfree_{\A}^{\p} b$. 
\end{proposition}
\begin{proof}
To prove $(1) \rightarrow (2)$, let $\bar d \models \p^\alpha$ be a $\p$-basis of $a$ over $A$. Then by definition of $\downfree_{\A}^{\p}$ and compactness, there is $\bar e \equiv_{\A a} \bar d$ with $\bar{e} \models \p^{\alpha}\rest_{\A b}$.

For $(2) \rightarrow (3)$ let $\bar{d} \models \p^{\alpha} \rest_{\A}$. As $\bar{e} \models \p^{\alpha} \rest_{\A b}$, we can find $\bar{d}' \equiv_{\A b} \bar{d}$ such that $\bar{e} \models \p^{\alpha}\rest_{\A b \bar{d}'}$. Since $\bar{e}$ is a $\p$-basis of $\tp(a/\A)$ and $\bar{d'}$ is independent from $\bar{e}$, we have $\bar{d}' \models \p^{\otimes n} \rest_{ \A a}$ as required. For $(3) \rightarrow (1)$ we argue as in $(2) \rightarrow (3)$ exchanging the roles of $a$ and $b$.\\

For the last part of the statement, let $\bar{e} \models \p^{\alpha}\rest_{\A}$ a $\p$-basis of $a$ over $\A$. Set $\alpha(\bar{y})=\p^{\alpha}$ and $\rho(x,\bar{y}) = \tp(a,\bar e/\A)$. By \cref{lem: existential gen stable}, the partial type $\pi_{\p}(x):= \exists \bar{y} \big(\alpha(\bar{y})\wedge \rho(x,y)\big)$ is generically stable over $\A$. Then $\pi_{\p}(x)$ is a generically stable type over $\A$ which is consistent with $\tp(a/\A)$. And $a \downfree_{\A}^{\p} b$ if and only if $b \downfree_{\A}^{\p} a$ if and only if $a \models \pi_{\p}(x)\rest_{\A b}$.
\end{proof}
\begin{lemma}\label{ lem: properties p relation} Let $\p$ be a global type, $a,b, c$ be finite tuples in $\M$ and $\A$ a small set of parameters. Assume that $\p^{\otimes n}$ is generically stable over $\A$ for all $n<\omega$. \\
The relation $\downfree^{\p}$ has the following properties:
\begin{enumerate}
\item \emph{(Finite Character)} if for all finite tuples $a' \subseteq a$, $b' \subseteq b$, $a' \downfree_{\A}^{\p} b'$, then $a \downfree_{\A}^{\p} b$.
\item \emph{(Extension)} if $a \downfree_{\A}^{\p} b$  then there is $c' \equiv_{\A b} c$ such that $a \downfree_{\A}^{\p} b c'$. 
\item \emph{(Symmetry)} $a \downfree_{\A}^{\p} b$ if and only if $b \downfree_{\A}^{\p} a$. 
\item \emph{(Transitivity)} If $a \downfree_{\A c}^{\p} b$ and $ c \downfree_{\A}^{\p} b$, then $ac \downfree_{\A}^{\p} b$. 
\item If $\bar{a} \models \p ^{\alpha} \rest_{\A}$ then $\bar{a} \downfree_{\A}^{\p} b$ if and only if $\bar{a} \models \p^{\alpha}\rest_{\A b}$. 
\end{enumerate}
\end{lemma}
\begin{proof}
Finite character follows immediately from the definition and compactness.

For extension, since $a \downfree_{A}^{\p} b$ one can take $\bar{e}$ a $\p$-basis of $a$ over $\A$ such that $\bar{e} \downfree{f}_{\A} b$ given by \cref{prop: basic facts on p rel}. Let $c' \equiv_{\A b} c$ such that $\bar{e} \downfree_{ \A} b c'$ (\emph{c.f.} \cref{prop: properties of non-forking}.(7)), then $a \downfree_{\A}^{\p} bc'$ by \cref{prop: basic facts on p rel}. Symmetry follows by \cref{prop: basic facts on p rel}.

We aim to show transitivity. By symmetry it is enough to show that $b \downfree_{\A}^{\p} ac$, this is given $\bar{d} \models \p^{\otimes n} \rest_{\A}$ there is 
$\bar{d}'' \equiv_{\A b} \bar{d}$ so that $\bar{d}'' \models \p\rest_{\A ac}$. 
Because $c \downfree_{\A}^{\p} b$, by symmetry, $b \downfree_{\A}^{\p} c$. Thus for any  $\bar{d} \models \p^{\otimes n} \rest_{\A}$  there is some $\bar{d}' \equiv_{\A b} \bar{d}$ so that $\bar{d}' \models \p^{\otimes  n} \rest_{\A c}$. By symmetry, $b \downfree_{\A c}^{\p} a$ as $a \downfree_{\A c}^{\p} b$. Then there is $\bar{d}'' \equiv_{\A bc} \bar{d}'$ such that $\bar{d}'' \models \p^{\otimes n}\rest_{\A ac}$. We conclude that $b \downfree_{\A}^{\p} ac$ as $\bar{d}'' \equiv_{\A b} \bar{d}$.\\
For the last statement, assume that $\bar{a} \models \p^{\alpha}\rest_{\A}$ then $\bar{a}$ is a $\p$-basis of itself over $\A$, so by \cref{prop: basic facts on p rel} we have $\bar{a} \downfree_{\A}^{\p} b$ if and only if $\bar{a} \models \p^{\alpha} \rest_{\A b}$. 
\end{proof}
\subsubsection{$\acl^{\vee}$-independence}

By \emph{graph}, we always mean undirected graph.

\begin{definition}\label{def: def code} Let $\A$ be a set of parameters.
\begin{itemize}
\item Let $\mathrm{G}=(\V, \R)$ be an $\emptyset$-definable graph, i.e. the set $V$ and the relation $\R(x,y)$ are $\emptyset$-definable. Given $a, a' \in \V$ and $n \in \mathbb{N}_{\geq 1}$ we write $\dist_{\mathrm{G}}(a,a') \leq n$ to denote that there is a path from $a$ to $a'$ in $\mathrm{G}$ of length less or equal than $n$. Note that this is expressible by a first order formula $\phi_{\dist_{\mathrm{G}, n}}(x,y)$. Given $\Y \subseteq \V$ we denote $\mathrm{diam}^{\mathrm{G}}(\Y)=\sup\{ \dist_{\mathrm{G}}(y,y') \ | \ (y,y') \in \Y \times \Y\}$. 
We write $[a]$ to denote the \emph{connected component} of $a \in \V$ in the graph $\mathrm{G}$, i.e. $[a]=\{ b \in \V \ | \ \dist_{\mathrm{G}}(a,b)< \infty\}$.
\item We say that the connected component $[a]$ is \emph{$\vee$-definable over $\A$}, if every automorphism $\sigma \in \Aut(\mathrm{M}/\A)$ fixes $[a]$, i.e. $[a]=[\sigma(a)]$. Likewise we say that the connected component $[a]$ is \emph{$\vee$-algebraic over $\A$} if the set $\{[\sigma(a)]:\sigma \in \Aut(\mathrm{M}/\A)\}$ is finite.
\item For each natural number $\ell \in \mathbb{N}_{\geq 2}$, define the relation $\R_{n}(x_{1},\dots,x_{\ell}; y_{1}, \dots, y_{\ell})$ by $\displaystyle{\bigvee_{\sigma \in \mathcal{S}_{\ell}} \bigwedge_{i \leq n} \R(x_{i}, y_{\sigma(i)})}$, where $\mathcal{S}_{\ell}$ is the set of permutations of $\ell$-elements. 
The relation $\R_{\ell}$ defines a graph on $\V^{\ell}$. We write $\mathrm{Code}(a_{1},\dots, a_{\ell})$ to denote the connected component of $(a_{1},\dots, a_{\ell})$ in the graph $(\V^{\ell}, \R_{\ell})$.  
\item We define $\acl^{\vee}(\A)$ to be the set of connected components of $\emptyset$-definable graphs that are $\vee$-algebraic over $\A$, and $\dcl^{\vee}(\A)$ the set of connected components of $\emptyset$-definable graphs that are $\vee$-definable over $\A$. Note that if $[a] \in \acl^{\vee}(\A)$ then $\Code(a_{1},\dots, a_{n}) \in \dcl^{\vee}(\A)$, where $\{[a_{1}], \dots, [a_{n}]\}$ is the set of conjugates of $[a]$ under the action of $\Aut(\mathrm{M}/\A)$. 
\end{itemize}
\end{definition}

\begin{remark}\label{rem: infinite tuples} In the definition above and in all that follows, tuples of variables are allowed to be infinite. Of course, if $x$ and $y$ are infinite tuples of variables and $\mathrm{G}=(\V(x), \R(x,y))$ is a definable graph, then only finitely many variables from the tuples $x$ and $y$ actually appear in the definition of $\mathrm{G}$. The rest are just dummy variables. Since we will be often manipulating infinite tuples (for instance infinite Morley sequences), it is convenient to allow such tuples in the definitions. None of the definitions or statements are affected by adding dummy variables to the formulas, so this creates no difficulties.
\end{remark}

\begin{remark}\label{rem: reduction to graphs over the emptyset} Let $\A$ be a set of parameters and $\mathrm{G}=(\V, \R)$ be an $\A$-definable graph. Let $c$ be a finite tuple in $\A$ such that $\V(x,c)$ and $\R(x,y, c)$ are $\La(c)$-formulas defining $\mathrm{G}$. Consider $\V'=\{ (x,z) \ | \ \mathrm{M} \models \V(x,z)\}$ and the relation $\R'$ on $\V'$ defined by:
\begin{equation*}
(x, z) \R' (y, z') \ \text{if and only if} \ z=z'\, \wedge \, \R(x,y, z).
\end{equation*}
Let $\mathrm{G}'=(\V', \R')$. This is a $\emptyset$-definable graph. For any element $a \in \V$ let $[a]$ be the connected $\mathrm{G}$-component and $[(a,c)]$ the connected $\mathrm{G}'$-component. Then for any $y \in \V$, $y \in [a]$ if and only if $(y,c) \in [(a,c)]$. In particular, for any $\A \subseteq \B$ the connected component $[a]$ has finitely many conjugates under the action of $\Aut(\mathrm{M}/ \B)$ if and only if $[(a,c)]$ has finitely many conjugates under the action of $\Aut(\mathrm{M}/ \B)$.

Thanks to this observation, we will only need to consider $\emptyset$-definable graphs, and not $\A$-definable graphs.
\end{remark}

\begin{lemma}\label{lem: characterizing being in aclV}
Let $\A$ be a small set of parameters and let $[a]$ be the connected component of some $\emptyset$-definable graph $\mathrm{G}$. Then $[a] \in \dcl^{\vee}(\A)$ if and only if there is a formula $\phi(x) \in \tp(a/\A)$, such that $\phi(x)$ implies that $x \in \mathrm{G}$ and $ \mathrm{diam}(\phi) \leq n$. Moreover, $[a] \in \acl^{\vee}(\A)$ if and only if there is an $\La(\A)$-formula $\psi(x) \in \tp(a/\A)$ that intersects finitely many connected $\mathrm{G}$-components, which are precisely the conjugates of $[a]$ under the action of $\Aut(\mathrm{M}/\A)$. 
\end{lemma}
\begin{proof} Let $a' \equiv_{\A} a$ and $\sigma \in \Aut(\mathrm{M}/ \A)$ such that $\sigma(a)=a'$. Since $[a] \in \dcl^{\vee}(\A)$, then $[a]=[\sigma(a)]=[a']$. Consequently, $\dist_{\mathrm{G}}(a,a')<\infty$. Since this holds for an arbitrary $a' \equiv_{\A} a$, by compactness there is a formula $\phi(x) \in \tp(a/\A)$ that implies that $x \in \mathrm{G}$ and  $\dist_{\mathrm{G}}(a,x)\leq k$ for some $k \in \mathbb{N}$. Then $\mathrm{diam}(\phi) \leq 2k=n$.

For the second part of the statement, $[a] \in \acl^{\vee}(\A)$ if and only if there are finitely many elements $a_{1}, \dots, a_{n}$ such that $\X=\{[a_{1}],\dots, [a_{n}]\}$ are the conjugates of $[a]$ under the action of $\Aut(\mathrm{M}/ \A)$. Given $d \equiv_{\A} a$ for some $i\leq n$ $[d]=[a_{i}]$,  thus $\displaystyle{\dist_{\mathrm{G}}(d, a_{i})<\infty}$. By compactness, there is some $k \in \mathbb{N}$ and $\phi(x) \in \tp(b/\A)$ such that for any $d\models \phi(x)$, we have $\displaystyle{\bigvee_{i=1}^{n} \dist_{\mathrm{G}}(d,a_{i})\leq k}$. Hence $\phi(x)$ intersects finitely many $\G$-components as required. 
\end{proof}
\begin{remark}\label{rem: finite character aclV closure} Let $\A$ be a small set of parameters. Then $\acl^{\vee}$ and $\dcl^{\vee}$ have \emph{finite character} i.e. $\displaystyle{\acl^{\vee}(\A)= \bigcup_{\A_{0} \subseteq_{\text{fin}} \A} \acl^{\vee}(\A_{0})}$ and $\displaystyle{\dcl^{\vee}(\A)= \bigcup_{\A_{0} \subseteq_{\text{fin}}\A} \dcl^{\vee}(\A_{0})}$.
\end{remark}
\begin{proof}
This is an immediate consequence of \cref{lem: characterizing being in aclV}.
\end{proof}

\begin{lemma}\label{lem: elements in aclV are definable when base is acl closed} Assume $\A=\acl(\A)$. If $[a] \in \acl^{\vee}(\A)$ then $[a] \in \dcl^{\vee}(\A)$.  
\end{lemma}
\begin{proof}
Let $[a] \in \acl^{\vee}(\A)$, then there is an $\emptyset$-definable graph $\mathrm{G}=(\V, \R)$  and finitely many elements $a_{1},\dots, a_{n}$ such that $[a]=[a_{1}]$ and $\{[a_{1}],\dots, [a_{n}]\}$ is the set of conjugates of $[a]$ under the action of $\Aut(\mathrm{M}/ \A)$. By \cref{lem: characterizing being in aclV} there is an $\A$-definable set $\X$ containing $a$ such that $\X$ intersects only the connected components $[a_{i}]$. By compactness, there is some $ n \in \mathbb{N}$ such that $\mathrm{diam}([a_{i}] \cap \X ) \leq n$. Consider the $\A$-definable equivalence relation $\mathrm{E}$ on $\X$ defined by:
\begin{equation*}
x\mathrm{E} y \ \text{if and only if} \ \mathrm{dist}_{\mathrm{G}} (x,y) \leq n \ \text{if and only if} \ \phi_{\mathrm{dist}_{\mathrm{G}, n}}(x,y).
\end{equation*}
Note that  $[a_{i}] \cap \X=\{ y \in \X \ | \ y \E a_{i}\}$ for all $i \leq n$. In particular, $a_{i}/ \mathrm{E} \in \acl(\A)=\A$ and $[a_{i}] \cap \X$ is an $\A$-definable set.  Consequently, $[a_{i}] \in \dcl^{\vee}(\A)$. 
\end{proof}
\begin{definition}
 We write $a \downfree_{\A}^{\acl^{\vee}} b$ if every connected component $[x]$ of an $\emptyset$-definable graph $\mathrm{G}$ that is $\vee$-algebraic over $\A a$ and $\A b$ is already $\vee$-algebraic over $\A$, i.e.
 \begin{equation*}
 \acl^{\vee}(\A a) \cap \acl^{\vee}(\A b)= \acl^{\vee}(\A).
 \end{equation*}
 Similarly, we write $a \downfree_{\A}^{\dcl^{\vee}} b$ to denote 
 \begin{equation*}
 \dcl^{\vee} (\A a) \cap \dcl^{\vee}(\A b) \subseteq \acl^{\vee}(\A).
 \end{equation*}
\end{definition}
Note that if $\acl$ satisfies exchange and hence defines a pregeometry, the corresponding notion of independence for $\acl$ coincides with independence in the pregeometry only when the latter is modular. So we cannot expect too much from this independence notion in general. However, it will be sufficient for our purposes, because we will mostly apply it to indiscernible sequences where it automatically becomes better behaved. For instance in a stable theory, it implies forking independence: see \cref{lem: acl implies Morley}.

\begin{lemma}
\begin{enumerate}
\item We have $a \downfree_{\A}^{\acl^{\vee}} b$ if and only if every connected component $[x]$ of an $\A$-definable graph $\mathrm{G}$ that is $\vee$-algebraic over $\A a$ and $\A b$ is already $\vee$-algebraic over $\A$.
\item If $\A=\acl(\A)$ if $a$ and $b$ are tuples, then $a \downfree_{\A}^{\dcl^{\vee}}b$ if and only if $\dcl^{\vee}(\A a) \cap \dcl^{\vee}(\A b) = \dcl^{\vee}(\A)$. 
\end{enumerate}
\end{lemma}
\begin{proof}
The first point follows from \cref{rem: reduction to graphs over the emptyset} and the second one follows from \cref{lem: elements in aclV are definable when base is acl closed}.

\end{proof}

\begin{lemma}\label{rem: properties aclV independence} Let $a,b, c$ be finite tuples in $\mathrm{M}$ and $\A$ a small set of parameters.
The relation $\downfree^{\acl^{\vee}}$ has the following properties:
\begin{itemize}
\item \emph{(Extension)}If $a \downfree_{\A}^{\acl^{\vee}} b$ then there is $c' \equiv_{\A b} c$ such that $a \downfree_{\A}^{\acl^{\vee}} bc'$;
\item \emph{(Symmetry)} $a \downfree_{\A}^{\acl^{\vee}} b$ if and only if $b \downfree_{\A}^{\acl^{\vee}} a$;
\item \emph{(Transitivity)} If $a \downfree_{\A c}^{\acl^{\vee}} b$ and $c \downfree_{\A}^{\acl^{\vee}} b$ then $ac \downfree_{\A}^{\acl^{\vee}} b$. 
\end{itemize}
\end{lemma}
\begin{proof}
Extension for $a \downfree_{\A}^{\acl^{\vee}} b$ follows from Neumann's Lemma as for $\acl$. ( \emph{c.f.} see for example \cite[Proposition $2.2$]{aclextension}). Transitivity and symmetry are straightforward to prove.
\end{proof}

\begin{lemma}\label{lem: equivalence indipendence dcl and acl} Let $\A$ be a small set of parameters and $\bar{b}=(b_{i})_{i \in \K}$ be an $\A$-indiscernible sequence in $\mathrm{M}$. The following are equivalent: 
\begin{enumerate}
\item $b_{\I} \downfree_{\A}^{\dcl^{\vee}} b_{\J}$ for all $\I \cap \J= \emptyset$;
\item $b_{\I} \downfree_{\A}^{\acl^{\vee}} b_{\J}$ for all $\I \cap \J= \emptyset$; 
\item $b_{\I} \downfree_{\A}^{\dcl^{\vee}} b_{\J}$ for $\I < \J$;
\item  $b_{\I} \downfree_{\A}^{\acl^{\vee}} b_{\J}$ for $\I < \J$. 
\end{enumerate}
Furthermore, if we take a sequence $(b'_{i})_{i \in (\mathbb{Z} \times \mathbb{Z}, \leq_{lex})}$ of same EM-type as $(b_{i})_{i \in K}$ and let $\mathbf{d}_{k}= (b'_{i})_{i \in \{k\} \times \mathbb{Z}}$, then $\mathbf{d}_{\I} \downfree_{\A \mathbf{d}_{n}}^{\acl^{\vee}} \mathbf{d}_{\J}$ holds for $n < \I < \J$.
\end{lemma}
\begin{proof} The direction $(2) \rightarrow (1)$ follows by definition. For $(1)\rightarrow (2)$, we may extend the sequence and assume that $\K$ is a dense linear order without end-points. Let $[a] \in \acl^{\vee}(\A b_{\I}) \cap \acl^{\vee}(\A b_{\J})$ and $\X=\{ [a_{i}] \ | i \leq n\}$ be the finite orbit of $[a]$ over $\A b_{\I}$. Without loss of generality we may assume that $\X$ has minimal cardinality, as we may increase the size of the tuple $b_{\I}$.\\
\underline{Claim:} \emph{ Let $\mathcal{J}$ be the set of finite tuples in $\K$ with the same order type as $\J$ with respect to $\I$. Then $[a] \in \acl^{\vee}(\A b_{\I}) \cap \acl^{\vee}(\A b_{\J'})$, for $\J' \in \mathcal{J}$.} \\
\textit{Proof:} By indiscernibility, $b_{\J} \equiv_{\A b_{\I}} b_{\J'}$ thus there is $\sigma \in \Aut(\mathrm{M}/ \A b_{\I})$ such that $\sigma([a])=[a]'$ and  $[a]' \in \acl^{\vee}(\A b_{\I}) \cap \acl^{\vee}(\A b_{\J'})$. Hence $[a],[a]' \in \acl^{\vee}(\A b_{\I}) \cap \acl^{\vee}(\A b_{\J} b_{\J'})$. If $[a] \notin \acl^{\vee}(\A b_{\J'})$ then $\mathrm{mult}([a]/ \A b_{\J} b_{\J'}) < \mathrm{mult}([a]/ \A b_{\J})$, but this is a contradiction to the minimality of $|\X|$.

\smallskip
By the Claim and the minimality of $|\X|$, the code $\Code(a_{1},\dots, a_{n})$ (\emph{c.f.} see \cref{def: def code}) lies in $\dcl^{\vee}(\A b_{\I}) \cap \dcl^{\vee}(\A b_{\J'})$ for every $\J' \in \mathcal{J}$. In particular, $\Code(a_{1},\dots, a_{n}) \in \dcl^{\vee}(\A b_{\I}) \cap \dcl^{\vee}(\A b_{\J}) \subseteq \acl^{\vee}(\A)$ by $\dcl^{\vee}$-independence. Since $[a] \in \acl^{\vee}(\Code(a_{1},\dots, a_{n}))$, then $[a] \in \acl^{\vee}(\A)$.

A similar argument gives the equivalence between $(3)$ and $(4)$.

The direction $(2)\rightarrow (3)$ is immediate. For the converse, by increasing the sequence we may assume that $\K$ is a dense order linear order without endpoints.  Let $[a] \in \dcl^{\vee}(\A b_{\I}) \cap \dcl^{\vee}(\A b_{\J})$. To simplify the notation let $\{ i_{1},\dots, i_{n}\}$ be an increasing enumeration of $\I$, and we will denote $b_{i_{1}},\dots, b_{i_{n}}$ the corresponding enumeration of $b_{\I}$. Set $i_{0}=-\infty$ and $i_{n+1}=\infty$ and let 
\begin{equation*}
K_{\ell}= \{ j \in \J \ | \ i_{\ell}< j< i_{\ell+1}\}, \text{for $\ell \in \{ 0, 1, \dots, n\}$}. 
\end{equation*}
Without loss of generality we assume each $K_{\ell}$ is not empty. For each tuple $ b_{K_{\ell}}$ we can find a tuple of the same length $ b_{K'_{\ell}}$ such that $K'_{\ell}<K_{\ell}$ and such that  $i_{\ell}< K_{\ell}'<i_{\ell+1}$. In particular, $\J$ and $\J'$ have the same order type over $\I$, where  $\displaystyle{\J'= \bigcup_{\ell \leq n} K_{\ell}'}$.  By indiscernibility over $\A$, there is $\sigma \in \Aut(\mathrm{M}/ \A b_{\I})$ such that $\sigma(b_{\J})=b_{\J'}$ and $\sigma([a])=[a]$. Consequently, $[a] \in \dcl^{\vee}(\A b_{\J}) \cap \dcl^{\vee}(\A b_{\J'})$.

Again, because $\K$ is a dense linear order without endpoints, for each $\ell \in \{ 0,\dots, n\}$ and tuple $ b_{K_{\ell}'}$ we can find a tuple of the same length $ b_{K''_{\ell}}$ such that $K''_{\ell}<K'_{\ell}$ and $K_{1}''<i_{1}$.  By indiscernibililty, there is an automorphism $\sigma \in \Aut(\mathrm{M}/ \A b_{\J})$ sending $\sigma(b_{\J'})=b_{\J''}$ and $\sigma([a])=[a]$. Thus $[a] \in \dcl^{\vee}(\A b_{\J}) \cap \dcl^{\vee}(\A b_{\J''})$. %%Now we can move the other sequence\\
By iterating this process finitely many times we can find a sequence $b_{\J^{*}}$ such that $[a] \in \dcl^{\vee}(\A b_{\J^{*}})$ and $\J^{*}< \I$.  Hence,  $[a] \in \acl^{\vee}(\A)$ as $b_{\I} \downfree_{\A}^{\dcl^{\vee}} b_{\J^{*}}$.

\smallskip
For the last part of the statement,  without loss of generality assume $\ell < \I < \J$, because the sequence $(\mathbf{d}_{k})_{k > \ell}$ is $\A \mathbf{d}_{\ell}$-indiscernible. By the first part of the statement, it is sufficient to argue that $\mathbf{d}_{\I} \downfree_{\A \mathbf{d}_{\ell}}^{\dcl^{\vee}} \mathbf{d}_{\J}$. Let $e \in \dcl^{\vee}(\A \mathbf{d}_{\ell} \mathbf{d}_{\I}) \cap \dcl^{\vee}(\A \mathbf{d}_{\ell} \mathbf{d}_{\J})$. By \cref{rem: finite character aclV closure} if $e \in \dcl^{\vee}(\A \mathbf{d}_{\ell} \mathbf{d}_{\I}) \cap \dcl^{\vee}(\A \mathbf{d}_{\ell} \mathbf{d}_{\J})$ there are finite sets $\B_{0} \subseteq \A \mathbf{d}_{\ell} \mathbf{d}_{\I}$ and $\B_{1} \subseteq \A \mathbf{d}_{\ell} \mathbf{d}_{\J}$ such that $e \in  \dcl^{\vee}(\B_{0}) \cap \dcl^{\vee}(\B_{1})$. Choose $n$ large enough such that $\B_{0} \subseteq \A (b_{i})_{i \in \{\ell\} \times [-n, n]} \mathbf{d}_{\I}$ and $\B_{1} \subseteq \A (b_{i})_{i \in \{ \ell\} \times [-n, n]} \mathbf{d}_{\J}$. Let $\I^{*}=\{ i \in \mathbb{Z} \times \mathbb{Z} \ | \ b_{i} \in \B_{0} \backslash \mathbf{d}_{\ell}\}$, and $i^{*}=\min \I^{*}$. By indiscernibility, we can replace $b_{i^{*}}$ by any $b_{j}$ with $(1,n)<j < \I$. So we can replace $b_{i^{*}}$ by an element in $\mathbf{d}_{\ell}$. Then we iterate this construction to move all of $b_{\I^{*}}$ inside $\mathbf{d}_{\ell}$ so $e \in \dcl^{\vee}(\A \mathbf{d}_{\ell})$.
\end{proof}
\begin{remark}\label{rem: equivalence for usual dcl and acl} A similar statement as in \cref{lem: equivalence indipendence dcl and acl} holds for the usual $\dcl$ and $\acl$-closure. In particular the last part of the statement holds also for $\acl$-independence i.e. $\mathbf{d}_{\I} \downfree_{\A \mathbf{d}_{n}}^{\acl} \mathbf{d}_{\J}$ for $n < \I < \J$.  The same proof in \cref{lem: equivalence indipendence dcl and acl}  works replacing $\dcl^{\vee}$ and $\acl^{\vee}$ by $\dcl$ and $\acl$. 
\end{remark}
\begin{definition} Let $\p$ be a global type and assume that $\p^{\otimes n}$ is generically stable over $\A$ for all $n<\omega$. We write $\downfree^{\p, \vee}$ to indicate the independence notion given by  $\downfree^{\p} \wedge \downfree^{\acl^{\vee}}$.
\end{definition}

\begin{proposition}\label{prop: GS independence is stronger} Let $\A$ be a small set of parameters, $a,b$ finite set of tuples in $\M$ and $\p$  a global type so that $\p^{\otimes n}$ is generically stable over $\A$ for all $n <\omega$. If $a \downfree_{\A}^{\GS} b$ then $a \downfree_{\A}^{\acl^{\vee}} b$ and $a \downfree_{\A}^{\p} b$. 
\end{proposition}
\begin{proof}
By assumption if $a \downfree_{\A}^{\GS} b$, then $b \models \pi \rest_{\A a}$ for every partial type $\pi$ generically stable stable over $\A$ consistent with $\tp(b/\A)$.  We need to show that each of $\downfree^{p}$ and $\downfree^{\acl^{\vee}}$ are given by a generically stable partial type. For $\downfree^p$, this is the last statement in Proposition \ref{prop: basic facts on p rel}.

Let us now consider $\downfree_{\A}^{\acl^{\vee}}$. Let $\pi^{\vee}(x)$ be the partial type defined by taking the closure under logical implication of \[\tp_x(b/\A) \cup \bigcup_{\mathrm{G}, \phi} \{\neg (\exists c) \phi(c,x) \wedge \dist_{\mathrm{G}}(c,d) > n:d \in \mathrm{M}\},\]
where the pair $(\mathrm{G},\phi)$ ranges over all $\A$-definable graphs $\G$ and formulas $\phi$ such that $\phi(\M,b)$ does not intersect any connected component of $\mathrm{G}$ that is $\vee$-algebraic over $\A$. This partial type is ind-definable over $\A$.

We show that for any tuple $b_* \equiv_{\A} b$ and $\B$ a set of parameters we have the equivalences:
\[b_*\downfree_\A^{\acl^\vee} \B \iff \B\downfree_\A^{\acl^\vee} b_* \iff b_*\models \pi^{\vee}(x) \rest_{\B}.\]

Indeed, if $b_*\models \pi^{\vee}(x) \rest_{\B}$, then there is $b' \equiv_{\B} a$ such that $b' \models \pi^{\vee}$ (over the monster model $\M$) and $b' \downfree_{\A}^{\acl^\vee} \M$ by definition of $\pi^{\vee}$. Conversely, assume $b_* \downfree_{\A}^{\acl^\vee} \B$ and $b_*\equiv_{\A} b$. Take a model $\M_0 \supseteq \B$. By extension (\emph{c.f.} \cref{rem: properties aclV independence}.(1)), there exists $b' \equiv_{\B} b_*$ with $b' \downfree_{\A}^{\acl^\vee} \M_0$. Then $b'$ satisfies
\[\tp_x(b/\A) \cup \bigcup_{\mathrm{G}, \phi} \{\neg (\exists c) \phi(c,x) \wedge \dist_{\mathrm{G}}(c,d) > n:d \in \mathrm{M_0}\}.\]
Therefore $b_*$ satisfies all formulas over $\B$ that are implied by that partial type. Since this is true for all $\M_0$, $b_*$ satisfies all the implications of $\pi^{\vee}(x)$ over $\B$, that is $b_*$ satisfies $\pi^{\vee}(x) \rest_{\B}$.\\

It is therefore sufficient to show that $\pi^{\vee}$ is generically stable. Let $d$ be a tuple in $\M$ and $(b_k)_{k<\omega}$ be a $\pi^{\vee}$-Morley sequence over $\A$ which is indiscernible over $\A d$. For any $\A$-definable graph $\mathrm{G}$ and $\A b$-definable set $X_b$ such that $X_b$ does not intersect a component of $\mathrm{G}$ that is $\vee$-algebraic over $\A$, the elements $X_{b_k}$ intersect distinct connected components of $\mathrm{G}$. By indiscernibility, they cannot intersect a connected component of $\mathrm{G}$ algebraic over $\A d$. Hence $b_k \models \pi^{\vee} \rest_{\A d}$ for each $k$, which shows that $\pi^{\vee}$ is generically stable.
\end{proof}

\begin{corollary}\label{cor: total sequence} Given $b, \A$ and an ordinal $\kappa$, there is an $\A$-indiscernible sequence $(b'_{i})_{i < \kappa}$ such that:
%$(b_{i})_{i < \omega}$ in  $\tp(b/\A)$ such that for any $\I < \J \subseteq \omega$, $b_{\I} \downfree_{\A}^{\p, \vee} b_{\J}$. In particular, we may assume that we can find $(b_{i})_{i \in |\T(\A)|^{+}}$ $\A$-indiscernible such that:
\begin{itemize}
\item $b$ is a subtuple of $b'_0$;
\item $b'_{\I} \downfree_{\A}^{\p} b'_{\J}$ for all sets $ \I < \J$;
\item $b'_{\I} \downfree_{\A}^{\acl^{\vee}} b'_{\J}$ for all sets $\I \cap \J=\emptyset$;
\item  the sequence $(b'_{i})_{i>0}$ is $\acl^{\vee}$ and $\acl$-independent over $\A b'_{0}$. 
\end{itemize}
\end{corollary}
\begin{proof}
By \cref{prop: total GS sequences exist} we can find an $\A$-indiscernible sequence $(b_{i})_{i \in \omega}$ in $\tp(b/\A)$ such that $b_{\J} \downfree_{\A}^{\GS} b_{\I}$ for all $\I > \J$. By \cref{prop: GS independence is stronger} $b_{\I} \downfree_{\A}^{\p, \vee} b_{\J}$. By \cref{lem: equivalence indipendence dcl and acl} and Ramsey we can replace the sequence by one that satisfies the required conditions.  
\end{proof}

\begin{lemma}\label{lem: staying indiscernible} Let $(b_{i})_{i \in \K}$ be an indiscernible and $\acl$-independent sequence over $\A$; and let $\p$ be a global $\A$-invariant generically stable type. Let $a \models \p\rest_{\A (b_{i})_{i \in \K}}$, then $(ab_{i})_{i \in \K}$ is an $\A a$-indiscernible sequence and is $\acl$-independent over $\A a$. 
\end{lemma}
\begin{proof}
By \cref{rem: equivalence for usual dcl and acl} it is enough to show that $(a b_{i})_{i \in \K}$ is $\dcl$-independent over $\A$. Let $I<I'<J$ subsets of $\K$ with $I$ of same order type as $I'$ and let $x \in \dcl(\A a b_{I}) \cap \dcl(\A a b_{J})$. Then there are $\A$-definable functions $\f$ and $\g$ such that $x=\f(a,b_{I})=\g(a, b_{J})$. By indiscernibility over $\A a$,  $\f(a, b_{I'})=\g(a, b_{J})$ and consequently $x= \f(a, b_{I})=\f(a, b_{I'})$. We write $\f_{y}(x)$ to denote the function $\f(x,y)$, then $\f_{b_{I}}(a)=\f_{b_{I'}}(a)$ and $e=[\f_{b_{I}}]_{\p}=[\f_{b_{I'}}]_{\p}$. In particular $e \in \dcl(\A b_{I}) \cap \dcl(\A b_{I'}) \subseteq \acl(\A)$ since $(b_{i})_{i\in \K}$ is $\acl$-independent over $\A$.  By \cref{stronggerms} the germ $e$ is strong over $\A$, so $x=\f_{b_{I}}(a) \in \dcl(\A a e) \subseteq \acl(\A a)$, as required. 
 \end{proof}
\subsection{Stable embeddedness and internality}
Recall that we assume that $\T$ eliminates imaginaries.

\begin{definition}

 Let $\A$ be a small subset of $\M$. Let $\X$ and $\Lat$  be definable sets, and let $\mathcal{X}
= (\X_{i})_{i \in I}$ and $\cY = (\Y_{j})_{j \in J}$ be (small) families of
$\A$-definable sets in $\M$. We often identify (small) families of definable sets \((X_i)_{i\in I}\) with the
associated ind-definable set \(\bigcup_i X_i\).
\begin{enumerate}
\item $\X$ is said to be $\Lat$-internal if $\X= \emptyset$ or there are $m \in \mathbb{N}$ and a surjective $\M$-definable function $\g \colon \ \Lat^{m} \rightarrow \X$.
\item Suppose $\Lat$ is an $\emptyset$-definable set, we write $\Int(\Lat, \A)$ to denote the union of the $\A$-definable sets internal to $\Lat$. 
\item An \(\A\)-definable subset of \(\cX\) is an \(\A\)-definable subset of
some finite product $\prod_{\ell\leq n} \X_{i_{\ell}}$, where $i_{\ell} \in I$. The family
\(\cX\) is \emph{stably embedded over $\A$} if any \(\M\)-definable subset
of \(\cX\) is definable with parameters from $\A \cup \bigcup_{i \in I}
\X_{i}(\M)$.\\
In that case, we denote by \(\cX^{eq}\) the family of all quotients of
\(\A\)-definable subsets of \(\cX\) by \(\A\)-definable equivalence relations.
\item The family \(\cX\) is \emph{$\cY$-internal} if there exists an
(ind-)\(\M\)-definable subset \(\mathrm{Z}\) of \(\cY\) and a surjective
(ind-)$M$-definable function $\g : \mathrm{Z} \rightarrow \cX$.
\end{enumerate}
\end{definition}
When \(\cX\) is stably embedded over \(\A\), we will often consider it as a
structure whose sorts are the \(\X_i\) with the full \(\A\)-induced structure on
products of sorts.\\
The following are folklore results (e.g. their proof can be find in \cite[Lemma $2.2.(2)$ and Lemma $2.6$]{resdom}).

\begin{lemma}\label{lem: facts stably embedded} Let $\A$ be a small set of parameters in $\M$ and let $\mathcal{X}=(\X_{i})_{i \in \I}$ be a (small ) family of $\A$-definable sets in $\M$. 
Assume that $\cX$ is stably embedded over $\A$. 
\begin{enumerate}
\item   Let $a$ be a tuple in $\mathcal{X}$ and let $b \in \M$. Then
$\tp(a/\A \mathcal{X}^{eq}(\A b)) \vdash \tp(a/\A b)$ and
$\tp(b/\A)\cup\tp(\mathcal{X}^{eq}(\A b)/\A a) \vdash \tp(b/ \A a)$.
\item Let \(c\) be a tuple of elements from \(\mathcal{X}\), let $d\in \M$ and let
$\B\subseteq \M$ that contains $\A$. Then 
\begin{equation*}
c \downfree_{\B} d \text{ if and only if } c \downfree_{\A \mathcal{X}^{eq}(\B)} \mathcal{X}^{eq}(\B d)  
\end{equation*}
where $\downfree$ denotes forking independence.\\
In particular, if $\mathcal{X}$ eliminates imaginaries, we have
\begin{equation*}
c \downfree_{\B} d \text{ if and only if } c \downfree_{\A \mathcal{X}(\B)} \mathcal{X}(\B d) 
\end{equation*}
\end{enumerate}
\end{lemma}

\begin{lemma}\label{lem: internal $p$-germs}  Let $\p$ be an $\A$-definable generically stable type, and $\Lat$ an $\A$-definable set. 
\begin{enumerate}
\item Let $\f_{b}$ be a definable family of functions to $\Lat$. Then the set of $\p$-germs of instances of $\f_{b}$ is $\Lat$-internal. 
 \item If $\Lat$ is stably embedded, then so is $\Int(\Lat, \A)$.  
\end{enumerate}
\end{lemma}
\begin{proof}
This is \cite[Lemma $2.1$ and $2.2$]{internal}.
\end{proof}
\subsection{Stable part of a structure and domination}
%\begin{definition}Let $\A$ be a small subset of $\M$. 
 %A (small) family $(\X_{i})_{i \in \I}$ of $\A$-definable sets in $\M$ is \emph{stably embedded} if for any finite sequence $(i_{1},\dots, i_{n})$ in $\I$, any $\M$-definable subset of $\displaystyle{\prod_{j=1}^{n} \X_{i_{j}}}$ is definable with parameters from $\displaystyle{\A \cup \bigcup_{i \in \I} \X_{i}}$.
%\end{definition}
\begin{definition}\label{def: stable part of the structure} Let $\A$ be a small set of parameters. We denote by $\St_{\A}$  the multi-sorted structure $(\D_{i}, \R_{j})_{i \in \I, j \in \J}$ whose sorts $\D_{i}$ are the $\A$-definable, stable, stably embedded subsets of $\M$. For each finite set of sorts $\D_i$, all the $\A$-definable relations on their finite product are included as $\emptyset$-definable relations $\R_{j}$. 
\end{definition}
By \cite[Lemma $3.2$]{HHM}, the structure $\St_{\A}$ is stable and stably embedded.
In \cite{HHM} is introduced the notion of a stably dominated type, we refer the reader to \cite[Section $2.2$]{tame} for a detailed treatment on pro-definable maps and stable domination.
\begin{notation} Let $\A$ be a small set of parameters, we write $\St_{\A}(\A; b)$ to denote $\St_{\A} \cap \dcl(\A; b)$. This is different from the notation in \cite{HHM}, where the $\A$ is implicit. They write $\St_{\A}(b)$ to denote $\St_{\A} \cap \dcl(\A b)$. As the parameter set $\A$ for the base will play an important role, we prefer the first notation to make some of the arguments more explicit. 
\end{notation}
\begin{definition}\label{def: stably dominated types}
\begin{enumerate}
\item Let $\q$ be a type over $\A$ and $\f=(\f_{i})_{i} \colon \q \rightarrow \St_{\A}$ a pro-$\A$-definable map. The type $\q$ is said to be \emph{stably dominated via $\f$} if for any tuple $b \in \M$ whenever $\St_{\A}(\A; b) \downfree^{f}_{\A} \f(a)$ then $\tp(b/\A \f(a))\vdash \tp(b/ \A a)$.
\item A global type $\p$ is said to be \emph{stably dominated over $\A$} if $\p\upharpoonright_{\A}$ is stably dominated. 
\item Let $\p$ be a global $\A$-invariant type and $b$ be a tuple in $\M$.  Let
$\f_{b}$ be an $\A b$-definable function. We say that \emph{$\p$ is dominated by $\f_{b}$ over $\A b$} if for any $a \models \p\rest_{\A b}$ and $d \in M$ whenever $\f_{b}(a) \models \f_{b}(\p) \rest_{\A b d}$ then $a \models \p\rest_{\A b d}$. 
\end{enumerate}
\end{definition}
\begin{remark} Note that definition $(3)$ could be stated in the following way: if $\p$ is a global $\A$-invariant type and $\f$ is an $\A$-definable function we say that $\p$ is dominated by $\f$ over $\A$ if for any $a \models \p \rest_{\A}$ and $d \in \M$ whenever $\f(a) \models \f(\p)\rest_{\A d}$ then $a \models \p \rest_{\A d}$. However, through the paper we will be interested in working with a global type $\p$ that is $\A$-invariant and that is dominated by an $\A b$-definable function $\f_{b}$, thus we find more precise to keep the formulation as stated in \cref{def: stably dominated types}.(3). 
\end{remark}
The following are \cite[ Corollary 3.31.(iii) and Proposition 3.13]{HHM}.  
\begin{proposition}\label{prop: basic facts stable domination} For all $a\in \M$ and $\A$ a small set of parameters of $\M$:
\begin{enumerate}
\item  $\tp(a/ \A)$ is stably dominated if and only if $\tp(a/ \acl(\A))$ is stably dominated; 
\item If $\A=\acl(\A)$ and $\tp(a/\A)$ is stably dominated via $\f$, then $\tp(a/\A)$ has a unique $\A$-definable global extension $\p$. Moreover, for all $\A \subseteq \B$, $a \models \p \rest_{\B}$ if and only if $\St_{\A}(\B) \downfree_{\A}^{f} \f(a)$.
\end{enumerate}
\end{proposition}
The following is \cite[Proposition $4.1$]{HHM}. This is the easy direction of change of base that allows one to always increase the base for stable domination and its proof does not use descent.
\begin{proposition}\label{prop: change of base going up} Let $\p$ be a global $\A$-invariant type and assume $\p$ is stably dominated over $\A$. Let $\A \subseteq \B$, then $\p$ is stably dominated over $\B$. 
 \end{proposition}
It is well known that in $\NIP$ theories, if $\p$ is generically stable so it is $\p^{\tensor n}$. In the general case this is still an open question. In \cite[Example $1.7$]{pillaygeneric} Adler, Casanovas and Pillay incorrectly provide an example of a type $\p$ that is generically stable while $p^{2}=\p\otimes \p$ is not. However, later in \cite[Remark $8.6$]{randomizations}, G. Conant, K. Gannon and J. Hanson argue that the type $\p$ given in \cite[Example $1.7$]{pillaygeneric} is not well defined. And  in \cite[Corollary $4.13$]{randomizations} they  proved that if the theory is $\mathrm{NTP}_{2}$ the product of two generically stable types is still generically stable. The general case is still open but we emphasize that it holds for stably dominated types. 
 \begin{fact}\label{fact: stable domination implies generically stable for all n} Let $\p$ be a global $\A$-invariant type that is stably dominated over $\A b$. Then $\p^{\tensor n}$ is generically stable over $\A$ for all $n< \omega$. 
\end{fact}
\begin{proof}
By Proposition \ref{prop: change of base going up} $\p$ is stably dominated over any $\B \supseteq \A b$, and by \cite[Proposition $6.11$]{HHM}\footnote{The reader might be concerned by the presence of some circularity in Section $4$, as this result is in a later chapter than the proof of descent. However, this result does not make any use of descent. In any case we first reprove descent for the invariant case where this fact is not used.}
 $\p^{\tensor n}\rest_{\B}$ is stably dominated over any $\B \supseteq \A b$ for every $n$. Consequently, $\p^{\otimes n}$ is generically stable over $\A b$ for all $n$. 
 Since $\p^{\otimes n}$ is $\A$-invariant, then it is generically stable over $\A$. (\emph{c.f.} \cref{fact: first things about generically stable}.(3)). 
\end{proof}
\begin{lemma}\label{lem: acl implies Morley} Assume $\T$ is stable. Let $(a_{i})_{i \in \I}$ be an infinite indiscernible sequence set over $\B$, and suppose that $(a_{i})_{i\in \I}$ is $\acl$-independent over $\B$, i.e. for any finite disjoint sets $\J, \J' \subseteq \I$, $\acl(\B a_{\J}) \cap \acl(\B a_{\J'}) \subseteq \acl(\B)$. Then $(a_{i})_{i \in \I}$ is a Morley sequence over $\B$. 
\end{lemma}
\begin{proof}
This is \cite[Lemma $2.6$]{HHM}. 
\end{proof}
The proof of the next statement follows very closely the argument in \cite[Proposition $3.16$]{HHM}, but we need a slight refined version. We include the proof for sake of completeness.
\begin{lemma}\label{lem: acl-independence implies Morley in St} Let $\A \subseteq \B$ be small sets of parameters. Let $(a_{i})_{i \in \I}$ be an infinite indiscernible sequence over $\B$ that is $\acl$-independent over $\B$, i.e. for any $\J,\J' \subseteq \I$ finite sets that are disjoint $\acl(\B a_{\J}) \cap \acl(\B a_{\J'}) \subseteq \acl(\B)$. Then $\St_{\A}(\A; \B a_{i})_{i \in \I}$ is a Morley sequence in the stable structure $\St_{\A}$ over $ \St_{\A}(\A,\B)$ (equivalently over $\B$). 
\end{lemma}
\begin{proof}
For each $i \in \I$ let $c_{i}= \St_{\A}(\A; \B a_{i})$. Since $(a_{i})_{i \in \I}$ is $\B$-indiscernible and $c_{i} \subseteq \dcl(\B a_{i})$ then $(c_{i})_{i \in \I}$ is also $\B$-indiscernible, in particular it is $\St_{\A}(\A;\B)$-indiscernible since $\St_{\A}(\A;\B) \subseteq \dcl(\B)$. \\
\underline{Claim $1$:} \emph{The sequence $(c_{i})_{i \in \I}$ is $\dcl$-independent over $\B$.}\\
\textit{Proof} %By \cref{rem: equivalence for usual dcl and acl} it is sufficient to show that the sequence $(c_{i})_{i \in \I}$ is $\dcl$-independent over $\B$, i.e. $\dcl(\B c_{\J}) \cap \dcl(\B c_{\J'}) \subseteq \dcl(\B)$ for any disjoint pair of finite subsets $\J, \J' \subseteq \I$.\\ 
Let $\J, \J'$ be two finite disjoint subsets of $\I$. By \cref{rem: equivalence for usual dcl and acl}, the sequence $(a_{i})_{i \in \I}$ is $\dcl$-independent over $\B$, and since  $c_{i} \in \dcl(\B a_{i})$ for all $i \in \I$ we have:
 \begin{equation*}
 \dcl(\B c_{\J}) \cap \dcl(\B c_{\J'}) \subseteq \dcl(\B a_{\J}) \cap \dcl(\B a_{\J'}) \subseteq \acl(\B). \end{equation*}
 This concludes the proof of the claim.\\
\underline{Claim $2$:} \emph{The sequence $(c_{i})_{i \in \I}$ is $\dcl$-independent over $\St_{\A}(\A;\acl(\B))$.}\\
Given $\J$ and $\J'$ two disjoint finite subsets of $\I$, by the first claim $\displaystyle{\dcl(c_{\J}\B) \cap \dcl(c_{\J'}\B) \subseteq \acl(\B) }$, by intersecting with $\St_{\A}$ this implies that:
\begin{equation}\label{eqpat}
\underbrace{\big(\St_{\A} \cap \dcl(c_{\J}\B)\big)}_{\St_{\A}(\A; \B c_{\J})} \cap \underbrace{\big(\St_{\A} \cap \dcl(c_{\J'}\B)\big)}_{\St_{\A}(\A; \B c_{\J'})}\subseteq  \underbrace{\St_{\A} \cap \acl(\B)}_{\St_{\A}(\A;\acl(\B))}.  
\end{equation}
Since $\St_{\A} \cap \dcl(c_{\J}; \St_{\A}(\A;\B)) \subseteq \St_{\A}(\A; \B c_{\J})$ and similarly for $\J'$, then \cref{eqpat} implies the statement of the claim.

\smallskip
By \cref{rem: equivalence for usual dcl and acl} together with the second claim we conclude that $(c_{i})_{i \in \I}$ is $\acl$-independent over $ \St_{\A}(\A;\acl(\B))$. By \cref{lem: acl implies Morley}, $(c_{i})_{i \in \I}$ is a Morley sequence in the stable structure $\St_{\A}$ over $\St_{\A}(\A;\acl(\B))$. Since $\St_{\A}$ eliminates imaginaries in particular it codes finite set, so $\St_{\A}(\A; \acl(\B)) \subseteq \acl_{\St_{\A}}(\St_{\A}(\A;\B))$. Since forking independence preserves algebraic closure (\emph{c.f.} \cite[Proposition 2.12]{3surprising}) then $(c_{i})_{i \in \I}$ is a Morley sequence over $\St_{\A}(\A;\B)$. 
The last part of the statement follows by \cref{lem: facts stably embedded}.

\end{proof}

%The following is \cite[Proposition $3.16$]{HHM}.
%\begin{lemma}\label{lem: acl-independence implies Morley in St} Let $\A$ be small sets of parameters. 
%Let $(a_{i})_{i \in \I}$ be an infinite indiscernible sequence over $\A$ that is $\acl$-independent over $\A$, i.e. for any $\J,\J' \subseteq \I$ finite sets that are disjoint $\acl(\A a_{\J}) \cap \acl(\A a_{\J'}) \subseteq \acl(\A)$. Then $\St_{\A}(\A; a_{i})_{i \in \I}$ is a Morley sequence in $\St_{\A}$. 
%\end{lemma}

\subsection{Facts on stable and $\NIP$ theories}
We briefly summarize a number of statements that we will use on stable theories, we make occasional use of the notion of weight.  
\begin{definition} Let $\T$ be stable. The \emph{pre-weight} of a type $\p=\tp(a/\A)$, denoted as $\mathrm{pre-wt(\p)}$ is the supremum of the set of cardinals $\kappa$ for which there is an $\A$-independent set $\{ b_{i} \ | \ i < \kappa\}$ such that $a \not \downfree^{f}_{\A} b_{i}$ for all $i$. The \emph{weight} $\mathrm{wt(\p)}$ is the supremum of $\{  \mathrm{pre-wt(\q)} \ | \ \q$ is a non-forking extension of $\p\}$. 
\end{definition}

\begin{lemma}\label{lem: weight stable} In a stable theory, any type has weight bounded by the cardinality of the language. In a superstable theory, any type in finitely many variables has finite weight.
\end{lemma}
 If $\T$ is stable and $\bar{b}=(b_{i})_{i \in \I}$ is an $\A$-indiscernible sequence then $\q=\Av(\bar{b}/ \M)$ is a complete type and is definable over $\A b_{\J}$ for any $\J \subseteq \I$ infinite. 
\begin{definition}\label{def: canonical base sequence} Let $\T$ be a stable theory and let $\bar b=(b_{i})_{i \in \I}$ be an $\A$-indiscernible sequence. We write $\Cb(\bar b)$ to denote the canonical base of $\bar b$ i.e. $\displaystyle{\bigcap_{\I_{0} \subseteq_{\text{infinite}} \I} \dcl(\A b_{\I_{0}})}$. 
\end{definition}
\begin{fact}\label{fact: facts canonical base sequence} Let $\T$ be a stable theory and let $\bar b=(b_{i})_{i \in \I}$ be an $\A$-indiscernible sequence where $|\I| > |\T(\A)|$.
\begin{enumerate}
\item Let $|\J| \geq |\T(\A)|$ and $\bar b'= (b_{i}')_{i \in \J}$ be such that $\bar b+\bar b'$ is an $\A$-indiscernible sequence, then $\Cb(\bar b)\subseteq \Cb( \bar b')$. In particular, $\Cb(\bar b+ \bar b')= \Cb(\bar b)$. Likewise, $\Cb(\bar b') \subseteq \Cb(\bar b)$ and $\Cb(\bar b+ \bar b')=\Cb(b')$. 
\item Let $d$ be a finite tuple in $\Cb(\bar b)$ then there is some $k < \omega$ such that for any $\J \subseteq \I$ of size $k$ we have $d \in \dcl(\A b_{\J})$. 
\end{enumerate}
\end{fact}
\begin{proof}
Both results are straightforward from the definition. 
\end{proof}
\begin{proposition}\label{shelah} Assume $\T$ is $\NIP$. Let $\A \subseteq \B$ and $\p \in \mathcal{S}^{m}(\A)$. Then there is $\C \subseteq \B$, $|\C| \leq |\T|$ and $\q \in \mathcal{S}^{m}(\B)$ such that $\p \subseteq \q$ and $\q$ does not split over $\A \cup \C$. 
\end{proposition}
\begin{proof}
This is \cite[Chapter~III, Theorem~7.5]{Shelah}. 
\end{proof}

\section{The bounded stabilizing property}
\textbf{Assumption:} Recall that we work under the hypothesis that $\T$ eliminates imaginaries.   

\begin{definition}\label{def: BS} 
Let $\kappa$ be an infinite cardinal, $\kappa \geq |\T|^{+}$.
\begin{enumerate}
\item We say that $\T$ has the \emph{ $\kappa$-bounded stabilizing property} ($\BS$) if for every $\A \subseteq \M$ small set of parameters and finite tuple $a \in \M$, there is no strictly increasing chain of definably closed sets between $\dcl(\A)$ and $\dcl(\A a)$ of length $\kappa^{+}$. Equivalently, there is no sequence $(b_{i})_{i \in \kappa^{+}}$ such that $b_{i} \in \dcl(\A a) \backslash \dcl(\A b_{<i})$. If $\kappa=|\T|$ we say that $\T$ has $\BS$.  
\item Likewise, we say that $\T$ has the \emph{$\kappa$- algebraic stabilizing property} ($\ABS$) if given $\A \subseteq \M$ a small set of parameters and a finite tuple $a \in \M$ there is no strictly increasing chain of algebraically closed sets of length $\kappa^{+}$ between $\acl(\A)$ and $\acl(\A a)$. Equivalently, there is no sequence $(b_{i})_{i \in \kappa^{+}}$ such that $b_{i} \in \acl(\A a) \backslash \acl(\A b_{<i})$. If $\kappa=|\T|$ we say that $\T$ has $\ABS$. 
\item  Let $\Lat$ be a $\C$-definable stably embedded set over $\C$. Then $\Lat$ has $\BS$ if for any tuple $a \in \M$ and small set of parameters $\C \subseteq \A$ there is no strictly increasing chain of definably closed sets of length $|\T(\C)|^{+}$ between $\Lat \cap \dcl(\A)$ and $\Lat \cap \dcl(\A a)$. Likewise, we say that it has $\ABS$ if there is no strictly increasing sequence of algebraically closed sets of length $|\T(\C)|^{+}$ between $\Lat \cap \acl(\A)$ and $\Lat \cap \acl(\A a)$. 
\end{enumerate}
\end{definition}
The main motivation of this section is to clarify and correct the following proposition that corresponds to \cite[Proposition $6.7$]{HHM}.
\begin{IS}\label{mistake} Let $\p$ be a global $\A$-definable type and assume that $\St_{\A}$ has $\BS$. Let $\f$  be a definable function on $\p(\M)$, the set of realizations of $\p \upharpoonright_{\A}$ and suppose that $\f(a) \in \St_{\A a}$ for all $a \in \p(\M)$. Then: 
\begin{itemize}
\item The germ $[\f]_{p}$ is strong over $\A$; 
\item $[\f]_{p} \in \St_{\A}$.
\end{itemize}
\end{IS}
The first part of this proposition is a generalization of Theorem \ref{stronggerms}. Unfortunately, the first result is false, while the second part remains true. The original proof of the second part used the strongness of the germ and the argument had significant gaps. We clarify the picture providing a counterexample to $(1)$ in Remark \ref{notstrong} and a proof of $(2)$ in  \cref{prop: correction}. We later use this lemma to give a simplified argument for an easier case of descent for stably dominated types. We also prove that $\BS$ holds for all stable sets in $\NIP$, which provides a simpler proof that the stable structure $\St_{\A}$ has the bounded stabilizing property $\BS$ (corresponds to \cite[Proposition $9.7$]{HHM}).\\
\begin{lemma}\label{lem: general dichotomy} Let $\A_{0}$ be a small set of parameters and $a$ be a finite tuple. 
Then exactly one of the following holds:
\begin{enumerate}
   \item There is a sequence $(d_{i})_{i \in |\T|^{+}}$ of elements $d_{i} \in \dcl(\A_{0} \mathbf{d}_{<i})$ where $\mathbf{d}_{j}$ denotes the finite set of conjugates of $d_{j}$ over $\A_{0}$. 
    \item There is no strictly increasing sequence of definably closed sets $(\mathrm{X}_{i})_{i \in |\T|^{+}}$ between $\dcl(\A_{0})$ and $\dcl(\A_{0} a)$ contained in $\acl(\A_{0})$.  
\end{enumerate}
\end{lemma}
\begin{proof}
We argue that either we can construct a sequence as in $(1)$ of length $|\T|^{+}$ or $(2)$ holds. We attempt to construct a sequence as in $(1)$ by transfinite induction. Take $d_{0} \in  \big(\acl(\A_{0}) \cap \dcl(\A_{0} a)\big)\backslash ( \dcl(\A_{0}))$. Let  $i < |\T|^{+}$ and assume the sequence $d_{<i} \subseteq \acl(\A_{0}) \cap \dcl(\A_{0} a)$ has been constructed, and let $\mathbf{d}_{<i}=(\mathbf{d}_{j})_{j< i}$ where each $\mathbf{d}_{j}$ is the finite set of conjugates of $d_{j}$ over $\A_{0}$. Let $\A_{0} \subseteq \B$  since $d_{j} \in \acl(\A_{0})$ then $d_{j} \in \acl(\B)$. In particular,  
\begin{equation}\label{isolated1}
\tp(\mathbf{d}_{j}/ \B) \ \text{is isolated and} \  \ulcorner \mathbf{d}_{j} \urcorner \in \dcl(\A_{0}) \subseteq \dcl(\B).
\end{equation}
Either there is some $d_{i} \in \big( \acl(\A_{0}) \cap \dcl(\A_{0} a) \big) \setminus \dcl(\A_{0} \mathbf{d}_{<i})$ or $ \dcl(\A_{0} \mathbf{d}_{<i})=\big( \dcl(\A_{0} a) \cap \acl(\A_{0})\big)$. We show that if $\dcl(\A_{0} \mathbf{d}_{<i})=(\dcl(\A_{0} a) \cap \acl(\A_{0}))$ then $(2)$ must hold. Let $(\X_{j})_{j \in |\T|^{+}}$ be a strictly increasing sequence of definably closed sets between $\dcl(\A_{0})$ and $ \dcl(\A_{0} a)$ contained in $\acl(\A_{0})$.\\

For each $i < |\T|^{+}$ and $z \in \mathbf{d}_{<i}$ the type $\displaystyle{\tp(z/\bigcup_{j \in |\T|^{+}} \X_{j})}$ is isolated ( by \cref{isolated1} which can be applied as $\A_{0} \subseteq \X_{i}$). Let $\displaystyle{z' \in \bigcup_{j \in |\T|^{+}} \X_{j}} \subseteq \acl(\A_{0})$ be such that $\displaystyle{\tp(z/ \A_{0}, z') \vdash \tp(z/  \bigcup_{j \in |\T|^{+}} \X_{j})}$. Let $\displaystyle{\Z'=\{ z' \ | \ z \in (\mathbf{d}_{<i})^{m} \ | \ m=1,2, \dots \}}$, then $\tp(\mathbf{d}_{<i}/ \A_{0} \Z') \vdash \tp(\mathbf{d}_{<i}/ \bigcup_{j \in |\T|^{+}} \X_{j})$. There is some $\mu \in |\T|^{+}$ such that $\Z' \subseteq \X_{\mu}$. This  as $|\Z'|=|[\mathbf{d}_{<i}]|^{<\omega}|= |\mathbf{d}_{<i}|=|\T|$, because $i < |\T|^{+}$.\\

We now argue that  the chain of $\X_{j}$ stabilizes for $j > \mu$. For each $j<|\T|^{+}$, we have  $\X_{j} \subseteq  \dcl(\A_{0} \mathbf{d}_{<i})$, as 
$ \X_{j} \subseteq  \big(\dcl(\A_{0} a) \cap \acl(\A_{0}) \big)=  \dcl(\A_{0} \mathbf{d}_{<i})$. Fix $j > \mu$
 and let  $y \in \X_{j}  \subseteq \dcl(\A_{0} \mathbf{d}_{<i})$, then there is some $\A_{0}$-definable function $\f$ and tuple $x \in \mathbf{d}_{<i}$ such that $\displaystyle{\f(x)=y \in \tp(x/\A_{0} \bigcup_{\ell \in |\T|^{+}} \X_{\ell})}$. That type is isolated by a formula in the type $\tp(x/ \A_{0} z')$ and therefore is $\A_{0} z'$-definable. Hence $y \in \dcl(\A_{0} z') \subseteq \X_{\mu}$ as $\X_{\mu}$ is definably closed. 
\end{proof}
We will need as well a similar version on the previous lemma but working inside $\Lat$ a $\C$-definable set and stably embedded over $\C$, and in this context we need to work with sequences of length $|\T(\C)|^{+}$.  The proof is essentially the same, but we include details for sake of completeness with the require modifications where needed.
\begin{lemma}\label{lem: stably embedded dichotomy} Let $\Lat$ be a $\C$-definable set  stably embedded  over $\C$. Let $a \in \M$ and $\A_{0} \supseteq \C$ a small set of parameters, then exactly one of the following holds:
\begin{enumerate}
\item there is a sequence $(d_{i})_{i \in |\T(C)|^{+}}$ of elements $d_{i} \in \Lat^{eq} \cap (\acl(\A_{0}) \cap \dcl(\A_{0} a))$ such that $d_{i} \notin \dcl(\A_{0} \mathbf{d}_{<i})$ for all $i <|\T(C)|^{+}$, where $\mathbf{d}_{j}$ denotes the finite set of conjugates of $d_{j}$ over $\A_{0}$, or
\item there is no strictly increasing sequence  of definably closed sets $(\X_{i})_{i \in |\T(\C)|^{+}}$ between $\Lat^{eq} \cap \dcl(\A_{0})$ and $\Lat^{eq} \cap  \dcl(\A_{0} a)$ contained in $\Lat^{eq} \cap \acl(\A_{0})$. 
\end{enumerate}
\end{lemma}
\begin{proof}
We argue that either we can construct a sequence as in $(1)$ of length $|\T(C)|^{+}$ or $(2)$ holds. We attempt to construct a sequence as in $(1)$ by transfinite induction. Take $d_{0} \in  \Lat^{eq} \cap (\acl(\A_{0}) \cap \dcl(\A_{0} a)) \backslash (\Lat^{eq} \cap \dcl(\A_{0}))$. Let  $i < |\T(C)|^{+}$ and assume the sequence $d_{<i} \subseteq  \Lat^{eq} \cap (\acl(\A_{0}) \cap \dcl(\A_{0} a))$ has been constructed, and let $\mathbf{d}_{<i}=(\mathbf{d}_{j})_{j< i}$ where each $\mathbf{d}_{j}$ is the finite set of conjugates in $\Lat^{eq}$ of $d_{j}$ over $\A_{0}$. Note that by stable embeddedness $\mathbf{d}_{j}$ are the conjugates in $\Lat^{eq}$ of $d_{j}$ over $\Lat^{eq} \cap \dcl(\A_{0})$.\\
By \cref{lem: facts stably embedded} $\tp(d_{j}/\C\Lat^{eq}(\acl(\A_{0})))\vdash \tp(d_{j}/ \acl(\A_{0}))$. Thus $\tp(d_{j}/ \C \Lat^{eq}(\acl(\A_{0})))$ is algebraic. Since $\Lat^{eq}$ codes finite sets, then $\Lat^{eq}(\acl(\A_{0})) \subseteq \acl( C \big(\Lat^{eq}\cap \dcl(\A_{0})\big))$, thus $\tp(d_{j}/ \acl(\C \big(\Lat^{eq} \cap \dcl (\A_{0})\big)$ is algebraic. Consequently,
\begin{equation}\label{eq: isolatedcorrected}
\text{for any} \  \C\Lat^{eq}(\A) \subseteq \B \subseteq \Lat^{eq} \ \text{the type 
$\tp(\mathbf{d}_{j}/ \B)$ is isolated.}
\end{equation}

Exactly one of the following cases hold:\\
a) we can find an element $d_{i} \in \Lat^{eq} \cap \big( \acl(\A_{0}) \cap \dcl(\A_{0} a) \big) \setminus \dcl(\A_{0} \mathbf{d}_{<i})$ and we keep building the sequence,\\
b) or $\Lat^{eq} \cap \dcl(\A_{0} \mathbf{d}_{<i})=\Lat^{eq} \cap \big( \dcl(\A_{0} a) \cap \acl(\A_{0})\big)$.\\

We show that if $\Lat^{eq} \cap \dcl(\A_{0} \mathbf{d}_{<i})= \Lat^{eq} \cap (\dcl(\A_{0} a) \cap \acl(\A_{0}))$ then $(2)$ must hold. \\
Let $(\X_{j})_{j \in |\T(\C)|^{+}}$ be a strictly increasing sequence of definably closed sets between $\Lat^{eq}\cap \dcl(\A_{0})$ and $\Lat^{eq} \cap \dcl(\A_{0} a)$ such that for each  $j \in |\T(\C)^{+}|$ we have $\X_{j} \subseteq \Lat^{eq} \cap \acl(\A_{0})$.\\
For each $j < |\T(\C)|^{+}$ and $z \in \mathbf{d}_{<i}$ the type $\displaystyle{\tp(z/\C \bigcup_{j \in |\T|^{+}} \X_{j})}$ is isolated (\cref{eq: isolatedcorrected})  and note that $\Lat^{eq}(\A_{0}) \subseteq \X_{i}$. Let $\displaystyle{z' \in \bigcup_{j \in |\T(\C)|^{+}} \X_{j}} \subseteq \Lat^{eq} \cap \acl(\A_{0})$ be such that $\displaystyle{\tp(z/ \A_{0}, z') \vdash \tp(z/ \C \bigcup_{j \in |\T(\C)|^{+}} \X_{j})}$. Let $\displaystyle{\Z'=\{ z' \ | \ z \in (\mathbf{d}_{<i})^{m} \ | \ m=1,2, \dots \}}$, then $\tp(\mathbf{d}_{<i}/ \A_{0} \Z') \vdash \tp(\mathbf{d}_{<i}/\C \bigcup_{j \in |\T(\C)|^{+}} \X_{j})$. There is some $\mu \in |\T(\C)|^{+}$ such that $\Z' \subseteq \X_{\mu}$ (since $|\Z^{'}|=|\mathbf{d}_{<i}|^{<\omega}=|\T(\C)|$ and by ). The argument to show that  the chain of $\X_{j}$ stabilizes for $j > \mu$ follows exactly as in \cref{lem: general dichotomy}.
\end{proof}
\begin{lemma}\label{lem: sequencestabilizes} Let $\T$ be $\mathrm{NIP}$, $\A_{0}$ be a small set of parameters and $a$ be a finite tuple. Let $(b_{i})_{i \in |\T|^{+}}$ be a sequence between $\dcl(\A_{0})$ and $\dcl(\A_{0}a)$ contained in $\acl(\A_{0})$. Let $\X_{i}=\dcl(\A_{0} b_{<i})$ then the chain stabilizes i.e. there is some $\mu \in |\T|^{+}$ such that for any $i > \mu$, $\X_{i}=\X_{\mu}$.
\end{lemma}
\begin{proof}
 We proceed by contradiction. By \cref{lem: stably embedded dichotomy} we may assume that we can find a sequence  $(d_{i})_{i \in |\T|^{+}} \subseteq \big( \acl(\A_{0}) \cap \dcl(\A_{0} a)\big)$ such that $d_{i} \notin \dcl(\A_{0} \mathbf{d}_{<i})$ for all $i <|\T|^{+}$ where $\mathbf{d}_{j}$ denotes the finite set of conjugates  of $d_{j}$ over $\A_{0}$.\\
By \cref{shelah} there is some $\eta< |\T|^{+}$ such that $\tp(a/\A_{0} (\mathbf{d}_{\ell})_{\ell \in |\T|^{+}})$ does not split over $\A \mathbf{d}_{<\eta}$. Let $z \in \A_{0} (\mathbf{d}_{\ell})_{\ell \in |\T|^{+}}$ and $\sigma \in \mathrm{Aut}(\M/ \A_{0}\mathbf{d}_{<\eta})$ then $\sigma(z) \in \A_{0} (\mathbf{d}_{\ell})_{\ell \in |\T|^{+}}$ and by non-splitting for any $\La(\A_{0}\mathbf{d}_{<\eta})$-formula $\phi(x,y)$ we have $\vDash \phi(a,z)$ if and only if $\vDash \phi(a, \sigma(z))$. Consequently,  $\tp(a/\A (\mathbf{d}_{\ell})_{\ell \in |\T|^{+}})$ is $\A \mathbf{d}_{<\eta}$-invariant.\\
This implies that $\tp(d_{\eta}/ \A_{0} \mathbf{d}_{<\eta}) \vdash \tp(d_{\eta}/ \A_{0} \mathbf{d}_{<\eta} a)$. For this, let $d' \equiv_{\A_{0} \mathbf{d}_{<\eta}} d_{\eta}$ if $d' \not\equiv_{\A_{0} \mathbf{d}_{<\eta} a}d_{\eta}$ then there is a $\La(\A_{0}\mathbf{d}_{<\eta})$-formula $\psi(x,y)$ such that $\vDash \psi(d', a)$ while $\vDash \neg \psi(d_{\eta}, a)$, but this is a contradiction as $\tp(a/ \A_{0} (\mathbf{d}_{\ell})_{\ell \in |\T|^{+}})$ is $\A_{0} \mathbf{d}_{\eta}$-invariant.\\
By hypothesis, $d_{\eta} \in \dcl(\A_{0} a)$, then $d_{\eta} \in \dcl(\A_{0} \mathbf{d_{<\eta}})$. This is a contradiction since $d_{\eta} \notin \dcl(\A_{0} \mathbf{d}_{<\eta})$, thus such sequence cannot exist and the statement must hold by \cref{lem: general dichotomy}. This concludes the proof of the Lemma.
\end{proof}
We will need a version of this result for $\Lat$ a $\C$-definable set stably embedded over $\C$. 
\begin{lemma}\label{lem: sequencestablyembedded} Let $\T$ be a $\mathrm{NIP}$ theory and let $\Lat$ be a $C$-definable and stably embedded set over $\C$. Let $\C \subseteq \A_{0}$ be a small set of parameters and $a$ a finite tuple. Then there is no strictly increasing sequence of definably closed sets $(\X_{i})_{i \in |\T(\C)|^{+}}$ between $\Lat^{eq} \cap \dcl(\A_{0})$ and $\Lat^{eq} \cap \dcl(\A_{0} a)$ contained in $\Lat^{eq}\cap \acl(\A_{0})$.
\end{lemma}
\begin{proof}

The proof is exactly the same as in \cref{lem: sequencestabilizes}, but we use  \ref{lem: stably embedded dichotomy} instead of \cref{lem: general dichotomy}, work with the cardinal $|\T(\C)|^{+}$ instead of $|\T|^{+}$ and we intersect everything with $\Lat^{eq}$ when definable closure or algebraic closure are considered.  Recall that any expansion by constants of a $\mathrm{NIP}$-theory is still $\mathrm{NIP}$. This allows us to work with the cardinal $|\T(\C)|^{+}$ instead of $|\T|$, when we apply \cref{shelah}.  
\end{proof}
\begin{proposition}\label{baby}
\begin{enumerate}
\item If $\T$ is rosy, then it has $\ABS$.
\item Assume $\T$ is $\NIP$ and let $\Lat$ be $\C$-definable set stably embedded over $\C$. If $\Lat^{eq}$ has $\ABS$, then it has $\BS$. In particular, if $\T$ is $\NIP$ and has $\ABS$ then it has $\BS$.
\item If $\T$ is rosy and $\NIP$, then it has $\BS$. In particular, any stable theory has $\BS$. 
\end{enumerate}
\end{proposition}
\begin{proof}
\begin{enumerate}
\item  Suppose there is a sequence $(b_{i})_{i \in |\T|^{+}} \subseteq \acl(\A a)$ such that  $b_{i} \notin \acl(\A b_{<i})$ but $b_{i} \in \acl(\A a)$, so $b_{i} \ndownfree_{\A b_{<i}}^{\text{\TH}} a$. By  local character of thorn forking (see \cite[Theorem $3.7$]{rosy}), $a \downfree_{\A b_{\leq|\T|}}^{\text{\TH}}b_{\leq |\T|^{+}}$. Let $i=|\T|+1$, by symmetry (\cite[Theorem $3.7$]{rosy}) $ b_{i}\downfree_{\A b_{\leq |\T^{+}|}}^{\text{\TH}}a$, which is a contradiction.
\item Let $\C \subseteq \A$ be a small set of parameters and $( b_{i})_{i \in |\T(\C)|^{+}}$ be a sequence between $\Lat^{eq} \cap \dcl(\A)$ and $\Lat^{eq} \cap \dcl(\A a)$. Let $\D_{i}=\Lat^{eq} \cap \acl(\A b_{\leq i})$ and $\B_{i}=\Lat^{eq} \cap \dcl(\A b_{\leq i})$, we want to show that the sequence $\B_{i}$ stabilizes. 
By $\ABS$ there is some $\mu \in |\T(\C)|^{+}$ such that for all $j \geq \mu$ we have $\D_{j}=\D_{\mu}$, consequently $(b_{i})_{i \in |\T(\C)|^{+}} \subseteq \acl(\A b_{< \mu})$.\\
Let $\A_{0}= \A b_{<\mu}$, by \cref{lem: sequencestablyembedded} there is no strictly increasing sequence of definably closed sets $(\X_{i})_{i \in |\T(\C)|^{+}}$ between $\Lat^{eq} \cap \dcl(\A_{0})$ and $\Lat^{eq} \cap  \dcl(\A_{0} a)$ contained in $\acl(\A_{0}) \cap \Lat^{eq}$.The sequence $(\B_{i})_{\mu<i< |\T(\C)|^{+}}$ is a sequence of definably closed sets between $\Lat^{eq} \cap \dcl(\A_{0})$ and $\Lat^{eq} \cap \dcl(\A_{0} a)$ contained in $\acl(\A_{0})\cap \ \Lat^{eq}$, by \cref{lem: sequencestablyembedded} such sequence stabilizes and we conclude that $\Lat^{eq}$ has $\mathrm{BS}$. 
\item Immediate from $(1)$ and $(2)$. \end{enumerate}\end{proof}

\begin{remark} Not every $\NIP$ theory has $\BS$, in fact $\ACVF$ does not have $\BS$. To see this, fix an element $a \in \M$ and  $\kappa > |\T|^{+}$. One can find $(\gamma_{i})_{i<\kappa}$ a strictly increasing sequence in the value group such that $ \sup\{ \gamma_{j} \ | j < i\} < \gamma_{i}$. Let $b_{i}$ be the closed ball centered at $a$ of radius $\gamma_{i}$, and $\A=(\gamma_{i})_{i < \kappa}$. Then $\dcl(\A) \subseteq (b_{i})_{i \in \kappa} \subseteq \dcl(\A a)$ and $b_{i} \notin \dcl(\A b_{<i})$. Consequently, $\T$ does not have $\BS$. 
\end{remark}

\begin{proposition}\label{prop: stable BS} Assume $\T$ is $\NIP$. Let $\Lat$ be a stable $\C$-definable set, then $\Lat^{eq}$ has $\BS$. 
\end{proposition}
\begin{proof} Recall that we assume that $\T$ eliminates imaginaries. By Proposition \ref{baby}.(2) it is sufficient to verify that $\Lat^{eq}$ has $\ABS$. We proceed by contradiction. Let $\A \supseteq \C$, $a \in \M$ and let $(b_{i})_{i \in |\T(\C)|^{+}}$ be a sequence  of elements of $\Lat^{eq} \cap \acl(\A a)$ such that $b_{i} \notin \Lat^{eq} \cap \acl(\A b_{<i})$. \\

\underline{Claim:} \emph{ There is a set of parameters $\A_0 \supseteq \A$ and a sequence $(b_{i}')_{i \in |\T(\C)|^{+}}$ such that:
\begin{enumerate}
\item for all $i < |\T(\C)|^{+}$ $b_{i}' \in \Lat^{eq} \cap \acl(\A_{0} a)$ and $b_{i}' \not \in \Lat^{eq} \cap \acl(\A_{0} b'_{<i})$, and
\item  for all $i < |\T(\C)|^{+}$, let $\p_{i}$ be  the global non-forking extension of $\q_{i}=\tp(b'_{i}/ \Lat^{eq} \cap \acl(\A_{0} b'_{<i}))$. Then $\mathrm{Cb}(\p_{i}) \subseteq \Lat^{eq}\cap \acl(\A_{0})$.  
\end{enumerate}
} 
\textit{Proof:} First we show that we can extend the sequence $(b_{i})_{i \in |\T(\C)|^{+}}$ to one of length $|\T(\C)|^{++}$. Indeed by the pigeonhole principle we can find a $\La$-formula $\phi(y,z,x)$, some $k \in \mathbb{N}$ and tuples $c_{i} \in \A$ such that $\M \vDash \phi(a,c_{i},b_{i})$ and $\phi(a,c_{i},\M)$ has exactly $k$-elements.\\
Add a predicate $P$ to distinguish $\A$ and consider the partial type:
\begin{align*}
\Sigma(y_{i})_{i<|\T(\C)|^{++}}&=\{ \exists c \in P (\phi(a,c, y_{i}) \wedge |\phi(a,c,\M)|=k) \}_{i <|\T(\C)|^{++}} \cup \{ y_{i} \in \Lat \}_{i< |\T(\C)|^{++}}\\
&\cup \{ \forall c \in P \big(|\psi(a,c,z, \M)|=n \rightarrow \neg \psi(a,c,z,y_{i})\big) \ | \ z \subseteq y_{<i}, n \in \mathbb{N}, \, \psi \text{ an $\La$-formula}\}_{i \in |\T(\C)|^{++}},
\end{align*}
which is consistent. Let $(M,A) \prec (N,P(N))$ be an elementary extension that contains a realization $(d_{i})_{i \in |\T(\C)|^{+}}$ of $\Sigma$. Then $d_{i} \notin \acl(P(N)d_{<i})$ and for every  $i< |\T(\C)|^{++}$, $d_{i} \in \acl(P(N)a)$. Replacing $A$ by $P(N)$ and considering $(d_{i})_{i \in |\T(\C)|^{++}}$, we may assume that the sequence has length $|\T(\C)|^{++}$.\\
Let $\p_{i}$ be the unique non-forking global extension of $\tp(b_{i}/\C\cup \big(\Lat^{eq} \cap \acl(\A b_{<i})\big))$. Note that  $\mathrm{Cb}(\p_{i}) \subseteq \C \cup \big(\Lat^{eq} \cap \acl(\A b_{<i})\big)$.

 If $\cof(i) \geq |\T(\C)|^{+}$ then there is some $j<i$ such that $\mathrm{Cb}(\p_{i}) \subseteq \C \cup \big(\Lat^{eq} \cap \acl(\A b_{<j})\big)$. (As the canonical base is the set of canonical parameters for each formula in the definition scheme of the type, and they are $|\T(\C)|$-many of those.)
 
 Let $\D=\{ i \in |\T(\C)|^{++} \ | \ \cof(i) \geq |\T(\C)|^{+}\}$ and $\f: \D \rightarrow |\T(\C)|^{++}$ be the regressive function defined by sending $i$ to $\min\{ j < i \ | \ \mathrm{Cb}(\p_{i}) \subseteq \C \cup \big(\Lat^{eq} \cap \acl(\A b_{<j})\big)\}$.
 
 By Fodor's Lemma there is some stationary set $\D_{0} \subseteq \D$ and an element $\mathbf{j} \in |\T(\C)|^{++}$ such that for all $i \in \D_{0}$, $\f(i)=\mathbf{j}$. 
 Let $( b_{i_{\ell}})_{\ell \in |\T(\C)|^{+}}$ be a subsequence of $(b_{i})_{i \in |\T(\C)|^{++}}$ where each $i_{\ell} \in \D_{0}$ and $i_{\ell}> \mathbf{j}$. Then for any $\ell \in |\T(\C)|^{+}$, we have that $\mathrm{Cb}(\p_{i_{\ell}}) \subseteq \acl(\A b_{< \mathbf{j}})$. Now take $\A_{0}=\acl(\A b_{< \mathbf{j}})$ and $(b_{i}')_{i \in |\T(\C)|^{+}}$ be the sub-sequence $(b_{i_{\ell}})_{\ell \in |\T(\C)|^{+}} \square_{\text{Claim}}$. \\

Let $\A_{0}$ and $(b'_{i})_{i \in |\T(\C)|^{+}}$ be as in the previous Claim and
$\B= \A_{0} (b'_{i})_{i \in |\T(\C)|^{+}}$. For each $i \in |\T(\C)|^{+}$ let $\I_{i}=(d_{i}^{j})_{j < |\T(\B)|^{+}}$ be a Morley sequence in $\p_{i}$ over 
$\B \I_{<i}$ i.e. $d_{i}^{j} \vDash \p_{i}\upharpoonright_{\B \I_{i} d_{i}^{<j}}$. Then $(b'_{i}+\I_{i} \ | \ i \in |\T(\C)|^{+})$ are mutually indiscernible over $\A_{0}$. The first element $b_{i}'$ of the sequence $(b'_{i}+ \I_{i})$ is algebraic over $\A_{0} a$. By \cref{fact: first things about generically stable}, there is $\mathrm{J} \subseteq \mathrm{I}$ such that $|\mathrm{I} \backslash \mathrm{J}| \leq |\T(\B)|$ and for every $j \in \J$ we have $d_{i}^{j} \models \p_{i} \rest_{\B a}$. As the elements $d_i^j$ are distinct as $j$ varies, we have $d_{i}^{j} \notin \acl(\A_{0} a)$. We conclude that  $(b_{i}' +\I_{i})$ is not $\A_{0}a$-indiscernible for each $i \in |\T(\C)|^{+}$. This is a contradiction to \cite[Proposition $4.8$]{simon} (and the fact that $\mathrm{NIP}$ is preserved after adding constants.)
\end{proof}

We now need to generalize this argument to $\Lat$ being a family of stable sorts, instead of one stable sort. This will be required to obtain a corrected version of the statement in \cite[Proposition $6.7$]{HHM}. It also generalizes the later argument in \cite[ Proposition 5.7]{AKE}, where it is shown that given an $\A$-definable type $\p$ in a henselian valued field of equicharacteristic zero with algebraically closed residue field, $a \models \p$  and  $\f$ a definable function with image in the \emph{linear structure} (see \cite[ Definition $3.17$]{AKE}) the germ of $\f$ on $\p$ lies in an $\A$-definable set that is almost internal to the residue field. 
\begin{proposition}\label{prop: correction} Assume $\T$ is $\NIP$. Let $\A$ be a small set of parameters. Let $\p$ be a global $\A$-definable type and let $\f_{b}$ be an
$\A$-definable family of functions such that, for all $b \in \M$ and $a \models \p \rest_{\A}$ we have $\f_{b}(a) \in \St_{\A a}(\M)$. Let $\Theta$ be the $\A$-definable set of all $[\f_{m}]_{\p}$, then $\Theta$ is stable and stably embedded. In particular, $[\f_{b}] \in \St_{\A}(\M)$. 
\end{proposition}
\begin{proof} 
Let $(a_\alpha)_{\alpha < |\T(\A)|^{++}}$ be a Morley sequence of $\p$ over $\A$. Let $c_\alpha = \f_b(a_\alpha) \in \St_{\A a_\alpha}$. This is an element in some stable stably embedded set $\mathrm{X}_{\alpha}$ definable over $\A a_\alpha$.  Let $\q_{\alpha}$ be the global non-forking extension of $\tp(c_\alpha /\A a_{\alpha} \X_{\alpha}^{eq}(\acl(\A a_{\leq \kappa} c_{<\alpha})))$.\\

\underline{Claim $1$:} \emph{There is a subsequence $(a_{\alpha_{\beta}})_{\beta \in |\T(\A)|^{++}}$ and an element $\mathbf{j} \in |\T(\A)|^{++}$ such that for every $\beta \in |\T(\A)|^{++}$ $\mathrm{Cb}(\q_{\alpha_{\beta}}) \subseteq \acl(\A (a_{< |\T(\A)|^{++}} c_{<\mathbf{j}}))$ .}\\
\textit{Proof:} This is similar to \cref{prop: stable BS}, we include details for sake of completeness.  For any $\alpha \in |\T(\A)|^{++}$ such that $\cof(\alpha)>|\T(\A)|$, there is $j<\alpha$ such that $\mathrm{Cb}(\q_\alpha) \subseteq \acl(\A a_{\leq \kappa} c_{<j})$ (we compute the canonical base inside the stable set $\X_{\alpha}^{eq}$ which is stably embedded over $\A a_{\alpha}$, so the defining scheme requires $|\T(\A)|$-many parameters).\\
As in the claim of \cref{prop: stable BS} we let $\D=\{ \alpha \in |\T(\A)|^{++} \ | \ \cof(\alpha) \geq |\T(\A)|^{+}\}$ and let $\f: \D \rightarrow |\T(\A)|^{++}$ be the regressive function defined by sending $\alpha$ to $\min\{ j < \alpha \ | \ \mathrm{Cb}(\q_{\alpha}) \subseteq  \acl(\A a_{\leq |\T(\A)|^{++}} c_{<j})\big)\}$. By Fodor's lemma, we find some $\mathbf{j} \in |\T(\A)|^{++}$ such that the set $C$ of $\alpha<|\T(\A)|^{++}$ for which $\mathrm{Cb}(\q_\alpha) \subseteq \acl(\A a_{<|\T(\A)|^{++}} c_{<\mathbf{j}})$ is cofinal in $|\T(\A)|^{++}$. Let $(a_{\alpha_{\beta}})_{\beta \in |\T(\A)|^{++}}$ be the subsequence of elements in $\C$. This concludes the proof of the first Claim.\\

\underline{Claim $2$:} \emph{ Let $\A_{0}=\acl(\A a_{<|\T(\A)|^{++}} c_{<\mathbf{j}})$ then for at most  $|\T(\A)|$-many $\alpha_{\beta}'s$ the type $\q_{\alpha_{\beta}} \rest_{\A_{0}}$ is non algebraic.}\\
\textit{Proof:} We proceed by contradiction and we assume that there is a subset $\Y \subseteq |\T(\A)|^{+}$ such that for any $\alpha_{\beta} \in \Y$ $\q_{\alpha_{\beta}} \rest_{\A_{0}}$ is not algebraic.  To simplify the notation, let $(a_{\delta})_{\delta \in |\T(\A)|^{+}}$ be a re-indexing of a subsequence of size $|\T(\A)|^{+}$ for the indices $\alpha_{\beta} \in \Y$. 
Let $\I_{\delta}=(d_{\delta}^{j})_{j < |\T(\A)|^{+}}$ be a Morley sequence in $\q_{\delta}$ over 
$\A_{0} \I_{<\delta}$ this is $d_{\delta}^{j} \vDash \q_{\delta}\upharpoonright_{\A_{0} \I_{<\delta} d_{\delta}^{<j}}$. Then $(c_{\delta}+\I_{\delta} \ | \ \delta \in |\T(\A)|^{+})$ are mutually indiscernible over $\A_{0}$ (since the types are stable and the base contains the canonical base). The first element $c_{\delta}$ of the sequence $(c_{\delta}+ \I_{\delta})$ is algebraic over $\A_{0} b$. As in \cref{prop: stable BS}, by \cref{fact: first things about generically stable}, there is $\mathrm{J} \subseteq |\T(\A)|^{+}$ such that $||\T(\A)|^{+} \backslash \mathrm{J}| \leq |\T(\A)|$ and for every $j \in \J$ we have $d_{\delta}^{j} \models \q_{\delta} \rest_{\A_{0} b}$. As the elements $d_\delta^j$ are distinct as $j$ varies, we have $d_\delta^{j} \notin \acl(\A_{0} b)$. We conclude that  $(c_\delta' +\I_\delta)$ is not $\A_{0}b$-indiscernible for each $\delta \in |\T(\A)|^{++}$. As in \cref{prop: stable BS} this is a contradiction to $\mathrm{NIP}$. This concludes the proof of the second Claim.\\

Restricting to a subsequence and re-indexing we have a sequence $(a_{\delta})_{\delta_{\in |\T(\A)|^{++}}}$ such that $\q_{\delta} \rest_{\A_{0}}$ is algebraic, i.e. $c_{\delta} \in \acl(\A_{0})$. Then $c_{\delta} \in \acl(\A_{0}) \cap \dcl(\A a_{\delta} b) \subseteq \acl(\A_{0}) \cap \dcl(\A_{0} b)$. Let $\D_{\delta}:= \dcl(\A a_{<|\T(\A)|^{++}} c_{< \delta}) \cap \acl(\A_{0})$, by \cref{lem: sequencestabilizes} there is $\mu \in |\T(\A)|^{++}$ such that $\D_{\mu}=\D_{\alpha}$ for all $\alpha> \mu$. In particular, for each $\alpha> \mu$ we have $c_{\alpha} \in \dcl(\A_{0} c_{<\mu})=\dcl(\A (a_{\delta})_{\delta \in |\T(\A)|^{++}} c_{<\mu})$.\\

\underline{Claim $3$:} \emph{There is some $n \in \mathbb{Z}_{>0}$ such that for any $\delta_{1}< \dots< \delta_{n}$ in $|\T(\A)|^{++}$ we have $c_{\delta_{n}} \in \dcl(\A a_{\delta_{1}}, \dots, a_{\delta_{n}}; c_{\delta_{1}},\dots, c_{\delta_{n-1}})$}. \\

\textit{Proof:}  To simplify the notation we write $\delta_{\leq s}$ for $\delta_{1}<\dots<\delta_{s}$, and
 $\eta_{\leq k}$ for $\eta_{1}<\dots<\eta_{k}$.
Since the sequence $(a_{\delta})_{\delta \in |\T(\A)|^{++}}$ is $\A b$-indiscernible there are $n,m \in \mathbb{Z}_{>0}$ such that $\delta_{1}<\dots< \delta_{n}<\eta_{1}< \dots <\eta_{m}$ and $ c_{\delta_{n}} \in \dcl(\A a_{\delta_{\leq n}},c_{\delta_{\leq n-1}}, a_{\eta_{\leq m}})$. \\

 Let $c'_{\delta_{n}} \equiv_{\A a_{\delta_{\leq n}}, c_{\delta_{\leq n-1}}} c_{\delta_{n}}$, we will show that $c_{\delta_{n}}=c'_{\delta_{n}}$. This implies that $c_{\delta_{n}} \in \dcl(\A a_{\delta_{\leq n}} c_{\delta_{\leq n-1}})$.\\
For each $k \leq m$ we construct a Morley sequence $(\alpha_{1}, \dots, \alpha_{m}) \models \p^{\otimes m} \rest_{\A b a_{\delta \in |\T(\A)|^{++}}c_{\delta \in |\T(\A)|^{++}} c'_{\delta_{n}}}$. The  two tuples $(a_{\delta_{\leq n}},a_{\eta_{\leq m}})$ and $(a_{\delta_{\leq n}}; \alpha_{1},\dots, \alpha_{m})$ both realize $\p^{\otimes m} \rest_{\A b}$, so they have the same type over $\A b$. Since $c_{\ell}=\f(a_{\ell})$ and $\f$ is $\A b$-definable it follows that:
\begin{equation*}
(a_{\delta_{\leq n}}, c_{\delta_{\leq n-1}}, c_{\delta_{n}}, a_{\eta_{\leq m}})\equiv_{\A b} (a_{\delta_{\leq n}}, c_{\delta_{\leq n-1}}, c_{\delta_{n}}, \alpha_{1}, \dots, \alpha_{m}).
\end{equation*}
Therefore,
\begin{equation}\label{eq1}
c_{\delta_{n}} \in \dcl(\A a_{\delta_{\leq n}},  c_{\delta_{\leq n-1}}; \alpha_{1},\dots, \alpha_{m}).
\end{equation}
Now the type $\p^{\otimes m}$ is $\A$-invariant so the fact that 
\begin{equation}\label{eq2}
(a_{\delta_{\leq n}}, c_{\delta_{\leq n-1}},c_{\delta_{n}}) \equiv_{\A} (a_{\delta_{\leq n}}, c_{\delta_{\leq n-1}}, c_{\delta_{n}}') 
\end{equation}
implies that
\begin{equation}\label{eq3}
(a_{\delta_{\leq n}}, c_{\delta_{\leq n-1}}, c_{\delta_{n}}, \alpha_{1}, \dots, \alpha_{m}) \equiv_{\A} (a_{\delta_{\leq n}} c_{\delta_{\leq n-1}}, c_{\delta_{n}}', \alpha_{1}, \dots, \alpha_{m}). 
\end{equation}
Since $(\alpha_{1},\dots, \alpha_{m})\models \p^{\otimes m} \rest_{\A b a_{\delta \in |\T(\A)|^{++}}c_{\delta \in |\T(\A)|^{++}} c'_{\delta_{n}}}$  (note that  this set contains both tuples in \cref{eq2}), then \cref{eq3} says that $c_{\delta_{n}}$ and $c_{\delta_{n}}'$ realize the same type over $(a_{\delta_{\leq n}}, b_{\delta_{< n-1}}, \alpha_{1}, \dots, \alpha_{m})$. Since $c_{\delta_{n}}$ is in the definable closure of this tuple by \cref{eq1}, it follows that $c_{\delta_{n}}'= c_{\delta_{n}}$. This completes the proof of the Claim.\\
By indiscernibility over $\A b $, for any \(a_{\leq n}\models \p^{\tensor
n}\rest_{\A b} \) and \(c_\delta = \f_{b}(a_{\delta})\), it follows that $c_{n} \in \dcl(\A a_{\leq n}
c_{<n})$. So there is an $\A a_{<n}b_{<n}$-definable function $\gamma( \cdot, a_{<n}c_{<n})$
sending \(a_n\) to \(c_n\) , this is  $c_n = \gamma(a_n,a_{<n}, c_{<n})$. \\

\underline{Claim $4$:}\emph{There is an $\A$-definable function $\g$ such that $[\f_{b}]_{\mathrm{p}}=\g(a_{<n}, c_{<n})$.}\\
Let $b'\models \tp(b/\A a_{<n}c_{<n})$. We claim that then $[f_{b'}]_{\p} = [f_b]_{\p}$. Indeed, take $a_* \models p \rest_{\A bb'a_{<n}}$. Then $f_b(a_*)=\gamma(a_*,a_{< n}, c_{<n})=f_{b'}(a_*)$. It follows that the germ $[f_b]_{\p}$ is invariant under automorphisms fixing $\A a_{<n} c_{<n}$, hence $[f_b]_{\p} \in \dcl(\A a_{<n}c_{<n})$. Thus we can write $[f_b]_{\p} = g(a_{<n},c_{<n})$ for some $\A$-definable function $g$. This concludes the proof of the claim.\\

\underline{Claim $5$:}\emph{The germ $[\f_{b}]_{\p} \in \St_{\A}(M)$.}\\
Fixing $\mathbf{a} := a_{<n}$, let $Y_{\mathbf{a}}$ be the image of the map $g(\mathbf{a},\cdot)$. Note that $Y_{\mathbf{a}}$ is $\A \mathbf{a}$-definable and is stable stably embedded: indeed it is internal to $\St_{\A \mathbf{a}}$ by construction. If $\M_0$ is a small model containing $\A$ such that $\mathbf{a}\models \p^{\otimes n}\rest_{\M_0}$, then $Y_{\mathbf{a}}$ contains all the germs $[f_{b'}]_{\p}$ for $b' \in M_0$. It follows that any small subset of the set of germs is included in a set of the form $Y_{\mathbf{a}'}$ for some $\mathbf{a}' \models \p^{\otimes n}$. This implies in particular that the set of germs is stable (any witness of instability can be included in a $Y_{\mathbf{a}'}$, which is stable).

In an NIP theory, any stable set is stably embedded. Hence the set of germs is stable stably embedded as required.
%The situation we have now is that the set $X$ of germs is stable and we have a definable family $Y_{\mathbf{a}}$ of definable subsets of $X$ such that any small subset of $X$ is included in one element of the family. This implies that finitely many of the $Y_{\mathbf{a}}$'s cover $X$: if some $Y_{\mathbf{a}_0},\ldots,Y_{\mathbf{a}_m}$ do not cover $X$, take $d_m\in X$ witnessing it and find $Y_{\mathbf{a}_{m+1}}$ containing $d_{<m}$. This gives a witness to the order property. It follows that the set $X$ of germs is internal to $\St_{\A \mathbf{a}_0\ldots \mathbf{a}_m}$ for some $\mathbf{a}_0\ldots \mathbf{a}_m$. Hence that set is stable and stably embedded, as required.
\end{proof}

\begin{remark}\label{rem: generalizations} Assume $\T$ is $\NIP$. Let $\A$ be a small set of parameters and $\N$ be a sufficiently saturated and homogenous model that contains $\A$. Let $\p$ be a global $\A$-definable type  and let $\f_{b}$ be an
$\A$-definable family of functions and let $\Lat$ be a stable and stably embedded set. 
Let $\Theta$ be the $\A$-interpretable set of all $[\f_{m}]_{\p}$, then $\Theta$ is stable and internal to $\Lat$. 
\end{remark}
\begin{proof}
 Fix $d \in \N$. The proof of the first statement follows exactly as in \cref{prop: correction}. We start with $(a_{\alpha})_{\alpha \in |\T(\A)|^{++}}$ a Morley sequence of $\p$ and let $c_{\alpha}=\f_{b}(a_{\alpha}) \in \Int(\Lat, \A a_{\alpha})$. Then $c_{\alpha}=f_{b}(a_{\alpha}) \in \mathrm{X}_{\alpha}$ where $\mathrm{X}_{\alpha}$ is an $\A a_{\alpha}$ -definable set internal to $\Lat$. Since $\Lat$ is a stable definable set, then $\mathrm{X}_{\alpha}$ is also a stable and stably embedded set, then the proof of the first, second, third and forth claim go trough. For the fifth claim  the set $\Y_{\A\mathbf{a}}$ is internal to  $\Int(\Lat, \A a_{< n})$-internal, instead of internal to $\St_{\A \mathbf{a}}$. In particular, it is internal to $\Lat$.\\
 The same argument in the fifth  Claim of \cref{prop: correction} shows that $\Theta$ is stable.\\
Since $\Theta$ is an $\A$-interpretable set, it is only left to show that it is $\Lat$-internal. \\
\underline{Claim:} Let $M_{0}$ be a small model then $\Theta(M_{0})$ is covered by finitely many of the $\Y_{\A \mathbf{a}}(\M_{0})$'s for $d \in \M_{0}$. \\
\textit{Proof:} We proceed by contradiction. By compactness there is some $n \in \mathbb{N}$ and a $\La(\A)$-formula $\psi(x,y)$ such that for every $d \in \M$ and realization $a_{<n} \models \p^{< \tensor n} \rest_{\A d}$ the set $\Y_{\A \bar a_{<n}}$ is defined by $\psi(x, \bar a_{< n})$. By induction we construct $a_{i}$ and $d_{i}$ such that:
\begin{itemize}
\item $\bar a_{i+1} \vDash \p^{\tensor n-1}\rest_{\A \bar a_{\leq i}d_{\leq i}}$. 
\item $e_{i}=[\f_{d_{i}}]_{\p} \in \Theta \backslash (\Y_{\bar a_{0}} \cup \dots \cup \Y_{\bar a_{i}})$.
\end{itemize}
Then $e_{i}=[\f_{d_{i}}]_{\p} \in \Y_{\bar a_{j}}$ if and only if $j>i$, but this contradicts the stability of $\Theta$. This concludes the proof of the Claim. Since each $\Y_{\A \bar{a}_j}$ is $\Lat$-internal, $\Theta$ also is. 
\end{proof}
We conclude this section with the counterexample to \cite[Proposition $6.7$]{HHM}(1). 

\begin{remark}\label{notstrong} 
\begin{enumerate}
\item Let $\M$ be a meet-tree (in the language $(\leq, \wedge)$) and $c \in \M$ a point. The closed cone of center $d$ is by definition $\C(d)=\{ x \in \M \ | \ x \geq d\}$. We define on $\C(d)$ a relation $E_{d}$ by: $xE_{d}y$ if $ x \wedge y > d$. One can easily check that this is an equivalence relation. We define an open cone of center $d$ to be an equivalence class under the relation $E_{d}$. Let $\mathcal C$ be the class of finite two-sorted structures $(\M_0,\Lat_0,g(x,y))$, where $\M_0$ is a meet-tree, $\Lat_0$ a set with no structure and for every $d\in \M$, $g$ is a function $g\colon \{(d,e)\in \M_0^2 \colon d < e\} \to \Lat_0$ such that $g_d := g(d,\cdot)$ is constant on open cones of center $d$, hence induces a map from $C(d)/ E_{d}$ to $\Lat$. This class is easily seen to be a Fra\"iss\'e class and we let $T$ be the theory of its Fra\"iss\'e limit and $(\M, \Lat)$ a model of it.

The theory $T$ is NIP as is easily checked by counting types over finite sets: Let $(M_0, \Lat_0)$ be a finite subset of $(\M, \Lat)$ of size $n$ and $(M_1,\Lat_1)$ the substrucutre generated by it. The meet tree $M_1$ generated by $M_0$ has size at most $n^2$ (since every element of $M_0$ is the meet of two elements of $M_0$). Then closing under $g$ also adds at most $n^2$ points to $\Lat$, hence $(M_1, \Lat_1)$ has size polynomial in $n$ (indeed it can be seen to have size linear in $n$, but we do not need this). There is a unique non-realized 1-type of an element of the sort $\Lat$ over $\Lat_1$. In particular, the sort $\Lat$ is strongly minimal. Let now $a$ be a new element of the tree sort. The set $\{a \wedge c \colon c\in \M_1\}$ is linearly ordered. Let $a_* := a\wedge c_*$ be its maximal element. Then the meet-tree generated by $\M_1 a$ is equal to $\M_1 \cup \{a, a_*\}$ (with possibly $a_* = a$, or $a_* \in \M_1$). From this tree, we can define at most two new open cones that are not definable from $\M_1$ alone (with center $a_*$ and containing $a$ and $c_*$ respectively). It follows that $\tp(a/\M_1)$ is given by its type in the meet-tree structure and at the types of the images of at most two pairs of elements of the resulting tree. Hence there are quadratically many 1-types over $\M_1$.

Let $\p(x)$ be the generic global type of $\M$: that is the type of an element in the tree incomparable with all elements in $\M$. Then $p$ is $\emptyset$-definable. For $c \in \M$ let $\f_{c} \colon \M \rightarrow \Lat$, defined as $\f_{c}(x)= \g_{x\wedge c}(x)$ (i.e. given $x \in \M$ we look at the maximal open ball containing $x$ in the closed cone around $x \wedge c$, and we send $x$ to the image of this maximal open ball under the map $\g_{x \wedge c}$). Let $a \vDash \p \upharpoonright_{M}$ then for any $c, c' \in \M$, $\f_{c}(a)=\f_{c'}(a)$ (as $c\wedge c' > c\wedge a = c'\wedge a$). Consequently, $[\f_{c}]_{\p}$ is $\emptyset$-definable. If the germ of $\f_{c}$ over $\p$ is strong, then there is a $\emptyset$-definable function $\alpha$ such that for any realization $a \vDash \p \upharpoonright_{c}$, $\f_{c}(a)=\alpha(a) \in \Lat$. In particular, we can define elements in $\Lat$ with a single element in the tree, but this is impossible: we have quantifier elimination in the meet-tree language with the function $g$ and from one element, one cannot define any other.
\item Let $\p$ be a global $\A$-definable type and $\f_{c}$ be a definable function. We say that the germ of $\f_{c}$ over $\p$ is \emph{almost strong} if for any $a \vDash \p\upharpoonright_{c}$, $\f_{c}(a) \in \acl(\A, [f_{c}]_{p}, a)$. The previous example illustrates that the germ of $f_{c}$ over $\p$ is not even almost strong under the hypothesis of Proposition \ref{mistake}. Take $\q$ be the generic type of $\M$ which is $\emptyset$-definable, and $\p=\q$. All the functions $\{ \f_{c} \ | \ c \vDash \q\}$ have the same $\p$-germ, i.e. for $c \vDash \q$, $c' \vDash \q$ if $a \vDash \p\upharpoonright_{c,c'}$ then $f_{c}(a)=f_{c'}(a)$. If $\p$ has almost strong germs, then there is some formula $\phi(x,y)$ such that for any $c \vDash \q$ and $a \vDash \p\upharpoonright_{c}$, $\phi(a,y)$ defines a finite set and $\f_{c}(a) \in \phi(a,y)$. But no finite set in $\Lat$ is definable from a single point in the tree. 
\end{enumerate}
\end{remark}

\section{Descent}
The main result of this section is the following theorem.
\begin{theorem}\label{thm:descent}
    Let $\p$ be a global $\A$-invariant type and let $b$ be such that $\p$ is stably dominated over $\A b$. Then $\p$ is stably dominated over $\A$.
\end{theorem}
This result generalizes \cite[Theorem $4.9$]{HHM}, which assumes that $\tp(b/\A)$ has a global $\A$-invariant extension. This provides a positive answer to the question posed by Hrushovski and Rideau-Kikuchi in \cite[Question 1.3(1)]{metastable}. 
\subsection{Easier cases of descent}\label{sec: easy descent}
In this subsection we start by presenting some simplified versions of descent, giving in particular a much simpler proof for the case of $\ACVF$.

\medskip
The following proposition is a generalization of \cite[Lemma $4.2$]{HHM}. It proves descent in a very special case, namely assuming that the parameters $b$ are not needed to define $\Lat$, nor the dominating function. When the dominating function takes value in a stable set, we do not need any other assumptions on $\tp(b/\A)$.

\begin{proposition}\label{prop:easy descent}
  Let $\p$ be a generically stable $\A$-invariant global type and assume that it is dominated over $\A b$ by an $\A$-definable function $\f$ into some sort $\Lat$. Assume furthermore either:
    \begin{enumerate}
    \item  $\tp(b/\A)$ does not fork over $\A$, or 
    \item  $\p^{ \otimes n}$ is generically stable for all $n$ (which holds when $\Lat$ is stable).
    \end{enumerate}
    Then $\f$ dominates $\p$ over $\A$. 
\end{proposition}
\begin{proof}
    Let $\q= \f(\p)$; this is a global $\A$-invariant and generically stable type. Let $a \models \p \rest_{\A}$ and $d$ be such that $\f(a) \models \q \rest_{\A d}$. We want to show that $a\models \p \rest_{\A d}$. By \cref{prop: generically stable properties}.(2) it is enough to show $a\downfree_{\A}^{f} d$.\\
\begin{enumerate}

  \item   Assume first that $\tp(b/\A)$ is non-forking over $\A$. \\
    
    \underline{Claim:} \emph{There is some $b'$ such that $b' \equiv_{\A} b$, $a\downfree_{\A}^{f} b'$ and $\f(a) \models \q \rest_{\A db'}$.}\\
    \textit{Proof:} By \cref{prop: properties of non-forking}.(6), since $b \downfree_{\A}^{f} \A$ there is $b' \equiv_{\A} b$ such that
    \begin{equation}\label{eqf1}
    b' \downfree_{\A}^{f} d a.  
    \end{equation}
 As  $b' \downfree_{\A d}^{f} a$ (by \cref{prop: properties of non-forking}.(4) and \cref{eqf1}), by  symmetry (\cref{prop: generically stable properties}.(4)) we  have $a \downfree_{\A d}^{f} b'$. Since $\f$ is $\A$-definable, it follows that
\begin{equation}\label{eqf2}
  \f(a) \downfree_{\A d}^{f} b'.
\end{equation}

Since $\f(a) \models \q\rest_{\A d}$ then 
\begin{equation}\label{eqf3}
\f(a) \downfree_{\A}^{f} d.
\end{equation}
Combining \cref{eqf3}, \cref{eqf2} and right transitivity (\cref{prop: generically stable properties}.(5)) we have 
    \begin{equation}\label{eqf4}
    \f(a) \downfree_{\A}^{f} b'd.
    \end{equation}
    The claim then follows from \cref{prop: generically stable properties}.(2). \\
    
 Because $\p$ is dominated by $\f$ over $\A b$, it is dominated by $\f$ over $\A b'$.
    By monotonicity (\cref{prop: properties of non-forking}.(3) and \cref{eqf1}) $b'\downfree_{\A}^{f} a$ and by symmetry (\cref{prop: generically stable properties}.(4)) it follows that: 
    \begin{equation}\label{eqf5}
    a \downfree_{\A}^{f} b'.
    \end{equation}
    By domination of $\p$ via the function $\f$ over $\A b'$, we have: 
    \begin{equation}\label{eqf6}
    a \downfree_{\A b'}^{f} d.
    \end{equation}
    Combining \cref{eqf5}, \cref{eqf6} and right transitivity (\cref{prop: generically stable properties}.(5)) we conclude $a \downfree_{\A}^{f} b'd$. It follows that $a \downfree_{\A}^{f} d$ by monotonicity of non-forking (\cref{prop: properties of non-forking}.(2)). This concludes the proof in the first case. 
    \item Suppose now that $\p^{\otimes n}$ is generically stable for all $n$. By Lemma \ref{lem: sequence for gen n stable} there is a sequence $(\bar a_{i}, b_{i} \ | \ i < |\T(\A)|^{+})$ such that $\bar a_{i}$ is a Morley sequence in $\p^{\omega}$ over $\A$,  $\bar a_{i}$ is a $\p$-basis of $b_{i}$ and $\bar a_{i} b_{i} \equiv_{\A} \bar a_{0} b_{0}$.\\
    Since $\f(a) \downfree_{\A}^{f} d$ we may assume $\f(a) \downfree_{\A}^{f} d b_{<|\T(\A)|^{+}} \bar{a}_{<|\T(\A)|^{+}}$ (\cref{prop: properties of non-forking}.(6)), thus by monotonicity and base monotonicity of non-forking  (\cref{prop: properties of non-forking}.(3-4)) we have:
    \begin{equation}\label{eqf5}
    \f(a) \downfree_{\A b_{i}}^{f} d \ \text{for any $i \in |\T(\A)|^{+}$}.
    \end{equation}
    Because $\p^{\omega}$ is generically stable and $\A$-invariant, there is some $i < |\T(\A)|^{+}$ such that $ \bar a_{i} \downfree_{\A}^{f} a$ (\cref{fact: first things about generically stable}.(2)). By symmetry $a \downfree_{\A}^{f} \bar a_{i}$ (\cref{prop: generically stable properties}.(4)), thus  as $\bar a_{i}$ is a $\p$-basis of $b_{i}$ we have:
    \begin{equation}\label{eqf6}
    a \downfree_{\A}^{f} b_{i}.
    \end{equation}
    
     Since $\p$ is dominated over $\A b_{i}$ by \cref{eqf5} we have:
     \begin{equation}\label{eqf7}
     a \downfree_{\A b_{i}}^{f} d. 
     \end{equation}
     Combining \cref{eqf6}, \cref{eqf7} and right transitivity 
    $a \downfree_{\A}^{f} b_{i}d$ (\cref{prop: generically stable properties}.(5)). We conclude that $a \downfree_{\A}^{f} d$ by monotonicity of non-forking (\cref{prop: properties of non-forking}.(3)). 
    \end{enumerate}
\end{proof}
We now move to a second case where the parameter $b$ is used to define the function witnessing domination, but is not needed to define $\Lat$. 
\begin{theorem}\label{thm: descent for internal sets} Let $\p$ be generically stable over $\A$ and assume that it is dominated over $\A b$ by an $\A b$-definable function $\f_{b}$ into some $\A$-definable stably embedded set $\Lat$.  Assume furthermore either:
\begin{enumerate}
\item $\tp(b/ \A)$ does not fork over $\A$, or
\item $\p^{\otimes n}$ is generically stable for all $n < \omega$.
\end{enumerate}
 Then there is an $\A$-definable function $\h \colon \p \rightarrow \Int(\Lat,\A)$ that dominates $\p$ over $\A$. 
\end{theorem}
\begin{proof}
 Let $e=[\f_{b}]_{\p}$ be the $\p$-germ of $\f_{b}$. By Theorem \ref{stronggerms}, there is an $\A$-definable function $\g(x,y)$ such that  $\f_{b}(a)=\g(e,a)$ for $a \models\p\rest_{\A b}$. By Proposition \ref{prop:easy descent}, $\p$ is dominated by $\g(e, \cdot)$ over $\A e$.
 
 Let $\Theta$ be the set of all $\p$-germs of  instances of $\f_{\underline{}}$. This is an $\A$-definable set and $\Lat$-internal by Lemma \ref{lem: internal $p$-germs}(1). Let $\X(x,y)$ be the $\A$-definable set defined by  
 \begin{equation*}
 x \in \Theta \wedge \g(x,y) \in \Lat \wedge \exists! z (\g(x,y)=z).
 \end{equation*}
 For $a \models \p\rest_{\A}$ let $\Gamma_{a}=\{ (e', \g(e',a)) \ | \ (e',a) \in \X \}$. This is an $\A a$-definable subset of $\Theta \times \Lat$.  By Lemma \ref{lem: internal $p$-germs}(2) $\Int(\Lat, \A)$ is stably embedded and note that it is closed under interpretable sets. Thus we can define the $\A$-definable function $\displaystyle{\h \colon \p \rightarrow \Int(\Lat,\A)}$    which sends $a \models \p \rest_{\A}$ to $\ulcorner \Gamma_{a} \urcorner$.\\ 
 
\underline{Claim:} \emph{The type $\p$ is dominated by $\h$ over $\A$.}\\
\textit{Proof:} Let $a \models \p \rest_{\A}$ and let $d$ be such that $\h(a) \downfree_{\A}^{f} d$ and we aim to show that $a \downfree_{\A}^{f} d$.\\
\begin{enumerate}
\item Assume first that $\tp(b/\A)$ does not fork over $\A$. Since $e \in \dcl(\A b)$, then $e \downfree_{\A}^{f}\A$. Take $e' \equiv_{\A} e$ such that $e' \downfree_{\A}^{f} a d$ (\emph{c.f.} \cref{prop: properties of non-forking}.(7)). As $\h(a) \in \dcl(\A a)$ we must have
\begin{equation}\label{eqnf21}
e' \downfree_{\A}^{f} \h(a) d.
\end{equation}
Since $e' \downfree_{\A}^{f} a d$ by monotonicity (\emph{c.f.} \cref{prop: properties of non-forking}.(3)) we have $ e' \downfree_{\A}^{f} a$, and by symmetry (\emph{c.f.} \cref{prop: generically stable properties}.(4))
we conclude that:
\begin{equation}\label{eqnf22}
a \downfree_{\A}^{f} e'. 
\end{equation}

We will show that $\g(e',a) \downfree_{\A e'}^{f} d$. Combining base monotonicity (\emph{c.f.} \cref{prop: properties of non-forking}.(4)) and \cref{eqnf21} we have that $e' \downfree_{\A d} \h(a)$. Because $h(\p)$ is a generically stable type $\A$-invariant type, by symmetry (\emph{c.f.} \cref{prop: generically stable properties}) we conclude
\begin{equation}\label{eqnf23}
\h(a) \downfree_{\A d}^{f} e'.
\end{equation}

Combining \cref{eqnf23}, the hypothesis that $\h(a) \downfree_{\A} d$ and right transitivity (\emph{c.f} \cref{prop: generically stable properties}.(5)),  we have that $\h(a) \downfree_{\A}^{f} e' d$ so $\h(a) \downfree_{\A e'}^{f} d$ (by base monotonicity \cref{prop: properties of non-forking}.(4)). As $\g(e',a) \in \dcl(\A e',\h(a))$, thus $\g(e', a) \downfree_{\A e'}^{f} d$. Since $\p$ is dominated by $\g(e', \cdot)$ over $\A e'$, then 
\begin{equation}\label{eqnf24}
a \downfree_{\A e'}^{f} d.
\end{equation}
By \cref{eqnf24}, \cref{eqnf22} and right transitivity (\emph{c.f} \cref{prop: generically stable properties}.(5)) $a \downfree_{\A}^{f} de'$. By monotonicity of non-forking $a \downfree_{\A}^{f} d$ (\emph{c.f.} \cref{prop: properties of non-forking}.(3)). This concludes the proof of the claim for the first case. \\
\item Assume now that $\p^{\otimes n}$ is generically stable for all $n< \omega$. By \cref{lem: sequence for gen n stable} there is a sequence $(\bar a_{i}, e_{i} \ | \ i < |\T(\A)|^{+})$ where $\bar a_{i}$ is a Morley sequence in $\p^{\omega}$ over $\A$,  $\bar a_{i}$ is a $\p$-basis of $e_{i}$ and $\bar{a}_{i}e_{i} \equiv_{\A} \bar{a}_{0}e$.\\
Since $\h(a) \downfree_{\A}^{f} d$, we may assume that $\h(a) \downfree_{\A}^{f} d e_{<|\T(\A)|^{+}} \bar a_{<|\T(\A)|^{+}}$ (\emph{c.f.} \cref{prop: properties of non-forking}.(7)). In particular, for any $k \in |\T(\A)|^{+}$ we have $\h(a) \downfree_{\A e_{k}}^{f} d$ (by base monotonicity \emph{c.f.} \cref{prop: properties of non dividing}.(4)). Thus $\g(e_{k},a) \downfree_{\A e_{k}}^{f} d$, as $\g(e_{k},a) \in \dcl(\A e_{k} \h(a))$. \\
Because $\p^{\omega}$ is generically stable over $\A$ there is some $i \in |\T(\A)|^{+}$ such that $a \downfree_{\A}^{f} \bar a_{i}$ (\emph{c.f.} \cref{fact: first things about generically stable}.(2) and symmetry -\emph{c.f.} \cref{prop: generically stable properties}.(4)), and since $\bar{a}_{i}$ is a $\p$-basis of $e_{i}$ we have 
\begin{equation}\label{eqnf25}
a \downfree_{\A}^{f} e_{i}.
\end{equation}
Since $\g(e_{i},a) \downfree_{\A e_{i}}^{f} d$ by domination of $\p$ via $\g(e_{i}, \cdot)$ over $\A e_{i}$ one has
\begin{equation}\label{eqnf26}
a \downfree_{\A e_{i}}^{f} d.
\end{equation}
By \cref{eqnf25}, \cref{eqnf26} and right transitivity (\cref{prop: generically stable properties}.(5)), $a \downfree_{\A}^{f} e_{i} d$ so $a \downfree_{\A}^{f} d$ (by monotonicity, \cref{prop: properties of non-forking}.(3)). This concludes the proof of the claim for the second case, and therefore the proof of the theorem.
\end{enumerate}
\end{proof}
 
\begin{remark}\label{rem: easy ACVF} This already provides a simplified proof of descent for stably dominated types in $\ACVF$. Let $\p$ be a global generically stable $\A=\acl(\A)$-invariant type and $k$ be the residue field sort. For any set of parameters $ \A \subseteq \C=\acl(\C)$ one can consider the structure $\mathrm{VS}_{k,\C}$ defined as in \cite[Definition 7.6]{HHM}. By \cite[Proposition $7.8$]{HHM}, $\St_{\C}$, $\Int(k,\C)$ and $\mathrm{VS}_{k,\C}$ are the same. Thus if $\p$ is stably dominated over $\C$, by adding parameters for the basis of each $k$-vector space $\red(s)=s/\mathcal{M}s$ where $s$ is a $\C$-definable $\mathcal{O}$-lattice we can assume that the domination function goes into $k$. By Theorem \ref{thm: descent for internal sets}, $\p$ is stably dominated over $\A$.  
\end{remark}
\subsection{A brief description of the proof of descent}
We would like now to move to the proof of Theorem \ref{thm:descent}. The difference with the previous theorem is that we allow the set $S=S_b$ to depend on $b$. This makes the proof substantially more difficult because we need to produce a stable stably-embedded set defined over $\A$ and it is not clear a priori how to do that. The proof is rather technical and consists of two parts: 
\begin{enumerate}
\item For the simplified case that $\tp(b/ \A)$ has a global $\A$ invariant extension $\q$, we show that given $(b_i)_{i \in \I}$, a sufficiently large Morley sequence in $\q$ over $\A$, and $a \models \p\rest_{\A}$, then we can remove a bounded number of points from the sequence $(b_i)_{i \in \I}$ so as to make it independent from $a$ (in the sense that $a$ realizes $\p$ over it, this is \cref{lem:independent subsequence for stably dominated types}).  %More precisely, it allows us to show that given a global $\A$-invariant extension $\q$ of $\tp(b/\A)$ and $(b_{i})_{i \in \I}$ a sufficiently large Morley sequence in $\q$, if $a \models \p\rest_{\A}$  then there is some element in the sequence $b_{i}$ such that $a \models \p_{\rest_{\A b_{i}}}$ (Proposition \ref{prop: a independent from some element in morley seq of q}).\\
For the general case, we show the existence of an $\A$-indiscernible sequence sequence $(b_{i})_{i \in |\T(\A)|^{+}}$ in the type $\tp(b/\A)$ such that:
\begin{enumerate}
\item  $b_{i} \downfree_{\A}^{\p} b_{< i}$ for all $i$,
\item  $(b_{i})_{i>0}$ is $\acl$-independent over $\A b_{0}$, and
\item  $(b_{i})_{i \in |\T(\A)|^{+}}$ is $\acl^{\vee}$-independent over $\A$.
\end{enumerate}
We prove that given $a \models \p\rest_{\A}$, we can remove a bounded number of points from the sequence $(b_i)_{i \in |\T(\A)|^{+}}$ so as to make it independent from $a$ (in the sense that $a$ realizes $\p$ over it, this is \cref{lem: independent over big tail p^{n} generically stable}, and it is where  $\p$-independence and $\acl$-independence of the sequence over the first element play a role). \\
This is the key step that allows us to simplify and generalize the arguments from \cite[Chapter $4$]{HHM}. 
\item The second part of the proof essentially follows the construction from \cite[Theorem $4.9$]{HHM}, though the details of it and the proof of why it works are slightly different. 
\end{enumerate}
Each of those two parts are substantially easier to prove if we assume that $\tp(b/\A)$ extends to a global $\A$-invariant type as in \cite[Theorem 4.9]{HHM} and this will be made explicit at the cost sometimes of giving two different proofs of the same result. To simplify the presentation we present first a complete proof under the assumption that $\tp(b/\A)$ has a global $\A$-invariant extension $\q$ (Theorem \ref{thm: descent with invariant extension}). Once this has been clarified, we prove Theorem \ref{thm:descent} making explicit the required modifications for the argument.
\subsection{The proof of descent assuming that $\tp(b/\A)$ has a global non-forking extension $\q$}
In this section we will prove the following theorem.
\begin{theorem}\label{thm: descent with invariant extension}  Let $\p$ be a global $\A$-invariant type and let $b$ be such that $\p$ is stably dominated over $\A b$. Assume the type $\tp(b/\A)$ has a global $\A$-invariant extension $\q$. Then $\p$ is stably dominated over $\A$.
\end{theorem}
In the following subsection we tackle the first step of the proof, i.e. we show that given $(b_i)_{i \in \I}$, a sufficiently large Morley sequence in $\q$ over $\A$, and $a \models \p\rest_{\A}$, one can remove a bounded number of points from the sequence $(b_i)_{i \in \I}$ so as to make it independent from $a$ (in the sense that $a$ realizes $\p$ over it, this is \cref{lem:independent subsequence for stably dominated types}).
\subsubsection{Extracting a large sub-sequence independent from $a \models \p_{\rest_{\A}}$ when $\tp(b/\A)$ has a global $\A$-invariant extension $\q$}

\begin{lemma}\label{lem:independent subsequence for stably dominated types}
    Let $\p$ be a global $\A$-invariant type and assume it is stably dominated over $\B \supseteq \A$.  Let $(b_i)_{i\in \I}$ be a sequence of finite tuples, indiscernible and $\acl$-independent over $\B$. Assume $|\I| \geq |\T(\B)|^{+}$ and let $a \models \p \rest_{\B}$, then there is $\J \subseteq \I$, $|\I \setminus \J| \leq |\T(\B)|$, such that $a \downfree^{f}_{\B} b_{\J}$.
    
    In particular, the statement holds if $(b_{i})_{i \in \I}$ is an indiscernible and non-dividing sequence over $\B$. 
\end{lemma}
\begin{proof}
    \underline{Claim:} \emph{For each  $k<\omega$,  there is $\I_{k} \subseteq \I$ such that $|\I \setminus \I_{k}| \leq | \T(\B)|$ and $a \downfree_{\B}^{f} b_{S}$ for any subset $S$ of size $k$ of $\I \setminus \I_{k}$.}\\
    \textit{Proof of claim:} If there is a subset $\J_{0}$ of $\I$ of size $k$  such that $a\ndownfree^{f}_{\B} b_{\J_{0}}$, let $\I'_{0} = \I \setminus \J_{0}$. If there is a subset $\J_{1}$ of $\I'_{0}$ of size $k$ such that $a\ndownfree^{f}_{\B} b_{\J_{1}}$, then let $\I'_{1} = \I'_{0} \setminus \J_{1}$, and keep going. It is sufficient to argue that the removal process must stop after less than $|\T(\B)|^+$ steps. We proceed by contradiction. First, by the pigeonhole principle we can assume that it is always the same formula $\psi(a,b_{J_{l}})$ that witnesses the non-independence $a \ndownfree^{f}_{\B} b_{J_{l}}$ for all $l \in |\T(\B)|^{+}$.
    
    Next, we can further assume that the sequence of indices $(\J_{l})_{l<|\T(\B)|^+}$ is quantifier-free indiscernible in the language ${<}$ on $\I$. (Why? Set $L=|\T(\B)|^+$. Consider the first order structure $N$ where we add to the universe a linearly ordered sort for $I$ along with functions from $I$ to $M$ enumerating $b_i$ for each $i\in I$, another sort $L$ along with $k$ functions from $L$ to $I$ enumerating $J_l$ for $l\in L$. In an elementary extension $N^*$ of this structure, take an indiscernible sequence $L_0 \subseteq L^*$ indexed by $|\T(\B)|^+$. Then replace $M$ by $M^*$, $(b_i)_{i\in I}$ by $(b_i)_{i\in I^*}$ and consider the sequence $(J_l)_{l\in L^*}$. We still have $a \ndownfree^{f}_{\B} b_{J_{l}}$ for all $l \in L^*$, as witnessed by the same formula $\psi$.)
    
    The sequence $(b_{\J_{i}})_{i<|\T(\B)|^+}$ is a $\B$-indiscernible sequence and $\acl$-independent over $\B$. By Lemma \ref{lem: acl-independence implies Morley in St} $(\St_{\B}(\B; b_{\J_{i}}):i<|\T(\B)|^+)$ is independent in the stable structure $\St_{\B}$. By Lemma \ref{lem: weight stable}  $\St_{\B}(\B; a) \downfree^{f}_{\B} \St_{\B}(\B; b_{\J_{i}})$ for some $i \in |\T(\B)|^{+}$. Because $\p$ is stably dominated over $\B$, then  $a\downfree^{f}_{\B} b_{\J_{i}}$, and this is a contradiction. This concludes the proof of the Claim.\\
    
    Now take $\displaystyle{\J=\bigcup_{k < \omega} \I_{k}}$, then by construction $a \downfree^{f}_{\B} b_{\J}$ and $|\I \setminus \J| \leq |\T(\B)|$.\\
    The last part of the statement is an immediate consequence of \cref{lem: acl non dividing}.
    \end{proof}

\begin{proposition} \label{prop: a independent from some element in morley seq of q} Let $\p$ be a global $\A$-invariant type and let $b$ be such that $\p$ is stably dominated over $\A b$.  Let $\q$ be a global $\A$-invariant extension of $\tp(b/\A)$. Let $(b_i)_{i\in \I}$ be a Morley sequence of $\q$ over $\A$ with $|\I|\geq |\T(\A)|^{+}$ and let $a \models \p\rest_{\A}$. Then there is $i\in \I$ such that $a \downfree_{\A}^{f} b_i$.    
\end{proposition}
\begin{proof}
Without loss of generality assume that $|\{ i \in \I \ | \ a \ndownfree_{\A}^{f} b_{i}\}| \geq |\T(\A)|^{+}$, recall that we write $\M$ to denote the monster model.\\ 
\underline{Step 1, Going up:} \emph{ Assume that there is some $i\in \I$ such that $a\models \p\rest_{\A b_i}$, then there is a subset $J\subseteq \I_{>i}$ such that $|\I_{>i}\setminus \J|\leq |\T(\A)|$ and we have $a\models \p\rest_{\A b_ib_{\J}} $}.\\
\textit{Proof:} Since $\p$ is dominated over $\A b$, then it is dominated over $\A b_{i}$. The sequence $b_{\I_{> i}}$ is indiscernible and non-dividing over $\A b_{i}$.  We conclude by Lemma \ref{lem:independent subsequence for stably dominated types} (applied to the base $\A b_{i}$).\\
\underline{Step 2, Going down:} \emph{Assume that there is some $i\in \I$ such that $a\models \p\rest_{\A b_i}$ and $\I$ does not have a maximal element.  Then $|\{ j < i \ | \ a \ndownfree_{\A}^{f} b_{j} \} | \leq |\T(\A)|$.}\\ 
\textit{Proof:} Assume not, by the pigeonhole principle  we may assume that forking dependence is witnessed by the same $\La(\A)$-formula $\phi(x,y)$, i.e.  for $d\models \q\rest_{\A}$, $\p\rest_{\A d} \vdash \phi(x,d)$ and $\M \vDash \neg \phi(a,b_j)$ holds for $j<i$.  By Step 1, we can further assume that $a\models \p\rest_{\A b_{\geq i}}$.\\
 Let $\R= \R^{0}+ \R^{-1}$ where $k \in \{ 0,-1\}$ and $\R^{k}$ is a copy of $(\R, \leq)$, a dense linear order without endpoints of size bigger than $|\T(\A)|^{+}$.  By compactness, we can construct a Morley sequence $(b_{\ell})_{\ell \in \R}$ in $\q$ over $\A$  such that $a\models \p \rest_{\A b_{\ell}}$ for $\ell \in \R^{-1}$ while $\M \models \neg \phi(a, b_{\ell})$ holds for $\ell \in \R^{0}$. \\
 We are now going to construct inductively sequences $(\bar b_{i})_{i < |\T(\A)|^{+}}$ and $(a_{i})_{i < |\T(\A)|^{+}}$ satisfying the following conditions:

 \begin{enumerate}
 \item for every $\eta < |\T(\A)|^{+}$, $a_{\eta} \models \p \rest_{\A a_{<\eta} \bar b_{<\eta}}$ and  $\M \models \neg \phi(a_{\eta},b)$ for $b$ and element of the sequence $\bar b_{\eta}$,
 \item the concatenation $ \bar b_{\eta}+ \dots+ \bar b_{0}+ \bar b_{-1}$ is a Morley sequence in $\q$ over $\A$, 
 \item if $\beta< \eta$ then $a_{\beta} \nmodels \p\rest_{A b}$ for $b$ an element of the sequence $\bar b_{\eta}$. 
 \end{enumerate}
Set $\bar b_{-1} = (b_{\ell})_{\ell \in \R^{-1}}$, $\bar b_{0} = (b_{\ell})_{\ell \in \R^{0}}$, and $a_0=a$. Assume $(a_{\beta})_{\beta< \eta}$ and $(\bar b_{\beta})_{\beta< \eta}$ have been constructed satisfying the $(a), (b)$ and $(c)$ requirements. Let $a_{\eta} \models \p \rest_{\A a_{<\eta} \bar b_{<\eta}}$. By compactness we can find $\bar b_{\eta}$ such that $\bar b_{\eta} + \dots+ \bar b_{0} + \bar b_{-1}$ is a Morley sequence of $\q$ over $\A$ and $\M \models \neg \phi(a_{\eta},b)$ for all elements $b$ in the sequence $\bar b_{\eta}$.\\ 
We just need to verify that condition $(c)$ holds. Fix $\beta < \eta$ and assume there is some $b$ in $\bar b_{\eta}$ such that $a_{\beta} \downfree^{f}_{\A} b$. By Step $1$ we would have $a_{\beta} \models \p\rest_{\A b'}$ for some $b'$ tuple in the sequence $\bar b_{\beta}$, but we know that this is not the case by the induction hypothesis, since condition $(c)$ holds for $\beta$.\\
Continuing this construction for $|\T(\A)|^+$ steps, we find that for any $\eta<|\T(\A)|^{+}$ and any $b$ element in the sequence $\bar b_{|\T(\A)|^{+}}$, $a_{\eta} \nmodels \p \rest_{\A b}$. This is a contradiction to the generic stability of $\p$ (\emph{c.f.} \cref{fact: first things about generically stable}.(2)).\\
\underline{Step 3, Going up and down:}\\
\textit{Proof:} Let $(b'_{i})_{i \in \I}$ be a Morley sequence of $\q$ over $\A a (b_{i})_{i \in \I}$. In particular, $b'_{0} \downfree^{f}_{\A} a (b_{i})_{i\in \I}$, then $a \downfree^{f}_{\A} b_{0}'$ (this follows by monotonicity of non-forking \cref{prop: properties of non-forking}.(3) and symmetry \cref{prop: generically stable properties}.(4)). The statement of the Proposition follows by Step $2$.  
\end{proof}
\begin{lemma} \label{lem: big tail independent when there is a non-forking extension} Let $\p$ be a global $\A$-invariant type and stably dominated over $\A b$. Let $\q$ be a global $\A$-invariant extension of $\tp(b/\A)$. Let $(b_i)_{i\in \I}$ be a Morley sequence of $\q$ over $\A$ with $|\I|\geq |\T(\A)|^{+}$ and let $a \models \p\rest_{\A}$. Then there is $\J \subseteq \I$ such that $|\I \setminus \J| \leq |\T(\A)|$ and $a \downfree^{f}_{\A} b_{\J}$. 
\end{lemma}
\begin{proof}
By Proposition \ref{prop: a independent from some element in morley seq of q} there is $i \in \I$ such that $a \downfree_{\A}^{f} b_{i}$, as $\p$ is dominated over $\A b$ it is dominated over $\A b_{i}$. The sequence $b_{\I_{> i}}$ is indiscernible and non-dividing over $\A b_{i}$.  The conclusion follows by Lemma \ref{lem:independent subsequence for stably dominated types}.
\end{proof}

The following proposition can be thought of as saying that we can modify a dominating function $\f_b$ so that it \emph{factors through} a conjugate function $\f_{b'}$.

\begin{proposition}\label{prop: replacing dominating function} Let $\p$ be a global $\A$-invariant generically stable type. Assume that $\p$ is stably dominated by $\f_{b}$ over $\A b$. Let $b' \vDash \tp(b/\A)$, and $\F \colon \p \rightarrow \St_{\A b}$ be the $\A bb'$-definable function defined by sending $a \models \p\rest_{\A bb'}$ to an enumeration of $\St_{\A b}(\A b;b'\f_{b'}(a))$.  Then $\F$ dominates $\p$ over $\A b' b$. 
\end{proposition}
\begin{proof}
Let $a \models\p\rest_{\A b' b}$ and $d$ be a tuple such that $\St_{\A b}(\A b; b'd) \downfree^{f}_{\A b'b} \St_{\A b}(\A b; b'\f_{b'}(a))$ inside the stable structure $\St_{\A b}$. To simplify the notation we will write $\D$ to indicate $\St_{\A b}(\A b; b'd)$. Let $\mathrm{r}$ be a global non-forking extension of $\tp(\D/\A b \St_{\A b}(\A b; b'\f_{b'}(a)))$ and $(\D_{i})_{i \in \I }$ be a Morley sequence of $\mathrm{r}$  over $\A b \St_{\A b}(\A b;  b'\f_{b'}(a))$ such that $\D_{0}=\D$. Note that $(\D_{i})_{i \in \I}$ is a sequence of (infinite) tuples in $\St_{\A b}$.\\
Because $\St_{\A b}$ is stably embedded over $\A b$ and eliminates imaginaries, by \cref{lem: facts stably embedded}.(1) for all $i \in 2^{|\T(\A b)|^{+}}$, $\displaystyle{
\tp(\D_{i}/\A b \St_{\A b}(\A b; b'\f_{b'}(a)) \vdash \tp(\D_{i}/ \A b b' \f_{b'}(a))}$, thus:
\begin{equation}\label{eqsametype}
 \D_{i} \equiv_{\A b b' \f_{b'}(a)} \D.
 \end{equation}
 By Lemma \ref{lem: weight stable} there is some $i \in 2^{|\T(\A b)|^{+}}$ such that $\f_{b}(a) \downfree^{f}_{\A b b' \f_{b'}(a)} \D_{i}$. As $\D_i \downfree^{f}_{\A b'b} \St_{\A b}(\A b; b'\f_{b'}(a))$ (by the first line of this proof), transitivity of non-forking in the stable structure $\St_{\A b}$ implies \[\f_{b}(a) \downfree^{f}_{\A b b'} \D_{i}.\] Since $\St_{\A b}$ is stably embedded over $\A b$ and eliminates imaginaries $\f_{b}(a) \downfree^{f}_{ \St_{\A b}(\A b; b')} \D_{i}$ (\emph{c.f.} \cref{lem: facts stably embedded}.(2)). In particular,
\begin{equation}\label{eqtran1}
\f_{b}(a) \downfree^{f}_{\A b \St_{\A b} (\A b; b')} \St_{\A b}(\A b; b' \D_{i}).
\end{equation}
Because  $a \models \p_{\rest_{\A b' b}}$ we have $a \downfree^{f}_{\A b} b'$ and by \cref{lem: facts stably embedded}.(2):
\begin{equation}\label{eqtrans2}
\f_{b}(a) \downfree^{f}_{\A b} \St_{\A b}(\A b; b').
\end{equation}
 By transitivity of non-forking in stable theories together with \cref{eqtran1} and \cref{eqtrans2} one has
 \begin{equation}\label{final}
 \f_{b}(a) \downfree^{f}_{\A b}\St_{\A b}(\A b; b'\D_{i}).
 \end{equation}
 Since $\f_{b}$ dominates $\p$ over $\A b$, \cref{final} implies $a \models \p \rest_{\A b b' \D_{i}}$. In particular, $a \downfree_{\A b'} \D_{i} b$, and by \cref{lem: facts stably embedded}.(2) we have $\f_{b'}(a) \downfree^{f}_{\A b'} \St_{\A b'}(\A b';\D_{i} b)$ inside the stable structure $\St_{\A b'}$. Since $\D_{i} \equiv_{\f_{b'}(a) \A bb'} \D$ (\emph{c.f.} \cref{eqsametype}), then $\f_{b'}(a) \downfree^{f}_{\A b'} \St_{\A b'}(\A b';\D b)$. 
 Since $\f_{b'}$ dominates $\p$ over $\A b'$ , this implies that $a \models \p \rest_{\A b' b \D}$ . Then $\f_{b}(a) \downfree^{f}_{\A b} \D$, as $\D=\St_{\A b}(\A b; b'd)$ and by domination of $\p$ over $\A b$ we conclude that $a \models \p\rest_{\A b b' d}$. 
\end{proof}

Let us give some intuition for the next step of the proof. Recall that given a dominating function $\f_b$, our goal is to find another one that does not depend on $b$. We have seen in the previous proposition, that we can in some sense factorize $\f_b$ through a conjugate $\f_{b'}$. It is then natural to push this further. Say that we had access to a Morley sequence $(b_i)_{i<\kappa}$ of $\tp(b/\A)$ (for instance if $\tp(b/\A)$ extends to an invariant type). Then we would have that $\f_b$ can factor through any $\f_{b_i}$. So we can hope that it would factor through something of the form $a \mapsto \bigcap_i \dcl(\A b_i a)$. If the sequence $(b_i)_i$ is sufficiently independent, we can hope that this in turns descends to a function over $\A$. The details are more complicated because we need to keep track of the base, juggle between the different stable structures, and we cannot literally take the intersection of $\dcl$, but use canonical bases of indiscernible sequences instead.

\begin{proposition}\label{prop: domination by the canonical base}
Let $\p$ be a global generically stable $\A$-invariant type. Assume that $\p$ is stably dominated over $\A b$ via the function $\f_{b}$. Let $\bar{b}=(b_{i})_{i \in \I}$ be an $\A$-indiscernible sequence where $b_{0}=b$ and $|\T(\A)|^{+} \leq |\I|$. Let $a \models \p\rest_{\A}$, and assume that:
\begin{itemize}
\item There is $\J \subset \I$ such that $|\I \backslash \J| \leq |\T(\A)|$ and $a \downfree^{f}_{\A} b_{\J}$.
\item The sequence $b_{\I_{>0}}$ is $\acl$-independent over $\A b_{0}$.
\end{itemize}
Then:
\begin{enumerate}
\item For each $i \in \J$, $ \{\St_{\A b_{i}}(\A b_{i}; b_{j} \St_{\A b_{j}}(\A b_{j};a)) \ | \ j \in \J_{> i}\}$ is a Morley sequence in the structure $\St_{\A b_{i}}$  over $\A b_{i} a$.
\item Let $\displaystyle{\C=\Cb(\{\St_{\A b_{i}}(\A b_{i};b_{j} \St_{\A b_{j}}(\A b_{j};a) \ | \ j \in \J_{>i}\})}$ (see \cref{def: canonical base sequence}). Then $\St_{\A b_{i}}(\A b_{i}; b_{\J_{>i}}) \downfree_{\A b_{i}}^{f} \C$ inside the stable structure $\St_{\A b_{i}}$.
\item The tuple $a$ is dominated over $\A b_{\J_{\geq i}}$ by $\C$;  i.e. if $d \in \M$ and $d \downfree^{f}_{\A b_{\J_{\geq i}}} \C$ then $a \downfree^{f}_{\A b_{\J_{\geq i}}} d$.
\end{enumerate}
In particular, if $\q$ is some global invariant extension of $\tp(b/ \A)$ and $(b_{i})_{i \in \I}$ is a Morley sequence in $\q$ where $|\I| \geq |\T(\A)|^{+}$ then the conclusions $(1)$, $(2)$ and $(3)$ hold for some $J$. 
\end{proposition}
\begin{proof}
\begin{enumerate}
%\item The first statement is the conclusion of \cref{lem: big tail independent when there is a non-forking extension}.
\item Let $i \in \J$. Then the sequence $b_{\J_{>i}}$ is indiscernible and $\acl$-independent over $\A b_{i}$. Fix $k \in \mathbb{N}$ and let $(\bar b_{\ell})_{\ell \in \K}$ be the sequence of finite tuples of length $k$ obtained by grouping $k$-terms of $b_{\J_{>i}}$. The sequence $(\bar b_{\ell})_{\ell \in \K}$ is still indiscernible and $\acl$-independent over $\A b_{i}$.

\smallskip
\underline{Claim $1$}: \emph{The sequence $(a \bar b_{\ell})_{\ell \in \K}$ is indiscernible  and $\acl$-independent over $\A b_{i} a$. Furthermore the sequence $\{ \St_{\A b_{i}}(\A b_{i}; \bar{b}_{\ell} a) \ | \ \ell \in \K \}$ is a Morley sequence over $\A b_{i}a$ in the stable structure $\St_{\A b_{i}}$.}\\
 Because $a \downfree_{\A}^{f} b_{\J}$, by monotonicity base monotonicity of non-forking (\cref{prop: properties of non-forking}.(3) and (4)) $a \downfree_{\A b_{i}}^{f} b_{\J_{>i}}$ and in particular $a \downfree_{\A b_{i}}^{f }(\bar b_{\ell})_{\ell \in \K}$.  Thus we can apply \cref{lem: staying indiscernible} (to the base $\A b_{i}$ and the sequence $(\bar{b}_{\ell})_{\ell \in \K}$), and we conclude that the sequence $(a \bar{b}_{\ell})_{\ell \in \K}$ is indiscernible and $\acl$-independent over $\A b_{i} a$. By \cref{lem: acl-independence implies Morley in St} the sequence $\{\St_{\A b_{i}}(\A b_{i}, \bar b_{\ell} a) \ | \  \ell \in \K\}$ is a Morley sequence in the stable structure $\St_{\A b_{i}}$ over $\A b_{i}a$ (in the statement take $\A b_{i}$ as the base and $\A b_{i} a$ as the larger set of parameters $\B$). This concludes the proof of the claim.\\

Taking $k=1$, the previous claim shows that $(\St_{\A b_{i}}(\A b_{i}; b_{j}a) \ | \ j \in \J_{>i})$ is a Morley sequence over $\A b_i a$. Note that for each $j \in \J_{>i}$, $\St_{\A b_{i}}(\A b_{i}; b_{j} \St_{\A b_{j}}(\A b_{j}; a) )\subseteq \St_{\A b_{i}} (\A b_{i}; b_{j} a)$, then 
\begin{equation*}
\{\St_{\A b_{i}}(\A b_{i}; b_{j}\St_{\A b_{j}}(\A b_{j}; a)) \ | \ j \in \J_{>i}\}
\end{equation*}
is also a Morley sequence over $\A b_{i} a$.\\
\item \underline{Claim $2$:} \emph{$\C \subseteq \dcl(\A b_{i} a)$. }\\
Let $\I_{0}$ be a fixed finite subset of $\J_{>i}$. For each $j \in \I_{0}$ note that
\begin{equation*}
\St_{\A b_{i}} (\A b_{i}; b_{j} \St_{\A b_{j}}(\A b_{j}; a)) \subseteq \St_{\A b_{i}}(\A b_{i}; b_{j} a) \subseteq \St_{\A b_{i}}(\A b_{i}; b_{\I_{0}} a).
\end{equation*}
Consequently, 
\begin{equation}\label{eq:detail}
\dcl(\A b_{i}a; \{ \St_{\A b_{i}}(\A b_{i}; b_{j} \St_{\A b_{j}}(\A b_{j}; a))\}_{j \in \I_{0}}) \subseteq \dcl(\A b_{i} a; \St_{\A b_{i}}(\A b_{i};b_{\I_{0}} a)).
\end{equation}
By the first claim for each $k \in \mathbb{N}$, the sequence
$\{\St_{\A b_{i}}(\A b_{i}, \bar{b}_{\ell} a) \ | \ \ell \in \K\}$ is a Morley sequence over $\A b_{i} a$, in particular it is $\acl$-independent over $\A b_{i} a$. By \cref{rem: equivalence for usual dcl and acl}, it is $\dcl$-independent over $\A b_{i} a$. This together with \cref{eq:detail} implies that $\C \subseteq \dcl(\A b_{i} a)$, and this concludes the proof of the second claim.\\

Given $k \in \mathbb{N}$, the sequence $(\bar b_{\ell})_{\ell \in \K}$ defined above is still $\acl$-independent and indiscernible over $\A b_{i}$. By Lemma \ref{lem: acl-independence implies Morley in St} $(\St_{\A b_{i}}(\bar{b}_{\ell}))_{\ell \in \K}$ is Morley over $\A b_{i}$.  Note that it is $\C$-indiscernible since $\St_{\A b_{i}}(\A b_{i}; \bar b_{\ell}) \subseteq \St_{\A b_{i}} (\A b_{i}; \bar{b}_{\ell} a)$ and by Claim $(1)$ and $(2)$ $\displaystyle{\{ \St_{\A b_{i}}(\A b_{i}; \bar{b}_{\ell} a) \ | \ \ell \in \K\}}$ is indiscernible over $\C$. 
Thus $\{\St_{\A b_{i}}(\A b_{i}; \bar b_{\ell}) | \ \ell \in \K \}$ is Morley over $\A b_{i}\C$.  
Consequently, $\St_{\A b_{i}}(\A b_{i};b_{\J_{>i}}) \downfree^{f}_{\A b_{i}} \C$ inside $\St_{\A b_{i}}$. 
\item \underline{Claim $3$:} \emph{Let $e$ be a (possibly infinite) tuple of $\M$, such that $e \downfree^{f}_{\A b_{i}} \C$. Then $a \models \p \rest_{\A b_{i} e}$.}\\
\textit{Proof:} It is enough to prove this for every finite tuples of $e$, so assume $e$ is a finite tuple. By Lemma \ref{lem: weight stable}  $\St_{\A b_{i}}(\A b_{i};e) \downfree^{f}_{\A b_{i} \C } \St_{\A b_{i}}(\A b_{i};b_{j}\St_{\A b_{j}}(\A b_{j};a))$ for some $j \in \J_{>i}$. On the other hand, because $e \downfree_{A b_{i}}^{f} \C$, we have $\displaystyle{\St_{\A b_{i}}(\A b_{i};e) \downfree^{f}_{\A b_{i}} \C}$. By transitivity $\St_{\A b_{i}}(\A b_{i};e) \downfree^{f}_{\A b_{i}} \C \St_{\A b_{i}}(\A b_{i}; b_{j}\St_{\A b_{j}}(\A b_{j};a))$. By monotonicity, $\St_{\A b_{i}}(\A b_{i}; e) \downfree^{f}_{\A b_{i}}  \St_{\A b_{i}}(b_{j} \St_{\A b_{j}}(\A b_{j}; a))$. Since $\St_{\A b_{i}}$ is stably embedded and eliminates imaginaries  (\emph{c.f.} \cref{lem: facts stably embedded}.(2)) this implies that 
\begin{equation*}
\St_{\A b_{i}}(\A b_{i};e) \downfree^{f}_{\A b_{i}}  \St_{\A b_{i}}(b_{j} \St_{\A b_{j}}(\A b_{i};a)) b_{j}.
\end{equation*}
By base monotonicity (\emph{c.f.} \cref{prop: properties of non-forking}.(4)), we have:
\begin{equation*}
\St_{\A b_{i}}(\A b_{i};e) \downfree^{f}_{\A b_{i} b_{j}} \St_{\A b_{i}}(b_{j},\St_{\A b_{j}}(\A b_{j};a)) \ \text{in the structure $\St_{\A b_{i}}$}.
\end{equation*}
By Proposition \ref{prop: replacing dominating function}, $a \models \p \rest_{\A b_{i}b_{j} e}$, in particular $a \models \p \rest_{\A b_{i} e}$ as required.This concludes the proof of the third claim.

\medskip
Let now $d \in \M$ and assume $d \downfree^{f}_{\A b_{\J_{\geq i}}} \C$. We aim to show that $d \downfree^{f}_{\A b_{\J_{\geq i}}} a$. Since  $d \downfree^{f}_{\A b_{i} b_{\J_{> i}}} \C$ then $\St_{\A b_{i}}(\A b_{i}; d b_{\J_{>i}}) \downfree^{f}_{\St_{\A b_{i}}(\A b_{i};b_{J_{>i}})} \C$ in the stable structure $\St_{\A b_{i}}$ (\emph{c.f.} \cref{lem: facts stably embedded}.(2)). Then by transitivity $\St_{\A b_{i}}(\A b_{i}; d b_{\J_{>i}}) \downfree^{f}_{\A b_{i}} \C$, as by the third statement $\St_{\A b_{i}}(\A b_{i}, b_{\J_{>i}}) \downfree^{f}_{\A b_{i}} \C$ . As $\St_{\A b_{i}}$ is stably embedded and eliminates imaginaries, by \cref{lem: facts stably embedded}.(2), we have $d b_{\J_{>i}} \downfree^{f}_{\A b_{i}} \C$. By the third claim $a \models \p \rest_{\A b_{i}  b_{\J_{>i}} d}$  thus $a \models \p \rest_{\A b_{\J_{\geq i}}d}$ as required. 
\end{enumerate}
\medskip
Assume now that $\q$ is a global non-forking extension of $\tp(b/ \A)$ and let $(b_{i})_{i \in \I}$ be a Morley sequence of $\q$ over $\A$. The first hypothesis of the statement is given by \cref{lem: big tail independent when there is a non-forking extension}. For the second requirement, for $k \in \mathbb{N}$ let $(\bar{d}_{\ell})_{\ell \in \K}$ be the sequence of finite of length $k$ obtained by grouping $k$-elements of $b_{\I_{>0}}$.  By \cref{prop: properties of non dividing}, the sequence $(\bar d_{\ell})_{\ell \in \K}$ is still non dividing over $\A b_{0}$, in particular it is $\acl$-independent over $\A b_{0}$.
\end{proof}

\begin{notation}\label{not notation almost for all} Let $\p$ be a global generically stable $\A$-invariant type. Assume that $\p$ is stably dominated over $\A b$ via the function $\f_{b}$. Let $\bar{b}=(b_{i})_{i \in \I}$ be an $\A$-indiscernible sequence where $b_{0}=b$ and $|\T(\A)|^{+} \leq |\I|$. Let $a \models \p\rest_{\A}$, and assume that:
\begin{itemize}
\item There is $\J \subset \I$ such that $|\I \backslash \J| \leq |\T(\A)|$ and $a \downfree^{f}_{\A} b_{\J}$.
\item The sequence $b_{\I_{>0}}$ is $\acl$-independent over $\A b_{0}$.
\end{itemize}

\begin{enumerate}
\item  We write and $\f(\bar{b},a)_{i}$ to indicate an enumeration of $\Cb(\St_{\A b_{i}}(b_{j}\St_{\A b_{j}}(a)) \ | \ j \in \J_{>i})$. Furthermore, given a projection $\pi$ of $\f(\bar{b},a)_{i}$ to a finite sub-tuple, we write $\f_{\pi}(\bar{b},a)_{i}$ to denote the image $\pi(\f(\bar{b},a)_{i})$. (This makes sense by \cref{prop: domination by the canonical base}). 
\item  Let  $P$ be a property. We say that \emph{$P$ holds for almost all $i \in \I$} and write $\forall_{i}^{a} P$ if 
\begin{equation*}
|\{ i \in \I \ | \ \neg P(b_{i}) \ \text{holds}\}| \leq |\T(\A)|. 
\end{equation*}
\end{enumerate}
\end{notation}
\subsubsection{The main proof assuming $\tp(b/\A)$ has a global $\A$-invariant extension $\q$}
In this subsection we go through the second step of the proof of \cref{thm: descent with invariant extension}. 
\begin{lemma}\label{lem: almost all are the same or different} Let $\p$ and $\q$ be as in \cref{thm: descent with invariant extension} and $\bar{b}=(b_{i})_{i \in \I}$ a Morley sequence in $\q$ over $\A$ with $|\I|\geq |\T(\A)|^{+}$. Let $a,a'\models \p \rest_{\A}$. Then either:
\begin{itemize}
\item for almost all $i \in \I$,   $a, a'\models p \rest_{\A b_{i}} \text{ and } \f(\bar b, a)_{i} = \f (\bar b, a')_{i}$ or
\item for almost all $i \in \I$,   $a,a'\models p \rest_{\A b_{i}}, \text{ and } \f(\bar b, a)_{i} \neq \f (\bar b, a')_{i}$.
\end{itemize}
A similar statement holds for $\f_{\pi}$ where $\pi$ is some fixed projection on some finite sub-tuple. 
\end{lemma}
\begin{proof}
By Lemma \ref{lem: big tail independent when there is a non-forking extension}, there is some $\J \subseteq \I$ such that $|\I \setminus \J| \leq |\T(\A)|$ and for any $i \in \J$ both $a$ and $a'$ realize $\p\rest_{\A b_{i}}$. Let $\bar b_*$ be a Morley sequence of $\q$ over $\A \bar{b} a a'$. By Fact \ref{fact: facts canonical base sequence}$(1)$ $\f(\bar b, a)_i \subseteq \dcl(\A b_i;\bar b_* \St_{\A \bar b_*}(\A \bar{b}_{*};a))$ and $\f(\bar b, a ')_i \subseteq \dcl(\A b_i;\bar b_* \St_{\A \bar b_*}(\A \bar b_{*}; a'))$. Since $\St_{\A \bar b_*}$ is stably embedded over $\A \bar b_*$, the property $\f(\bar b, a)_i =\f (\bar b, a')_{i}$ only depends on $\tp(\St_{\A \bar b_*}(\A \bar{b}_{*}; b_i)/\St_{\A \bar b_*}(\A \bar b_{*}; aa'))$ (\emph{c.f.} \cref{lem: facts stably embedded}.(1)).

The sequence $(\St_{\A \bar b_*}(\A \bar b_{*}; b_j))_{j \in \J}$ is indiscernible over $\A \bar{b}_{*}$ in the stable structure $\St_{\A \bar b_*}$, thus after removing at most $|\T(\A)|$ elements from the sequence $\bar b=(b_j)_{j \in \J}$, we may assume that it is indiscernible over $\St_{\A \bar b_*}( \A \bar b_{*}; aa')$. The result follows. The same argument applies to $\f_{\pi}$. 
\end{proof}
\begin{corollary}\label{cor: reducing to n} Let $\p$ and $\q$ as in \cref{thm: descent with invariant extension} and $\bar{b}=(b_{i})_{i \in \I}$ be a Morley sequence in $\q$ over $\A$ with $|\I|\geq|\T(\A)|^{+}$. Let $\pi$ be the projection of $\f(\bar b,a)_{i}$ on some finite sub-tuple, then there is an $n < \omega$ such that either $\f_{\pi}(\bar b, a)_{i}= \f_{\pi}(\bar b, a')_{i}$ for at most $n$ values of $i \in \I$ or $\f_{\pi}(\bar b, a)_{i}\neq \f_{\pi}(\bar b, a')_{i}$ for at most $n$ values of $i \in \I$. 
\end{corollary}
\begin{proof}
 Otherwise, since $\p$ is $\A$-definable by compactness we can find $\bar b'=(b'_{j})_{j \in |\T(\A)|^{+}}$ a Morley sequence in $\q$ over $\A$  such that 
 \begin{align*}
 \B_{0}&= \{ j< |\T(\A)|^{+} \ | \ a,a' \models \p \rest_{\A b_{j}} \ \text{and} \  \ \f_{\pi}(\bar b', a)_{j}= \f_{\pi}(\bar b', a')_{j}\}, \ \text{and}\\
 \B_{1}&= \{ j< |\T(\A)|^{+} \ | \ a,a' \models \p \rest_{\A b_{j}} \ \text{and} \  \ \f_{\pi}(\bar b', a)_{j}\neq \f_{\pi}(\bar b', a')_{j}\} 
 \end{align*}
 have both size $|\T(\A)|^{+}$. 
 This contradicts Lemma \ref{lem: almost all are the same or different}.  
\end{proof}
\begin{lemma}\label{lem: definability of almost everywhere both sequences agree}Let $\p$ and $\q$ as in \cref{thm: descent with invariant extension}. Let $\R$ be the $\A$-type-definable set  
\begin{equation*}
\{(a,a', \bar{b}) \ | \ a,a'\models\p\rest_{\A} \  \text{and} \ \bar b= (b_{i})_{i \in |\T(\A)|^{+}} \ \text{is a Morley sequence of} \ \q \  \text{over} \ \A\}. 
\end{equation*}
Let $\pi$ be some fixed projection of $\f(\bar{b},a)_{i}$ on some finite sub-tuple. The statement \emph{$ \forall^{a}_{i} \f_{\pi}(\bar b, a)_{i}= \f_{\pi}(\bar b, a')_{i}$} is a definable condition for $(a,a',\bar b) \in \R$.
\end{lemma}
\begin{proof}
  Since $\p$ is $\A$-definable we can apply compactness and Corollary \ref{cor: reducing to n} to show the existence of some  $n< \omega$ such that for any $a,a'\models\p\rest_{\A}$ and every Morley sequence $\bar b=(b_{i})_{i \in |\T(\A)|^{+}}$ in $\q$ over $\A$ either $ \f_{\pi}(\bar b, a)_{i}= \f_{\pi}(\bar b, a')_{i}$ holds for at most $n$-elements or  $\f_{\pi}(\bar b, a)_{i}\neq \f_{\pi}(\bar b, a')_{i}$ holds for at most $n$-elements.  Let $\F=\{ 1,\dots,2n+1\}$. Then 
 
 \begin{align*}
 \forall^{a} i \ \f_{\pi}(\bar b, a)_{i}&= \f_{\pi}(\bar b, a')_{i} \iff \bigvee_{\B \subseteq \F , |\B|= n+1} \bigwedge_{i \in \B} \f_{\pi}(\bar b, a)_{i}= \f_{\pi}(\bar b, a')_{i}. \\ 
%  \forall^{a} i \ \f_{\pi}(\bar b, a)_{i}&\neq \f_{\pi}(\bar b, a')_{i} \iff \bigvee_{\B \subseteq \F, |\B|= n+1} \bigwedge_{i \in \B} \f_{\pi}(\bar b, a)_{i}\neq \f_{\pi}(\bar b, a')_{i}.
 \end{align*}
Thus both conditions and its complement are type-definable when restricted to $(a,a', \bar{b}) \in \R$, so they are definable. 
\end{proof}
\begin{definition}\label{def:same germs} Let $\p$ and $\q$ as in \cref{thm: descent with invariant extension}. Let $a,a'\models\p\rest_{\A}$, and let $\bar b= (b_i)_{i\in \I}$ be a Morley sequence of $\q$ over $\A$  where $|\I|\geq|\T(\A)|^{+}$.  Fix a projection $\pi$ of $\f(\bar b, a)_{i}$ to a finite sub-tuple. We say that \emph{$a$ and $a'$ have the same $\pi$-germ on $\bar b$ if $\f_{\pi}(\bar b, a)_{i}=\f_{\pi}(\bar b, a')_{i}$ for almost all $i \in \I$}.
\end{definition}
\begin{lemma}\label{lem:same germs} Let $\p$ and $\q$ as in Theorem \ref{thm: descent with invariant extension}. Let $a,a'\models\p\rest_{\A}$, $\bar b=(b_i)_{i\in \I}$, $\bar b'=(b_{j})_{j \in \I'}$ be Morley sequences of $\q$ over $\A$  where $|\I|,|\I'|\geq |\T(\A)|^{+}$. Assume that $a$ and $a'$ have the same $\pi$-germ on $\bar b$, then they have the same $\pi$-germ on $\bar b'$.
\end{lemma}
\begin{proof}
    Let $\bar b_{*}=(b^{*}_{k})_{k \in \K}$ be a Morley sequence of $\q$ over $\A \bar b \bar b' a a'$, where $|\K| \geq |\T(\A)|^{+}$. By Fact \ref{fact: facts canonical base sequence} (1) $\f(\bar b, a)_{i}= \f(\bar b+\bar b_{*},a)_{i}$  for all $i \in \I$, and similarly for $a'$. Since $a$ and $a'$ have the same $\pi$-germ over $\bar b$ then $\f_{\pi}(\bar{b},a)_{i}=\f_{\pi}(\bar{b},a')_{i}$ for almost all $i \in \I$, and therefore
     $ \f_{\pi}(\bar b+ \bar b_{*}, a)_i = \f_{\pi}(\bar b+ \bar b_{*},a')_i$ for almost all $i \in \I$. By Lemma \ref{lem: almost all are the same or different}, $\f_{\pi}(\bar b+\bar b_{*},a)_j = \f_{\pi}(\bar b+\bar b_{*},a')_j$ holds for almost all $j \in \I' \cup \K$. In particular  $|\{ k \in \K \ | \ \f_{\pi}(\bar b_{*},a)_k \neq \f_{\pi}(\bar b_{*},a')_k \}| \leq |\T(\A)|$, since  $\f_{\pi}(\bar b_{*},a)_k= \f_{\pi}(\bar b+\bar b_{*},a)_k$ for $k \in \K$ by the definition $\f$; and similarly for $a'$. Likewise  $\f_{\pi}(\bar b'+ \bar b_{*}, a)_{k}= \f_{\pi}(\bar b_{*}, a)_{k}$ for $k \in \K$ and for $a'$ as well.
     
    By Lemma \ref{lem: almost all are the same or different}, $\f_{\pi}(\bar b'+\bar b_{*},a)_j = \f_{\pi}(\bar b'+\bar b_{*},a')_j$ holds for almost all $j \in \I' \cup \K$. Hence $|\{ j \in \I' \ | \ \f_{\pi}(\bar b'+\bar b_{*},a)_j \neq \f_{\pi}(\bar b'+\bar b_{*},a')_j \}| \leq  |\T(\A)|$, and by Fact \ref{fact: facts canonical base sequence} (1) $\f(\bar b'+ \bar b_{*}, a)_{j}= \f(\bar b', a)_{j}$ for $j \in \I'$ and correspondingly for $a'$. Consequently, $a$ and $a'$ have the same $\pi$-germ on $\bar b'$. 
\end{proof}
 \begin{definition}\label{def: Q and relations} Let $\p$ and $\q$ as in \cref{thm: descent with invariant extension}. Let
 \begin{align*}
 \Q =&\{ \tp(\bar b,\bar d/ \A) \ |  \bar b= (b_{i})_{i < |\T(\A)|^{+}} \ \text{is a Morley sequence of $\q$ over $\A$}  \\
&\text{and there is}  \ a'\vDash \p\rest_{\A} \ \text{such that} \ d_{i}= \f(\bar b, a')_{i}\}.
\end{align*}
For each fixed projection $\pi$ on a finite sub-tuple, let $\E_{\pi}$ be the relation on instances of $\Q$ defined by  \[ (\bar b, \bar d) \E_{\pi} (\bar b', \bar d') \iff (\exists a\models \p\rest_{\A}) \big(\forall^a i \ \f_{\pi}(\bar b, a)_i = \pi(d_i) \wedge \f_{\pi}(\bar b',a)_i = \pi(d'_i)\big).\]
\end{definition}
By Lemma \ref{lem: definability of almost everywhere both sequences agree} and since $p$ is a definable type, $\E_{\pi}$ is a definable relation.

By Lemma \ref{lem:same germs}, the relation $\E_{\pi}$ can also be defined by 
\[ (\bar b, \bar d) \E_{\pi} (\bar b', \bar d') \iff (\forall a\models \p \rest_{\A}) (\forall^a i \ \f_{\pi}(\bar b, a)_i = \pi(d_i)) \to (\forall^a i \ \f_{\pi}(\bar b',a)_i = \pi(d'_i)).\]
Thus $\E_{\pi}$ is an equivalence relation on $\Q$. By compactness, there is an $\La(\A)$-definable set $\X$ containing $\Q$ such that $\E_{\pi}$ is an equivalence relation on $\X$. We let $\Lat_{\pi}$ be the $\La(\A)$-definable quotient $\X/ \E_{\pi}$ and let $g_\pi: \X \to \Lat_{\pi}$ be the quotient map.\\

 \begin{proof}[Proof of Theorem \ref{thm: descent with invariant extension}] \  \\
 \underline{Claim $1$:}
  \emph{Let $\bar b=(b_{i})_{i< |\T(\A)|^{+}}$ be a Morley sequence of $\q$ over $\A$. For each finite projection $\pi$, every class of $\Q/ \E_{\pi}$ admits a representative of the form $(\bar b,\bar d) \in \Q$.}\\
 \textit{Proof:} Let $(\bar b',\bar d')$ be an element of $\Q$. Take $a\models \p \rest_{\A}$ such that $\f(\bar b',a)_i = d'_i$ for almost all $i$. Then $(\bar b, \f(\bar b,a))\in \Q$ and is $\E_{\pi}$-equivalent to $(\bar b',\bar d')$. \\

Fix $\bar b$ as in the claim. The set $V_{\bar b} = \{\bar d : (\bar b, \bar d) \in \Q\}$ is a pro-type-definable subset of $\St_{\A \bar b}$. It admits an $\A\bar b$-pro-definable map $h_{\bar b}$ to $\Lat_{\pi}$ whose image contains $\Q/ \E_{\pi}$. By compactness, the map $h_{\bar b}$ only depends on finitely many variables, hence is defined on a definable $U_{\bar b} \subseteq \St_{\A \bar b}$ containing $V_{\bar b}$. Hence the image $W_{\bar b} \subseteq \Lat_{\pi}$ of $U_{\bar b}$ under the definable map $h_{\bar b}$ is an $\A\bar b$-definable set containing $\Q/ \E_{\pi}$, and it is stable and stably embedded. Let $\psi(x,\bar b)$ be a formula defining $W_{\bar b}$, which we can take so that $(\forall \bar y) \psi(x,\bar y) \to (x \in \Lat_{\pi})$.

By the claim, the following type-definable conditions on a tuple $(c,\bar b)$ are inconsistent:

\begin{itemize}
\item $c\in \Q/ \E_{\pi}$;
\item $\bar b$ is a Morley sequence of $\q$;
\item $\neg \psi(c, \bar b)$.
\end{itemize}

By compactness, there is some $\A$-definable condition $\eta(\bar b)$ and some $\A$-definable $\Lat'_{\pi} \subseteq \Lat_{\pi}$ containing $\Q/ \E_{\pi}$ such that: \[\eta(\bar b) \wedge c\in \Lat'_{\pi} \to \psi(c, \bar b).\]

Now consider the definable subset $\Lat^*_{\pi}$ of $\Lat'_{\pi}$ defined by \[\theta(x) \equiv (\forall \bar y) \eta(\bar y) \to \psi(x, \bar y).\]

The set $\Lat^*_{\pi}$ is $\A$-definable. Furthermore, for any $\bar b$ as above, it is included in the $\A\bar b$-definable set $W_{\bar b}$, which we know to be stable and stably embedded. Therefore $\Lat^*_{\pi}$ is stable and stably embedded.

Let $\Lat^* = \sqcup \Lat^*_{\pi}$ for $\pi$ ranging over all projections to a finite sub-tuple. Then $\Lat^*$ is in $\St_\A$. Let $\alpha$ be the pro-definable function that enumerates $\dcl(\A a) \cap \Lat^*$ for $a\models \p \upharpoonright_{\A}$.

\underline{Claim $2$:} \emph{The $\A$ pro-definable map $\alpha \colon \p \rightarrow \Lat^*$ dominates $\p$ over $\A$.}\\
\textit{Proof:} Let $a\models\p\rest_{\A}$ and assume that $\alpha(a)\downfree^{f}_{\A} e$. We aim to show that $a\downfree^{f}_{\A}e$. Let $\bar{b}=(b_{i})_{i \in |\T(\A)|^{+}}$ be a Morley sequence in $\q$ over $\A$ such that $\alpha(a) \downfree^{f}_{\A} e \bar b$. By Lemma \ref{lem: big tail independent when there is a non-forking extension} there is some $\J$ such that $a \downfree^{f}_{\A } \bar{b}_{\J}$ and $||\T(\A)|^{+} \backslash \J| \leq |\T(\A)|$. In particular, since $\alpha(a) \downfree_{\A}^{f} e \bar b$ by monotonicity and base monotonicity (\cref{prop: properties of non-forking}.(3) and (4)) we have:
\begin{equation}\label{eqlast1}
\alpha(a) \downfree^{f}_{\A \bar b_{\J}} e. 
\end{equation}
By compactness and Claim $1$ we can find $\bar{d}$ such that for each finite projection $\pi$, $\pi(\bar{b}_{\J}, \bar{d})=\alpha(a)\cap \Lat_{\pi}$. Since $\bar{d}\subseteq \St_{\A \bar{b}_{\J}}$, the type $\tp(\bar{d}/ \A \bar{b}_{\J} \alpha(a))$ is a stable type. As any set in a stable theory is an extension basis, we may further assume that 
\begin{equation}\label{eqlast2}
\bar{d} \downfree^{f}_{\A \bar{b}_{\J} \alpha(a)} e.
\end{equation}
By \cref{eqlast1}, \cref{eqlast2} and left transitivity (\emph{c.f.} \cref{prop: properties of non-forking}.(5)) $\bar{d} \alpha(a) \downfree^{f}_{\A \bar{b}_{\J}} e$ thus $\bar{d} \downfree^{f}_{\A \bar{b}_{\J}} e$. By Proposition \ref{prop: domination by the canonical base}, $a \downfree^{f}_{\A \bar{b}_{\J}} e$. Since $a \downfree_{\A} \bar{b}_{\J}$ by
right transitivity $a \downfree^{f}_{\A} \bar{b_{\J}} e$ (\emph{c.f.} \cref{prop: generically stable properties}.(5)). In particular,  $a \downfree^{f}_{\A} e$. This concludes the proof of Theorem \ref{thm: descent with invariant extension}.\end{proof}

\subsection{The general case}

We now move to the proof of \cref{thm:descent}, which follows closely the argument in Theorem \ref{thm: descent with invariant extension} in the previous section. Assume that $\p$ is a global $\A$-invariant type that is stably dominated over $\A b$, we aim to show that $\p$ is stably dominated over $\A$.
\begin{remark}\label{rem:results} We summarize a few results that we had already obtained that will be used in this section.  
\begin{enumerate}
\item By Proposition \ref{prop: basic facts stable domination}.(1) we may assume that the base is algebraically closed i.e. $\A=\acl(\A)$.
\item By \cref{prop: domination by the canonical base} $\p^{\tensor n}$ is generically stable for all $n< \omega$.  
\item By \cref{cor: total sequence}, there is an $\A$-indiscernible sequence $\bar{b}_{0}=(b_{i})_{i \in |\T(\A)|^{+}}$ satisfying the following conditions:
\begin{itemize}
\item $b_{i} \downfree_{\A}^{\p} b_{<i}$ for all $i$,
\item $\bar{b}_{0}$ is $\acl^{\vee}$-independent over $\A$, i.e. for any $\acl^{\vee}(\A b_{\I}) \cap \acl^{\vee}(\A b_{\J})$ for all finite sets $\I, \J$ with $\I \cap \J=\emptyset$.
\item $\bar{b}_{0}$ is $\acl$-independent over $\A b_{0}$, i.e. for all $0<\I< \J \subseteq \T(\A)^{+}$ we have $b_{\I} \downfree_{\A b_{0}}^{\acl} b_{\J}$. 
\end{itemize}
\end{enumerate}
\end{remark}

\begin{notation}\label{not: sequences of interest} Through the entire section we fix $\p$ a global $\A$-invariant type that is stably dominated over $\A b_0$. Let $\bar{b}_{0}$ be the sequence given by \cref{rem:results}.$(3)$. Through this section we will work with $\A$-indiscernible sequences $\bar b=(b_{i})_{i \in \I}$ where $|\I| \geq |\T(\A)|^{+}$ such that $\EMtype(\bar b/ \A)= \EMtype(\bar b_{0}/ \A)$ instead of Morley sequences of $\q$ over $\A$. We will write $\bar{b}\equiv_{\EM(\A)} \bar{b}_{0}$ to indicate $\EMtype(\bar b/ \A)= \EMtype(\bar b_{0}/ \A)$.
\subsubsection{Extracting a large subsequence independent from $a \models \p \rest_{\A}$ in the general case}
In the following lemma we tackle the first step in the proof; it corresponds to a generalization of \cref{lem: big tail independent when there is a non-forking extension}.
\end{notation}
\begin{lemma}\label{lem: independent over big tail p^{n} generically stable} Let $\p$ be a global $\A$-invariant type. Let $(b_{i})_{i \in |\T(\A)|^{+}}$ be an $\A$-indiscernible sequence such that:
\begin{itemize}
\item  for all $i$ $b_{>i} \downfree_{\A}^{\p} b_{\leq i}$;
%\item $(b_{i})_{i \in \I}$ is $\acl^{\vee}$-independent over $\A$,
\item 
$(b_{i})_{i >0}$ is $\acl$-independent over $\A b_{0}$. 
\end{itemize}

Assume $\p$ is stably dominated over $\A b_{0}$. Then there is a subset $\J \subseteq \I$ such that $|\I \setminus \J| \leq |\T(\A)|$ and such that $a \downfree^{f}_{\A} b_{J}$. 
\end{lemma}
\begin{proof}
By Ramsey and compactness we may extend the sequence to a sequence $(b_{i})_{i \in \I_{0}+\I}$ where $|\I_{0}| \geq |\T(\A)|^{+}$. By \cref{rem:results}.(2) $\p^{\tensor n}$ is generically stable for all $n< \omega$.  Let $a \models \p\rest_{\A}$. Let $\pi_{\p}(x)$ be the generically stable partial type given by \cref{prop: basic facts on p rel} such that for any $e$, $b_{i} \models \pi_{\p}(x) \rest_{\A e}$ if and only if $b_{i} \downfree_{\A}^{\p} e$. We then have $b_i \models \pi_{\p} \rest_{\A b_{<i}}$ for every $i$ and by generic stability, there is some $i_{*} \in \I_{0}$ such that $b_{i_*} \models \pi_{\p} \rest_{\A a}$. By definition of $\downfree^{\p}$, this implies that $a \models \p\rest_{\A b_{i_{*}}}$, in particular 
\begin{equation}\label{eqind}
a \downfree_{\A}^{f} b_{i_{*}}.
\end{equation}

 The sequence $(b_{i})_{i \in \I}$ is indiscernible and $\acl$-independent over $\A b_{i_{*}}$ and by assumption $\p$ is stably dominated over $\A b_{i_{*}}$. By Lemma \ref{lem:independent subsequence for stably dominated types}, there is a subset $\J$ such that $|\I \setminus \J| \leq |\T(\A)|$ and such that $a \downfree_{\A b_{i_{*}}}^{f} b_{\J}$. By right transitivity and \cref{eqind} we have $a \downfree_{\A}^{f} b_{i_{*}}b_{\J}$. By monotonicity (\emph{c.f} \cref{prop: properties of non-forking}.(3)) we conclude that $a \downfree_{\A}^{f} b_{\J}$ as required.
\end{proof}
\begin{remark}Let  $\p$ be a global type and $\bar{b}=(b_{i})_{i \in \I}$ be a sequence as in \cref{not: sequences of interest}. Note that the hypothesis of \cref{prop: domination by the canonical base} hold:
\begin{itemize}
\item By Lemma \ref{lem: independent over big tail p^{n} generically stable} there is $\J \subseteq \I$ such that $|\I \setminus \J| \leq |\T(\A)|$ and $a \downfree^{f}_{\A} b_{\J}$. 
\item Moreover, $\bar{b}_{\I_{>0}}$ is $\acl$-independent over $\A b_{0}$ as this holds for the sequence $\bar{b}_{0}$ (\emph{c.f.} \cref{rem:results}.$(3)$).  
\end{itemize}
Therefore, for $i \in \J$ we can define $\f(\bar b, a)_{i}$ and $\f_{\pi}(\bar b, a)_{i}$ as in \cref{not notation almost for all}.
\end{remark}
\subsection{The main proof for the general case}
The following is a generalization of \cref{lem: almost all are the same or different} with essentially the same proof.  
\begin{lemma}\label{lem: almost all general} Let $\p$ be a global type and $\bar{b}=(b_{i})_{i \in \I}$ a sequence as in \cref{not: sequences of interest}. Let $a,a'\models \p \rest_{\A}$. Then either:
\begin{itemize}
\item for almost all $i \in \I$,   $a, a'\models \p \rest_{\A b_{i}} \text{ and } \f(\bar b, a)_{i} = \f (\bar b, a')_{i}$ or
\item for almost all $i \in \I$,   $a,a'\models \p \rest_{\A b_{i}}, \text{ and } \f(\bar b, a)_{i} \neq \f (\bar b, a')_{i}$.
\end{itemize}
A similar statement holds for $\f_{\pi}$ where $\pi$ is some fixed projection on some finite sub-tuple. 
\end{lemma}
\begin{proof}
By \cref{lem: independent over big tail p^{n} generically stable}, there is some $\J \subseteq \I$ such that $|\I \setminus \J| \leq |\T(\A)|$ and for any $i \in \J$ both $a$ and $a'$ realize $\p\rest_{\A b_{i}}$. By compactness we can find a sequence $\bar b_*$ as in \cref{not: sequences of interest} such that $\bar{b}+\bar b_{*}$ is $\A$-indiscernible. By Fact \ref{fact: facts canonical base sequence}$(1)$ $\f(\bar b, a)_i \subseteq \dcl(\A b_i;\bar b_* \St_{\A \bar b_*}(\A \bar{b}_{*};a))$ and $\f(\bar b, a ')_i \subseteq \dcl(\A b_i;\bar b_* \St_{\A \bar b_*}(\A \bar b_{*}; a'))$. Since $\St_{\A \bar b_*}$ is stably embedded over $\A \bar b_*$, the property $\f(\bar b, a)_i =\f (\bar b, a')_{i}$ only depends on $\tp(\St_{\A \bar b_*}(\A \bar{b}_{*}; b_i)/\St_{\A \bar b_*}(\A \bar b_{*}; aa'))$ (\emph{c.f.} \cref{lem: facts stably embedded}.(1)).

The sequence $(\St_{\A \bar b_*}(\A \bar b_{*}; b_j))_{j \in \J}$ is indiscernible over $\A \bar{b}_{*}$ in the stable structure $\St_{\A \bar b_*}$, thus after removing at most $|\T(\A)|$ elements from the sequence $\bar b=(b_j)_{j \in \J}$, we may assume that it is indiscernible over $\St_{\A \bar b_*}( \A \bar b_{*}; aa')$. The result follows. The same argument applies to $\f_{\pi}$. 
\end{proof}
In particular, we can provide an analogue of  \cref{def:same germs}.
\begin{definition}\label{def: same germs gen} Let $\p$ be a global type and $\bar{b}$ a sequence as in \cref{not: sequences of interest}. Let $a,a' \models \p \rest_{\A}$.  Fix a projection $\pi$ of $\f(\bar b, a)_{i}$ to a finite sub-tuple. We say that \emph{$a$ and $a'$ have the same $\pi$-germ on $\bar b$ if $\f_{\pi}(\bar b, a)_{i}=\f_{\pi}(\bar b, a')_{i}$ for almost all $i$}.
\end{definition}
Consequently,  analogues of  \cref{cor: reducing to n} and \cref{lem: definability of almost everywhere both sequences agree} hold in this case. We include their statements and their proofs, which require minor modifications. 
\begin{corollary}\label{cor: reducing to n general} Let $\p$ be a global type and $\bar{b}$ a sequence as in \cref{not: sequences of interest}. Let $\pi$ be the projection of $\f(\bar b,a)_{i}$ on some finite sub-tuple, then there is an $n < \omega$ such that either $\f_{\pi}(\bar b, a)_{i}= \f_{\pi}(\bar b, a')_{i}$ for at most $n$ values of $i \in \I$ or $\f_{\pi}(\bar b, a)_{i}\neq \f_{\pi}(\bar b, a')_{i}$ for at most $n$ values of $i \in \I$. 
\end{corollary}
\begin{proof}
 Assume not. Since $\p$ is generically stable, it $\A$-definable. By compactness we can find $\bar b'=(b'_{j})_{j \in |\T(\A)|^{+}}$ such that $\bar b' \equiv_{\EM(\A)} \bar{b}_{0}$ such that  
 \begin{align*}
 \B_{0}&= \{ j< |\T(\A)|^{+} \ | \ a,a' \models \p \rest_{\A b_{j}} \ \text{and} \  \ \f_{\pi}(\bar b', a)_{j}= \f_{\pi}(\bar b', a')_{j}\}, \ \text{and}\\
 \B_{1}&= \{ j< |\T(\A)|^{+} \ | \ a,a' \models \p \rest_{\A b_{j}} \ \text{and} \  \ \f_{\pi}(\bar b', a)_{j}\neq \f_{\pi}(\bar b', a')_{j}\} 
 \end{align*}
 have both size $|\T(\A)|^{+}$. 
 This contradicts \cref{lem: almost all general}.  
\end{proof}
\begin{lemma}\label{lem: definability of almost everywhere both sequences agree general} Let $\p$ be a global type as in \cref{not: sequences of interest}. Let $\R'$ be the $\A$-type definable set  
\begin{equation*}
\R'=\{ (a,a', \bar b) \ | \ a,a' \models \p \rest_{\A} \  \text{and} \ \bar{b}=(b_{i})_{i \in |\T(\A)|^{+}} \ \text{is a sequence such that}\ \bar{b} \equiv_{\EM(\A)} \bar{b}_{0}\} 
\end{equation*}
Let $\pi$ be some fixed projection of $\f(\bar{b},a)_{i}$ on some finite sub-tuple. The statement \emph{$ \forall^{a}_{i} \f_{\pi}(\bar b, a)_{i}= \f_{\pi}(\bar b, a')_{i}$} is a definable condition for $(a,a',\bar b) \in \R'$.
\end{lemma}
\begin{proof}
 Since $\p$ is $\A$-definable we can apply compactness and \cref{cor: reducing to n general} to show the existence of some  $n< \omega$ such that for any $a,a'\models\p\rest_{\A}$ and every sequence $\bar{b} \equiv_{\EM(\A)} \bar{b}_{0}$ of length $|\T(\A)|^{+}$ either $ \f_{\pi}(\bar b, a)_{i}= \f_{\pi}(\bar b, a')_{i}$ holds for at most $n$-elements or  $\f_{\pi}(\bar b, a)_{i}\neq \f_{\pi}(\bar b, a')_{i}$ holds for at most $n$-elements.  Let $\F=\{ 1,\dots,2n+1\}$. Then 
 
 \begin{align*}
 \forall^{a} i \ \f_{\pi}(\bar b, a)_{i}&= \f_{\pi}(\bar b, a')_{i} \iff \bigvee_{\B \subseteq \F , |\B|= n+1} \bigwedge_{i \in \B} \f_{\pi}(\bar b, a)_{i}= \f_{\pi}(\bar b, a')_{i}. 
 \end{align*}
 As before, this shows that the given condition and its complement are type definable, the statement $\forall_{i}^{a} \f_{\pi}(\bar{b}, a)= \f_{\pi}(\bar{b},a')_{i}$ is a definable condition for $(a,a',\bar{b}) \in \R^{'}$. 
\end{proof}

Let $\Xi_{\pi}(a,a',\bar b)$ be a formula defining $\forall_{i}^{a} \f_{\pi}(\bar{b}, a)= \f_{\pi}(\bar{b},a')_{i}$. Of course only finitely many elements from the tuple $\bar b$ appear in $\Xi_{\pi}$. Note that it follows from the arguments above that $\neg \Xi_{\pi}(a,a',\bar b)$ is equivalent to $\forall_{i}^{a} \f_{\pi}(\bar{b}, a) \neq \f_{\pi}(\bar{b},a')_{i}$.

We can now give an analogue of  \cref{def: Q and relations}. 
\begin{definition}\label{def: Q and relations general} Let $\p$ be a global type as in \cref{not: sequences of interest}. Let
\[\Q' =\{ \tp(\bar b, \bar d/ \A) \ | \ \bar{b}=(b_{i})_{i \in |\T(\A)|^{+}} \ \text{such that} \ \bar b \equiv_{\EM(\A)} \bar b_{0}, \ 
\text{and there is}  \ a \vDash \p \rest_{\A} \ \text{such that} \ d_{i}= \f(\bar b, a)_{i}\}\].\\
Let $\E'_{\pi}$ be the relation on the set of realizations of  $\Q'$  defined by  \[ (\bar b, \bar d) \E'_{\pi} (\bar b', \bar d') \iff \exists a\models \p\rest_{\A} \big(\forall^a i \ \f_{\pi}(\bar b, a)_i = d_i \wedge \f_{\pi}(\bar b',a)_i = d'_i\big).\]
\end{definition}
 The proof of \cref{lem:same germs} does not go through, so we cannot assert immediately that we can replace $\exists a$ by $\forall a$ in the definition of $\E'_{\pi}$ and it is not longer clear that $\E'_{\pi}$ is an equivalence relation on the set of realizations of $\Q'$. The proof of this fact will be slightly more involved.

\begin{definition} Let $\bar{b}$, $\bar{b}'$ be sequences as in \cref{not: sequences of interest}.  Let $\pi$ be some fixed projection on a finite sub-tuple. The pair $(\bar b, \bar b')$ satisfies \emph{ $\exists=_{\pi}\forall$}  if whenever $a,a' \models \p \rest_{\A}$ have the same $\pi$ germ on $\bar b$ they have the same $\pi$ germ on $\bar b'$, i.e.
\begin{equation*}
\forall^{a}_{i} \f_{\pi}(\bar b, a)_{i}= \f_{\pi}(\bar b, a') \rightarrow \forall^{a}_{i}\f_{\pi}(\bar b', a)_{i}= \f_{\pi}(\bar b', a'). 
\end{equation*}
\end{definition}
The following is a useful observation that follows immediately by definition.
\begin{remark}\label{rem: transitivity} Let $\bar{b}, \bar{b}', \bar{b}''$ be sequences as in \cref{not: sequences of interest}. Assume $(\bar b, \bar b')$ and $(\bar b', \bar b'')$ satisfy $\exists =_{\pi}\forall$. Then $(\bar b, \bar b'')$ satisfies $\exists =_{\pi} \forall$. 
\end{remark}

\begin{remark}\label{rem: if sequence extended then same germ} Let  $\bar{b}=(b_{i})_{i \in \I}$ and $\bar{b}'=(b_{i})_{i \in \I'}$ be such that $\bar b \equiv_{\EM(\A)} \bar b_{0}$, $\bar b' \equiv_{\EM(\A)} \bar b_{0}$ and both $|\I|, |\I'| \geq |\T(\A)|^{+}$. If there is a sequence $\bar{b}''=(b_{k})_{k \in \K}$ such that $|\K| \geq |\T(\A)|^{+}$ and $\bar b'' \equiv_{\EM(\A)} \bar{b}_{0}$  such that $\bar{b}+\bar{b}''$ and $\bar{b}'+\bar{b}''$ are $\A$-indiscernible, then $(\bar b, \bar b')$ have the same $\pi$-germ for every finite projection $\pi$.
\end{remark}
\begin{proof}
  If there is a sequence $\bar b''$ such that $\bar b +\bar b'' \equiv_{\EM(\A)} \bar{b}_{0}$ and $\bar b'+\bar b'' \equiv_{\EM(\A)} \bar{b}_{0}$. Then the proof of Lemma \ref{lem:same germs} goes through. We include details for sake of completeness. \\
  By Fact \ref{fact: facts canonical base sequence} (1) $\f(\bar b, a)_{i}= \f(\bar b+\bar b'',a)_{i}$  for all $i \in \I$, and similarly for $a'$. Since $a$ and $a'$ have the same $\pi$-germ over $\bar b$ then $\f_{\pi}(\bar{b},a)_{i}=\f_{\pi}(\bar{b},a')_{i}$ for almost all $i \in \I$, and therefore
     $ \f_{\pi}(\bar b+ \bar b'', a)_i = \f_{\pi}(\bar b+ \bar b'',a')_i$ for almost all $i \in \I$. By \cref{lem: almost all general}, $\f_{\pi}(\bar b+\bar b'',a)_j = \f_{\pi}(\bar b+\bar b'',a')_j$ holds for almost all $j \in \I' \cup \K$. In particular  $|\{ k \in \K \ | \ \f_{\pi}(\bar b'',a)_k \neq \f_{\pi}(\bar b'',a')_k \}| \leq |\T(\A)|$, since  $\f_{\pi}(\bar b'',a)_k= \f_{\pi}(\bar b+\bar b'',a)_k$ for $k \in \K$ by the definition $\f$; and similarly for $a'$. Likewise  $\f_{\pi}(\bar b'+ \bar b'', a)_{k}= \f_{\pi}(\bar b'', a)_{k}$ for $k \in \K$ and for $a'$ as well.\\ 
    By \cref{lem: almost all general}, $\f_{\pi}(\bar b'+\bar b'',a)_j = \f_{\pi}(\bar b'+\bar b'',a')_j$ holds for almost all $j \in \I' \cup \K$. Hence $|\{ j \in \I' \ | \ \f_{\pi}(\bar b'+\bar b'',a)_j \neq \f_{\pi}(\bar b'+\bar b'',a')_j \}| \leq  |\T(\A)|$, and by Fact \ref{fact: facts canonical base sequence} (1) $\f(\bar b'+ \bar b'', a)_{j}= \f(\bar b', a)_{j}$ for $j \in \I'$ and correspondingly for $a'$. Consequently, $a$ and $a'$ have the same $\pi$-germ on $\bar b'$. 
\end{proof}

\begin{lemma}\label{lem: fixing equivalence relation general case} Let $\Q'$ and $\E'_{\pi}$ as in \cref{def: Q and relations general}. The relation  $\E'_{\pi}$ is an equivalence relation on the set of realizations of $\Q'$. 
\end{lemma}
\begin{proof}
It is clear that $\E_{\pi}'$ is a reflexive and symmetric relation. To show that $\E_{\pi}'$ is transitive it is sufficient to show that all pairs $(\bar b, \bar b')$ where $\bar b, \bar{b}'$ are sequences as in \cref{not: sequences of interest} satisfy $\exists=_{\pi} \forall$.\\
By \cref{rem: if sequence extended then same germ} it is inconsistent to find sequences as in \cref{not: sequences of interest} and realizations $a,a'\models \p \rest_{\A}$ such that:
\begin{enumerate}
\item There exists $\bar{b}''$ such that $\bar{b}+\bar{b}''$ and $\bar{b}'+\bar{b}''$ are $\A$-indiscernible;
\item $\Xi_{\pi}(a,a',\bar b) \wedge \neg \Xi_{\pi}(a,a',\bar b')$.
%$\f_{\pi}(\bar b,a)_i = \f_{\pi}(\bar b,a')_i$ for almost all $i$, but $\f_{\pi}(\bar b',a)_i \neq \f_{\pi}(\bar b',a')_i$ for almost all $i$.
\end{enumerate}
Those are type-definable conditions, hence by compactness there is some formula $\phi(\bar x,\bar y)$ such that for any $\bar b, \bar b'$ as in \cref{not: sequences of interest}, $\phi(\bar b, \bar b')$ holds if there is a sequence $\bar b''$ as in \cref{not: sequences of interest} such that $\bar b+ \bar b ''$ and $\bar{b}'+\bar b''$ are $\A$-indiscernible sequence and 
\begin{equation*}
\models \phi(\bar b, \bar b') \rightarrow \neg \big( \Xi_{\pi}(a,a',\bar b) \wedge \neg \Xi_{\pi}(a,a',\bar b')\big).
\end{equation*}

Consider the graph $\mathrm{G}$ whose edges are defined by $\phi(\bar x,\bar y)\vee \phi(\bar y,\bar x)$. (The reader might want to recall \cref{rem: infinite tuples} about infinite tuples of variables.) Since $\bar{b}+\bar{b''}$ is an $\A$-indiscernible sequence, then $\bar{b}$ and $\bar{b''}$ lie in the same connected component $[\bar x]$. Since $\bar b$ and $\bar {b}''$ are $\acl^{\vee}$-independent over $\A$, $[\bar x] \in \acl^{\vee}(\A \bar{b}) \cap \acl^{\vee}(\A \bar{b}'')=\acl^{\vee}(\A)$.  Since $\A=\acl(\A)$ by \cref{lem: elements in aclV are definable when base is acl closed} the connected component $[\bar b]=[\bar x] \in \dcl^{\vee}(\A)$. By \cref{lem: characterizing being in aclV} there is a formula $\psi(\bar y) \in \tp(\bar b/\A)$ and $n \in \mathbb{N}$ such that $\mathrm{diam}(\psi)\leq n$. Hence, since $\bar{b}, \bar{b}' \equiv_{\EM(\A)} \bar{b}_{0}$ then they lie in the same connected component. By \cref{rem: transitivity} the pair $(\bar b,\bar b')$ satisfies $\exists =_{\pi} \forall$, as required.
\end{proof}
By compactness we can find  an $\La(\A)$-definable set $\X$ containing $\Q'$ such that $\E_{\pi}'$ is an equivalence relation on $\X$, and denote $\Lat_{\pi}'$ the $\A$-interpretable set $\X/ \E'_{\pi}$.
The rest of the proof follows exactly as in \cref{thm: descent with invariant extension}; we include details for sake of completeness.

\bigskip
 \begin{proof}[Proof of Theorem \ref{thm:descent}] \  \\
 \underline{Claim $1$:}
  \emph{Let $\bar b=(b_{i})_{i< |\T(\A)|^{+}}$ be a sequence as in \cref{not: sequences of interest}. For each finite projection $\pi$, every class of $\Q'/\E_{\pi}^{'}$ admits a representative of the form $(\bar b,\bar d) \in \Q'$.}\\
 \textit{Proof:} Let $(\bar b',\bar d')$ be an element of $\Q'$. Take $a\models \p \rest_{\A}$ such that $\f(\bar b',a)_i = d'_i$ for almost all $i$. Then $(\bar b, \f(\bar b,a))\in \Q'$ and is $\E_{\pi}$-equivalent to $(\bar b',\bar d')$. \\
Fix $\bar b$ be a sequence as in \cref{not: sequences of interest}. The set $V'_{\bar b} = \{\bar d : (\bar b, \bar d) \in \Q'\}$ is a pro-type-definable subset of $\St_{\A \bar b}$. It admits an $\A\bar b$-pro-definable map $h_{\bar b}$ to $\Lat'_{\pi}$ whose image contains $\Q'/ \E'_{\pi}$. By compactness, the map $h_{\bar b}$ only depends on finitely many variables, hence is defined on a definable $U_{\bar b} \subseteq \St_{\A \bar b}$ containing $V_{\bar b}$. Hence the image $W_{\bar b} \subseteq \Lat'_{\pi}$ of $U_{\bar b}$ under the definable map $h_{\bar b}$ is an $\A\bar b$-definable set containing $\Q'/ \E'_{\pi}$, and it is stable and stably embedded. Let $\psi(x,\bar b)$ be a formula defining $W_{\bar b}$, which we can take so that $(\forall \bar y) \psi(x,\bar y) \to (x \in \Lat'_{\pi})$.

By the first claim, the following type-definable conditions on a tuple $(c,\bar b)$ are inconsistent:

\begin{itemize}
\item $c\in \Q'/ \E_{\pi}'$;
\item $\bar b$ is a sequence as in \cref{not: sequences of interest};
\item $\neg \psi(c, \bar b)$.
\end{itemize}

By compactness, there is some $\A$-definable condition $\eta(\bar b)$ and some $\A$-definable $\Lat'_{\pi} \subseteq \Lat_{\pi}$ containing $\Q'/ \E'_{\pi}$ such that: \[\eta(\bar b) \wedge c\in \Lat'_{\pi} \to \psi(c, \bar b).\]

Now consider the definable subset $\Lat^*_{\pi}$ of $\Lat'_{\pi}$ defined by \[\theta(x) \equiv (\forall \bar y) \eta(\bar y) \to \psi(x, \bar y).\]

The set $\Lat^*_{\pi}$ is $\A$-definable. Furthermore, for any $\bar b$ as above, it is included in the $\A\bar b$-definable set $W_{\bar b}$, which we know to be stable and stably embedded. Therefore $\Lat^*_{\pi}$ is stable and stably embedded.

Let $\Lat^* = \sqcup \Lat^*_{\pi}$ for $\pi$ ranging over all projections to a finite sub-tuple. Then $\Lat^*$ is in $\St_\A$. Let $\alpha$ be the pro-definable function that enumerates $\dcl(\A a) \cap \Lat^*$ for $a\models \p \upharpoonright_{\A}$.

\underline{Claim $2$:} \emph{The $\A$ pro-definable map $\alpha \colon \p \rightarrow \Lat^{*}$ dominates $\p$ over $\A$.}\\
\textit{Proof:} Let $a\models\p\rest_{\A}$ and assume that $\alpha(a)\downfree^{f}_{\A} e$. We aim to show that $a\downfree^{f}_{\A}e$. Let $\bar{b}=(b_{i})_{i \in |\T(\A)|^{+}}$ be a sequence as in \cref{not: sequences of interest} such that $\alpha(a) \downfree^{f}_{\A} e \bar b$. By Lemma \ref{lem: independent over big tail p^{n} generically stable} there is some $\J$ such that $a \downfree^{f}_{\A } \bar{b}_{\J}$ and $||\T(\A)|^{+} \backslash \J| \leq |\T(\A)|$. In particular, since $\alpha(a) \downfree_{\A}^{f} e \bar b$ by monotonicity and base monotonicity (\cref{prop: properties of non-forking}.(3) and (4)) we have:
\begin{equation}\label{eqlast1}
\alpha(a) \downfree^{f}_{\A \bar b_{\J}} e. 
\end{equation}
By compactness and Claim $1$ we can find $\bar{d}$ such that for each finite projection $\pi$, $\pi(\bar{b}_{\J}, \bar{d})=\alpha(a)\cap \Lat_{\pi}$. Since $\bar{d}\subseteq \St_{\A \bar{b}_{\J}}$, the type $\tp(\bar{d}/ \A \bar{b}_{\J} \alpha(a))$ is a stable type. As any set in a stable theory is an extension basis, we may further assume that 
\begin{equation}\label{eqlast2}
\bar{d} \downfree^{f}_{\A \bar{b}_{\J} \alpha(a)} e.
\end{equation}
By \cref{eqlast1}, \cref{eqlast2} and left transitivity (\emph{c.f.} \cref{prop: properties of non-forking}.(5)) $\bar{d} \alpha(a) \downfree^{f}_{\A \bar{b}_{\J}} e$ thus $\bar{d} \downfree^{f}_{\A \bar{b}_{\J}} e$. By Proposition \ref{prop: domination by the canonical base}, $a \downfree^{f}_{\A \bar{b}_{\J}} e$. Since $a \downfree_{\A} \bar{b}_{\J}$ by
right transitivity $a \downfree^{f}_{\A} \bar{b_{\J}} e$ (\emph{c.f.} \cref{prop: generically stable properties}.(5)). In particular,  $a \downfree^{f}_{\A} e$. This concludes the proof of Theorem \ref{thm: descent with invariant extension}.
\end{proof}
\sloppy
\printbibliography 
\end{document}